\documentclass[11pt,letterpaper]{amsart}
\usepackage{amssymb,mathrsfs,graphicx,enumerate,color}
\usepackage{amsthm,amsfonts,amssymb,epsfig,graphics,amsmath,amsbsy,enumerate}
\usepackage{colortbl}
\definecolor{black}{rgb}{0.0, 0.0, 0.0}

\title[Stability of composite waves for the system of ideal gas without viscosity]
{Time-asymptotic stability of generic Riemann solutions for the system of heat-conductive ideal gas without viscosity}

\author[Y. Guo]{Yilin Guo}
\address[Yilin Guo]{\newline School of Mathematical Sciences, Laboratory of Mathematics and Complex Systems, MOE, Beijing Normal University, Beijing 100875, P. R. China.}
\email{202531130032@mail.bnu.edu.cn}

\author[L.-A. Li]{Lin-An Li}
\address[Lin-An Li]{\newline School of Mathematical Sciences, Laboratory of Mathematics and Complex Systems, MOE, Beijing Normal University, Beijing 100875, P. R. China.}
\email{linanli@amss.ac.cn}

\author[J. Wu]{Jiahong Wu}
\address[Jiahong Wu]{\newline Department of Mathematics, University of Notre Dame, Notre Dame, IN 46556, USA.}
\email{jwu29@nd.edu}

\author[X. Xu]{Xiaojing Xu}
\address[Xiaojing Xu]{\newline School of Mathematical Sciences, Laboratory of Mathematics and Complex Systems, MOE, Beijing Normal University, Beijing 100875, P. R. China.}
\email{xjxu@bnu.edu.cn}

\allowdisplaybreaks

\newtheorem{theorem}{Theorem}[section]
\newtheorem{lemma}{Lemma}[section]

\newtheorem{proposition}{Proposition}[section]
\newtheorem{remark}{Remark}[section]

\newcommand{\beq}{\begin{equation}}
	\newcommand{\eeq}{\end{equation}}
\newcommand{\bsp}{\begin{split}}
	\newcommand{\esp}{\end{split}}

\newcommand{\di}{\displaystyle}
\newcommand{\s}{\sigma}
\newcommand{\x}{\xi}

\newcommand{\deltas}{\delta_S}
\newcommand{\deltar}{\delta_R}
\newcommand{\dc}{\delta_C}
\newcommand{\ds}{\delta_S}

\newcommand{\bbr}{\mathbb {R}}
\newcommand{\R}{\mathbb {R}}
\newcommand{\mb}{\mathbf}

\def\eps{\varepsilon }

\def\lam{\lambda}  
\def\l{\lambda}

\newcommand \vc{v^{C}}
\DeclareMathSizes{11}{11}{3}{2}

\begin{document}

\date{\today}

\subjclass[2020]{}


\begin{abstract}
	This paper is concerned with the time-asymptotic stability of the generic Riemann solution for the one-dimensional system of heat-conductive ideal gas without viscosity, where the generic Riemann solution consists of a shock, a contact discontinuity, and a rarefaction wave. We prove that, as time tends to infinity, the solution of the non-viscous and heat-conductive ideal gas system converges uniformly to a composite wave composed of rarefaction wave, viscous contact wave, and viscous shock wave with a time-dependent shift. Motivated by the recent work of Kang-Vasseur-Wang [Arch. Ration. Mech. Anal. 249: 42 (2025)], we overcome the difficulties arising from the concurrence of shock and rarefaction waves for the partially dissipative hyperbolic-parabolic system with dissipation acting only on a single variable. More notably, the absence of velocity dissipation gives rise to new and intrinsic difficulties when handling the terms associated with the density and velocity. To resolve this, we exploit the precise structure of the governing equations and the additional properties of shock waves. Furthermore, we utilize the wave structure of the system without viscosity and perform separate space-time estimates for the density and velocity.

\end{abstract}

\maketitle

\tableofcontents

\section{Introduction}
\setcounter{equation}{0}

Consider the equations for the heat-conductive ideal gas without viscosity, which are described by the following system in Lagrangian mass coordinates:
\begin{equation}\label{NS}
\begin{cases}
    v_t-u_x=0, 
    & x\in\mathbb{R},\ t>0,\\[1mm]
    u_t+p(v,\theta)_x=0,
    & \\[1mm]
    E_t+(p(v,\theta)u)_x
    =\left(\kappa\frac{\theta_x}{v}\right)_x,
    &
\end{cases}
\end{equation}
where $v>0$, $u$, $\theta>0$, $E=e+\frac{u^2}2$ and $p$ represent the specific volume, velocity, the absolute temperature, the total energy and the pressure of the gas, respectively, and $\kappa$ is a positive constant representing the heat-conductivity coefficient. Since we consider an ideal polytropic gas, the pressure $p$ and the internal energy  $e$ are given by 
\begin{equation}\label{state}
	p=\frac{R\theta}v=A v^{-\gamma} \exp \left(\frac{\gamma-1}{R}s\right),\qquad e=\frac{R}{\gamma-1}\theta+const.,
\end{equation}
where $A>0, R>0, \gamma>1$ are constants and $s=\frac{R}{\gamma-1} \log(\frac{R}{A} \theta v^{\gamma-1})$ is entropy. We study the Cauchy problem for \eqref{NS} with initial conditions:
\begin{equation}\label{in}
	(v, u, \theta)(0, x)=(v_{0}, u_{0}, \theta _{0})(x) \to (v_{\pm },u_{\pm },\theta _{\pm }),  \quad \text{as} \quad x\to \pm \infty,
\end{equation}
where the states $v_{ \pm}>0$, $\theta_{ \pm}>0$ and $u_{ \pm} \in \mathbb{R}$ at far fields are given constants. We aim to prove that the long-time behavior of solutions to \eqref{NS}, with initial values verifying \eqref{in}, is determined by the Riemann problem of the associated full Euler equations: 
\begin{equation}\label{E}
\begin{cases}
    v_t-u_x=0, 
    & x\in\mathbb{R},\ t>0,\\[1mm]
    u_t+p(v,\theta)_x=0,
    & \\[1mm]
    E_t+(p(v,\theta)u)_x
    =0
    &
\end{cases}
\end{equation}
with the Riemann initial data
\begin{equation}\label{Ei}
(v,u,\theta)(0,x)=
\begin{cases}
(v_-,u_-,\theta_-), & x<0,\\[1mm]
(v_+,u_+,\theta_+), & x>0,
\end{cases}
\end{equation}
corresponding to the end states \eqref{in}. In this paper, we consider a generic three-wave configuration in which the Riemann solution to \eqref{E}–\eqref{Ei} consists of a shock wave, a contact discontinuity, and a rarefaction wave.

When the ideal gas has both viscosity and heat-conductivity with  initial values verifying \eqref{in}, there have been extensive research results. The corresponding governing equations are classical compressible Navier-Stokes-Fourier equations as follows:
\begin{equation}
	\begin{cases}\label{CNS}
		v_t - u_x = 0,   \qquad\qquad x \in \mathbb{R},\ t > 0, \\
		u_t + p_x = \mu\left(\frac{u_x}{v}\right)_x, \\
		\left(e+\frac{u^2}{2}\right)_t + (pu)_x = \kappa\left(\frac{\theta_x}{v}\right)_x + \mu\left(\frac{uu_x}{v}\right)_x,
	\end{cases}
\end{equation}
where $\mu$	is a positive constant representing the viscosity coefficient. 

For compressible Navier-Stokes equations, including isentropic system and non-isentropic system \eqref{CNS}, nonlinear stability  has been established for viscous shock waves, rarefaction waves, viscous contact discontinuities, and various classes of composite waves. For viscous shock waves, Goodman \cite{Good} and Matsumura-Nishihara \cite{MN85} studied the stability problems under the zero-mass assumption. Subsequently, Liu \cite{Liu85}, Szepessy-Xin \cite{SX93} and Liu-Zeng \cite{LZ15} further removed the restriction of the zero-mass assumption by introducing a constant shift for the viscous shock and diffusion waves, as well as for the coupled diffusion waves in the transverse characteristic fields. Nevertheless, all the aforementioned investigations employed the anti-derivative method to perform stability analysis and yielded corresponding results. For rarefaction waves, Matsumura-Nishihara \cite{MN92,MN86}, Liu-Xin \cite{LX88}, Jiu-Wang-Xin \cite{JWX13} and Nishihara-Yang-Zhao \cite{NYZ04} successfully established the stability via the direct energy method. And Huang-Xin-Yang \cite{HXY}, Huang-Matsumura-Xin \cite{HMX} and Liu-Xin \cite{Xin-c} have analyzed the stability of viscous contact discontinuities. In the context of composite wave research, Huang-Matsumura \cite{HM} have conducted stability analysis on the superposition of two shock waves, as well as Huang-Li-Matsumura \cite{HLM} and Huang-Wang \cite{HW16} have studied the stability of the composition of rarefaction waves and viscous contact discontinuities. However, the stability of shock superimposed with rarefaction waves, or more generally, composite waves of shocks, rarefaction waves, and contact discontinuities remained an open problem for a long time. This was primarily because the classical anti-derivative method used for analyzing shock stability was not well compatible with the direct energy method used for rarefaction and contact discontinuity waves. To circumvent the anti-derivative framework, Kang-Vasseur \cite{KV17} first introduced the theory of $a$-contraction with shifts in the context of artificial viscosity, and subsequently Kang-Vasseur \cite{KVJEMS21} applied it to the barotropic Navier-Stokes equations. Recently, in references \cite{KVW23, KVW25}, Kang-Vasseur-Wang applied the $a$-contraction with shifts method, successfully solving the stability problem for composite waves consisting of a shock wave, a rarefaction wave, and a contact discontinuity.

Compared to system \eqref{CNS}, although system \eqref{NS} lacks velocity viscosity, it still satisfies the ``Kawashima-Shizuta condition'' \cite{SK1, SK2}. This condition ensures the recovery of dissipation for all variables (density, velocity, and temperature), from which it can be inferred that system \eqref{NS} may still possess stability. Based on this observation, Murakami investigated the stability of the rarefaction wave for system \eqref{NS} in \cite{TM}. Hou established the stability of a rarefaction wave on the half-line in \cite{Hou}. Ma-Wang studied the stability of contact discontinuities in \cite{MW}. Fan-Matsumura \cite{FM} conducted stability analysis on the composite of two shock waves through the anti-derivative method, and Fan-Gong-Tang \cite{FGT} obtained stability results for the composite of rarefaction waves and contact discontinuities through the direct energy integration method. Similar to the system \eqref{CNS}, due to the emergence of both shock and rarefaction waves, the stability problem of the composite wave formed by the superposition of rarefaction waves, viscous shock waves and viscous contact discontinuities remains unsolved to date.

In this paper, for the partially dissipative hyperbolic-parabolic system \eqref{NS} with dissipation acting only on a single variable, we rigorously prove that the Kawashima-Shizuta transfer mechanism is sufficient to guarantee the time-asymptotic stability of the aforementioned generic Riemann solution.

\subsection{Riemann problem for the inviscid system}

First, we recall the Riemann problem for the Euler equations \eqref{E}-\eqref{Ei}, which was proposed and solved by Riemann \cite{Riemann} in the 1860s, then generalized by Lax \cite{Lax}. Since the Euler system \eqref{E} is strictly hyperbolic, set $V=(v,u,E)^t$ and the system \eqref{E} can be written as $$V_t+\mathbf{A}(V)V_x=0,$$
where the Jacobian matrix $\mathbf{A}(V)$ is given by
\[
\mathbf{A}(V)=
\begin{pmatrix}
0 & -1 & 0 \\[1mm]
-\dfrac{p}{v} 
& -\dfrac{(\gamma-1)u}{v} 
& \dfrac{\gamma-1}{v} \\[2mm]
-\dfrac{pu}{v}
& p-\dfrac{(\gamma-1)u^2}{v}
& \dfrac{(\gamma-1)u}{v}
\end{pmatrix}
\]
with three real eigenvalues:
\[
\lambda_1 = \lambda_1(v,\theta) = -\sqrt{\frac{\gamma p}{v}} < 0, \quad 
\lambda_2 = 0, \quad \text{and} \quad 
\lambda_3 = \lambda_3(v,\theta) = \sqrt{\frac{\gamma p}{v}} > 0.
\]
According to the references \cite{Smoller, Dafermos1}, these eigenvalues generate two genuinely nonlinear characteristic fields corresponding to $\lambda_1$ and $\lambda_3$, and one linearly degenerate characteristic field corresponding to $\lambda_2$. Consequently, the solution to the Riemann problem is determined by a combination of at most three elementary waves, one from each characteristic family: a 1-rarefaction or 1-shock, a 2-contact discontinuity, and a 3-rarefaction or 3-shock. Without loss of generality, in this paper, we mainly consider a composite wave composed of 1-rarefaction wave, 2-contact discontinuity wave, and 3-shock wave, which is a solution to Riemann problem of Euler equations \eqref{E} with the initial values \eqref{Ei}. For convenience, we first use
\begin{equation}\label{ELR}
(v,u,\theta)(0,x)=
\begin{cases}
(v_L,u_L,\theta_L), & x<0,\\[1mm]
(v_R,u_R,\theta_R), & x>0,
\end{cases}
\end{equation}
as the Riemann data to construct the elementary wave of system \eqref{E}, and then introduce the generic Riemann solution consisting of 1-rarefaction wave, 2-contact discontinuity wave, and 3-shock wave.

For the 1-rarefaction wave, given any state $(v_{L}, u_{L}, \theta_{L}) \in \mathbb{R}_+ \times \mathbb{R} \times \mathbb{R}_+$, there exists 1-rarefaction wave integral curve determined by the eigenvalue $\lambda_1$:
\begin{equation*}
	R_1 (v_L, u_L, \theta_L) := \left\{ (v, u, \theta) \left| v > v_L ,~ s(v,\theta) = s(v_L, \theta_L) := s_L,~ u = u_L - \int_{v_L}^v \lambda_1(s_L, v^\prime) \, dv^\prime \right. \right\},
\end{equation*}
where the entropy function $s$ and $\lambda_1$ are defined as
\begin{align}\label{entropy}
	\begin{aligned}
	s(v,\theta) = \frac{R}{\gamma-1} \log\left( \frac{R}{A}\theta v^{\gamma-1}\right), \qquad \lambda_1(s,v) = -\sqrt{\frac{\gamma p(v,s)}{v}}.
		\end{aligned}
\end{align}
Therefore, for any initial values \eqref{ELR} with $(v_{R}, u_{R}, \theta_{R}) \in R_1 (v_L, u_L, \theta_L)$, the solution $(v^{r}, u^{r}, \theta^{r})$ of the Riemann problem \eqref{E}-\eqref{ELR} is 1-rarefaction wave, which is defined as
\begin{equation} \label{rarefaction}
\lambda_1(v^r(t, x),\theta^r(t,x))  = \begin{cases}
\lambda_1(v_L,\theta_L) , \qquad x < \lambda_1(v_L,\theta_L)t, \\[2mm]
\di \frac{x}{t}, \qquad  \lambda_1(v_L,\theta_L)t \leq x \leq  \lambda_1(v_R,\theta_R)t, \\[2mm]
\lambda_1(v_R,\theta_R), \qquad x > \lambda_1(v_R,\theta_R)t,
\end{cases}
\end{equation}
together with 
\begin{equation*}
\begin{array}{ll}
  &z_1(v^r(t,x), u^r(t,x))=  z_1 (v_L, u_L)=z_1 (v_R,u_R),\\[3mm]
  &s(v^r(t,x), \theta^r(t,x))=s(v_L, \theta_L)=s(v_R,\theta_R),
\end{array}
\end{equation*}
where $\di z_1(v,u)=u+\int^v \lambda_1(s,v^\prime)dv^\prime$ and $s(v,\theta)$ are the 1-Riemann invariants to the Euler equations \eqref{E}.

For the 2-contact discontinuity, given any state $(v_{L}, u_{L}, \theta_{L})$,  there exists the 2-contact discontinuity curve $CD_2(v_{L}, u_{L}, \theta_{L})$ corresponding to the eigenvalue $\l_2\equiv 0$, which can be defined by
\begin{equation*}
 CD_2 (v_L, u_L, \theta_L):=\Big{ \{} (v, u, \theta)\Big{ |}u=u_L ,~p(v,\theta)=p(v_L, \theta_L)\Big{ \}}.
\end{equation*}
Then, for any initial values \eqref{ELR} satisfying $(v_{R}, u_{R}, \theta_{R})\in CD_2(v_{L}, u_{L}, \theta_{L})$, the 2-contact discontinuity $(v^c,u^c,\theta^c)$ is uniquely determined by 
\begin{equation*} 
	(v^c,u^c,\theta^c)(t,x)= \begin{cases}
		(v_L,u_L,\theta_L) , \qquad x < 0, \\
		(v_R,u_R,\theta_R), \qquad x > 0,
	\end{cases}
\end{equation*}
as the solution of the Riemann problem \eqref{E}-\eqref{ELR}.

For the 3-shock wave, given any state $(v_{L}, u_{L}, \theta_{L})$, the 3-shock curve $S_3$ associated with the eigenvalue $\lambda_3>0$ can be derived from the Rankine-Hugoniot conditions, as follows:
\begin{equation*}
\begin{aligned}
&-\sigma(v_L-v)-(u_L-u)=0,\\
& -\sigma(u_L-u)+(p(v_L,\theta_L)-p(v,\theta))=0,\\
& -\sigma(E_L-E)+(p(v_L,\theta_L)u_L-p(v,\theta)u)=0,
\end{aligned}
\end{equation*}
where $E_L=e(v_L,\theta_L)+\frac{u_L^2}{2}$, and $\sigma=\sqrt{-\frac{p(v,\theta)-p(v_L,\theta_L)}{v -v_L}}>0$ for $v_L<v$ satisfies Lax entropy condition.

Thus, for any initial values \eqref{ELR} satisfying $(v_R, u_R,\theta_{R}) \in S_3 (v_L, u_L, \theta_L)$, the solution $(v^{s}, u^{s}, \theta^{s})$ of the Riemann problem \eqref{E}-\eqref{ELR} is 3-shock wave:  
\begin{equation*} 
	(v^s, u^s, \theta^s)(t, x) = \begin{cases}
		(v_L, u_L, \theta_L), & x < \sigma t, \\
		(v_{R}, u_{R}, \theta_{R}), & x > \sigma t.
	\end{cases}
\end{equation*}

For the generic Riemann solutions case, given two states $(v_-, u_-, \theta_-),~(v_+, u_+, \theta_+) \in \bbr_+\times\bbr\times\bbr_+$, possibly nearby, it can be shown that there exist two (unique) intermediate states $(v_*, u_*, \theta_*),~(v^*, u^*, \theta^*) \in \bbr_+\times\bbr\times\bbr_+$ such that $(v_*,u_*,\theta_*)\in R_1(v_-,u_-,\theta_-)$, $(v^*,u^*,\theta^*)\in CD_2(v_*,u_*,\theta_*)$ and $(v_+,u_+,\theta_+)\in S_3(v^*,u^*,\theta^*)$.
Then we can obtain the solution $(v,u,\theta)$  to Euler system \eqref{E}-\eqref{Ei}, which consists of three associated waves
\begin{equation}\label{ICW}
	(v,u, \theta)(t,x)=(v^r,u^r, \theta^r)(t,x)+(v^c,u^c, \theta^c)(t,x)+(v^s,u^s, \theta^s)(t,x)-(v_*, u_*, \theta_*)-(v^*,u^*,\theta^*).
\end{equation}
 Similarly, we can also consider the composite wave composed of 1-shock wave, 2-contact discontinuity wave, and 3-rarefaction wave.

Subsequently, we focus primarily on the viscous waves associated with the Riemann solution \eqref{ICW} for system \eqref{E}, specifically the viscous contact discontinuity and viscous shock wave. For the viscous contact discontinuity corresponding to the inviscid 2-contact discontinuity wave connecting the states $(v_{*}, u_{*}, \theta_{*})$ with $(v^{*}, u^{*}, \theta^{*})$, it can be defined by \cite{HXY} as follows:
\begin{equation}\label{vc}
	\begin{aligned}
		&  v^{C}(\frac{x}{\sqrt{1+t}})
		=  \frac{R\Theta}{p_*},\\
		&u^{C}(t,\frac{x}{\sqrt{1+t}})
		= u_* +\frac{(\gamma-1)\kappa\Theta_x}{R\gamma\Theta},\\
		& \theta^{C}(\frac{x}{\sqrt{1+t}})
		= \Theta,
	\end{aligned}
\end{equation}
where $p_*=p(v_*, \theta_*)$ and $\Theta=\Theta\left(\frac{x}{\sqrt{1+t}}\right)$ is the unique self-similar solution to the following nonlinear diffusion equation
\begin{equation}
	\begin{cases}
		\Theta_t= \frac{(\gamma-1)\kappa p_*}{R^2 \gamma}\left(\frac{\Theta_x}{\Theta}\right)_x, \cr
		\Theta(t, -\infty)=\theta_{*},~\Theta(t, +\infty)=\theta^{*}.
	\end{cases}
\end{equation}
For the 3-viscous shock wave also called the traveling wave $(v^S, u^S, \theta^S)(\xi)(\xi=x-\sigma t)$ which connects the states $(v^*, u^*, \theta^*)$ with $(v_+,u_+,\theta_+)$, it can be defined by \cite{FM} from the following ODEs:
\begin{equation}\label{VS}
\begin{cases}
-\sigma (v^S)'- (u^S)'=0,\\[3mm]
-\sigma (u^S)'+(p^S)'=0,\\[3mm]
-\sigma (E^S)' +(p^S u^S)'
=\left(\kappa\frac{(\theta^S)'}{v^S}\right)',\\[3mm]
(v^S,u^S,\theta^S)(-\infty)=(v^*,u^*,\theta^*),\qquad
(v^S,u^S,\theta^S)(+\infty)=(v_+,u_+,\theta_+),
\end{cases}
\end{equation}
where $E^S=\frac{R}{\gamma-1}\theta^S+\frac{(u^S)^2}{2}$, $p^S=p(v^S, \theta^S)$, the shock speed $\sigma=\sqrt{-\frac{p(v_+,\theta_+)-p(v^*,\theta^*)}{v_+-v^*}}$ satisfies the Lax entropy condition
\begin{equation}\label{ec}
	\lambda_3(v_+,\theta_+)<\sigma<\lambda_3(v^*,\theta^*).
\end{equation}
The properties of viscous contact discontinuity and viscous shock wave will be introduced in the next section. 

As viscous ansatz, the composite wave, defined as 
\beq\label{SW}
\begin{array}{ll}
	\di (\tilde v, \tilde u,\tilde \theta) (t,x):=
    \Big( v^r(\frac xt)+ v^C(\frac{x}{\sqrt{1+t}})+ v^S(x-\s t -\mb X(t))-v_*-v^*,\\[4mm]
	\di \qquad\qquad\qquad\qquad u^r(\frac xt)+ u^C(t,\frac{x}{\sqrt{1+t}})+ u^S(x-\s t-\mb X(t))-u_*-u^*, \\[4mm]
	\di \qquad\qquad\qquad\qquad  \theta^r(\frac xt)+ \theta^C(\frac{x}{\sqrt{1+t}})+ \theta^S(x-\s t-\mb X(t))-\theta_*-\theta^*\Big),
\end{array}
\eeq
is composed of the 1-rarefaction wave $(v^r,u^r,\theta^r)(\frac xt)$ defined in \eqref{rarefaction} with end states $(v_-,u_-,\theta_-)$ and $(v_*,u_*,\theta_*)$, the 2-viscous contact wave $(v^C, \theta^C)(\frac{x}{\sqrt{1+t}})$ and $u^C(t, \frac{x}{\sqrt{1+t}})$ defined in \eqref{vc} with end states $(v_*,u_*,\theta_*)$ and $(v^*,u^*,\theta^*)$, and  the 3-viscous shock wave shifted by $\mb X(t)$ (to be defined in \eqref{X(t)}) $(v^S,u^S,\theta^S)(x-\sigma t -\mb X(t))$ of  \eqref{VS} with end states $(v^*,u^*,\theta^*)$ and $(v_+,u_+,\theta_+)$.
\subsection{The main results}
\begin{theorem}\label{thm:main}
    Given any state $(v_{-}, u_{-}, \theta_{-})\in\bbr_+\times\bbr\times\bbr_+$, consider the intermediate states $(v_*,u_*,\theta_*)$, $(v^*,u^*,\theta^*)$ and the right state $(v_{+},u_{+},\theta_{+})$ satisfying 
$(v_*,u_*,\theta_*)\in R_1(v_-,u_-,\theta_-)$, 
$(v^*,u^*,\theta^*)\in CD_2(v_*,u_*,\theta_*)$, and 
$(v_+,u_+,\theta_+)\in S_3(v^*,u^*,\theta^*)$, 
along with the associated composite wave $(\tilde v, \tilde u,\tilde \theta) (t,x)$ defined in \eqref{SW}. Then, there exist positive constants $\delta_0$ and $\eps_0$ such that if the wave strength satisfies
\[
    |v_+-v^*|+|v^*-v_*|+|v_*-v_-|\leq \delta_0,
\]
and the initial perturbation satisfies
\begin{equation}\label{i-p}
    \sum_{\pm} \|(v_0-v_\pm, u_0-u_\pm,\theta_0-\theta_\pm)\|_{L^2(\bbr_\pm)} + \|(v_{0x},u_{0x},\theta_{0x})\|_{H^1(\bbr)} < \eps_0,
\end{equation}
where $\bbr_- := (-\infty,0)$, then the Cauchy problem \eqref{NS}-\eqref{state} admits a unique global-in-time solution $(v,u,\theta)(t,x)$. Furthermore, there exists an absolutely continuous time-dependent shift $\mb{X}(t)$ (defined in \eqref{X(t)}) such that
	\begin{align}
		\begin{aligned}\label{ext-main}
			&(v- \tilde v, u-\tilde{u}, \theta-\tilde \theta)(t,x) \in C(0,+\infty;H^1(\bbr)),\\
			& (v_{xx}, u_{xx}, \theta_{xx})(t,x) \in C(0,+\infty;L^2(\bbr)),\\
			& \theta_{xxx}(t,x)-\theta^S_{xxx}(x-\sigma t-\mb X(t))\in L^2(0,+\infty; L^2(\bbr)),
		\end{aligned}
	\end{align}
    and for any $t>0$,
    \begin{align}
		\begin{aligned}\label{ext-main2}
			\|(v- \tilde v, u-\tilde{u}, \theta-\tilde \theta)(t,\cdot)\|_{L^2(\bbr)} + \|(v, u, \theta)_{x}(t,\cdot)\|_{H^1(\bbr)}\leq C(\eps_0+\delta_0^{\frac{1}{4}}).
		\end{aligned}
	\end{align}
	In addition, the asymptotic behaviors are as follows, as $t\to+\infty$,
	\begin{equation}\label{con}
		\begin{array}{l}
			\di \sup_{x\in\bbr}\Big| (v- \tilde v, u-\tilde{u}, \theta-\tilde \theta)(t,x) \Big| \to 0,
		\end{array}
	\end{equation}
	and
	\begin{equation}\label{as}
		\lim_{t\rightarrow+\infty} |\dot {\mb X}(t) |=0,
	\end{equation}
    which means that $\mathbf{X}(t)$ grows at most sub-linearly in time and the shifted viscous shock wave retains the original traveling wave profile time-asymptotically.
\end{theorem}

\begin{remark}
 Theorem \ref{thm:main} states that whenever the far-field states $(v_\pm, u_\pm, \theta_\pm)$ in \eqref{in} are connected by the superposition of a 1-rarefaction wave, a 2-contact discontinuity, and a 3-shock wave, the solution of the non-viscous and heat-conductive system \eqref{NS} or \eqref{NS-0} converges in the long-time limit to the composite wave consisting of an inviscid self-similar rarefaction wave, a viscous contact wave, and a viscous shock wave with shift $\mathbf{X}(t)$. The same conclusion holds when the far-field states $(v_\pm, u_\pm, \theta_\pm)$ in \eqref{in} are connected by a 1-shock wave, a 2-contact discontinuity, and a 3-rarefaction wave.
\end{remark}


\subsection{Main difficulties and key ideas of the proof}

First, the coexistence of shocks, rarefaction waves and contact discontinuities in the composite wave creates an incompatibility between the classical anti-derivative framework for shock stability and the direct energy method for rarefaction waves. To overcome this difficulty, we employ the $a$-contraction with shifts method \cite{KVW23, KVW25}, which introduces a time-dependent shift to establish shock stability without resorting to the anti-derivative approach.

Although we adopt this method, the absence of velocity viscosity in system \eqref{NS} creates an essential difficulty: direct dissipative control of the velocity derivative $(u-\bar u)_\xi$ is no longer available. In \cite{KVW25}, the velocity dissipation $\mu\left\|(u-\bar u)_\xi\right\|_{L^2(\bbr)}^{2}$ plays two essential roles in the proof: (i) in the relative entropy estimate, the viscous momentum equation with $\mu>0$ is used to create a new estimate that can produce the good term needed to apply the Poincar\'e-type inequality \eqref{poincare}, thereby controlling all the nonlinear terms; and (ii) in the higher-order derivative estimates, it directly provides higher-order dissipation for the velocity, which is needed to close the higher-order energy estimates. Neither of these direct controls is available for system \eqref{NS}.

To address (i), we first exploit the structure of the equations and
the properties of the shock wave to establish the new estimate
\eqref{sharp-D}, which is based on the energy equation rather than
the viscous momentum equation used in \cite{KVW25}. More precisely,
\eqref{sharp-D} converts the available temperature dissipation into
a good term. Then, we transform terms involving velocity and density into temperature-related terms, and subsequently control the main bad terms using the Poincar\'{e}-type inequality \eqref{poincare}.

For (ii), due to the absence of velocity viscosity, the momentum equation provides no direct parabolic smoothing, making it non-trivial to close the energy estimates. Since the system \eqref{NS} structurally satisfies the ``Kawashima-Shizuta condition'', one can expect to recover the dissipation for both the density and velocity. Indeed, by linearizing the pressure and cross-differentiating the mass and momentum equations, we formally obtain the following coupled wave-type system for the perturbations \(\phi\) and \(\psi\) defined in \eqref{per}:
\begin{equation}\label{eq:rough_wave_intro}
\begin{aligned}
    (\partial_t-\sigma\partial_\xi)^2
    \begin{pmatrix}
        \phi\\
        \psi
    \end{pmatrix}
    -
    c^2\partial_{\xi\xi}
    \begin{pmatrix}
        \phi\\
        \psi
    \end{pmatrix}   =
    \begin{pmatrix}
        -(p_\theta\vartheta)_{\xi\xi}\\
        -(\partial_t-\sigma\partial_\xi)
        (p_\theta\vartheta)_\xi
    \end{pmatrix}
    +\mathcal{N},
\end{aligned}
\end{equation}
where $c = \sqrt{-p_v} > 0$, and $\mathcal{N}$ denotes the nonlinear remainders. Guided by the above wave structure \eqref{eq:rough_wave_intro}, we introduce new energy functionals to carry out separate space-time estimates for the density and the velocity. This procedure successfully transfers the strict dissipation from the temperature to the hyperbolic variables, fully compensating for the loss of regularity and dissipation resulting from the absence of velocity viscosity, and thereby closing the complete \textit{a priori} estimates.

The remainder of this paper is structured as follows. In Section 2, we introduce the properties of rarefaction waves, viscous shock waves, and viscous contact discontinuities, with a primary focus on establishing novel shock wave properties based on the structure of the equations. In Section 3, we sketch how both local existence and the main Theorem \ref{thm:main} follow from the {\it a priori} estimates derived as Proposition \ref{prop2}. Sections 4 and 5 provide detailed proofs of Proposition \ref{prop2}. In Section 4, we first obtain the $L^2$ estimates via the relative entropy and $a$-contraction with shifts method. Section 5 then derives the higher-order energy estimates along with the space-time estimates of the density and velocity, which in turn lead to the full energy estimates.

\noindent$\bullet$ {\bf Notations:} We use the following notations  for notational simplicity. \\
1. $C$ denotes a generic positive constant of order $O(1)$ that may differ from line to line, but is independent of the small parameters $\delta_0, \varepsilon_1, \delta_S, \delta_R, \delta_C$, $\lambda$ (see \eqref{weight}) and the time $T$.
\\
2. For any function $f : \bbr_+\times \bbr\to \bbr$ and any time-dependent shift $\mb X(t)$, 
\[
f^{\pm \mb X}(t, \xi):=f(t,\xi\pm \mb X(t)).
\]
To simplify the notation, we drop the explicit arguments of the waves whenever the context is clear: for instance,
\begin{align*}
	\begin{aligned}
		&u^R:=u^R(t,\xi+\s t),\quad  u^{\mb X}:=u(t,\xi+\s t + \mb X(t)),\\
		&u^C:=u^C(t,\xi+\s t),\quad  (u^C)^{\mb X}:=u^C(t,\xi+\s t + \mb X(t)),\\
		&\bar u^{\mb X}:=u^R(t,\xi+\s t + \mb X(t))+u^C(\frac{\xi+\s t+ \mb X(t)}{\sqrt{1+t}})+u^S(\xi)-u_*-u^*.
	\end{aligned}
\end{align*}
3. For functional spaces, $H^s(\mathbb{R})$ denotes the $s$-th order Sobolev space with the norm  
\[  
\|f\|_{H^s(\mathbb{R})} := \sum_{j=0}^{s} \|\partial_x^j f\|_{L^2(\mathbb{R})}.  
\]  

\section{Preliminaries} \label{sec-preli}
\setcounter{equation}{0}

Before presenting the proof, we collect in this section some well-known results that will be used throughout the rest of the paper. For simplicity, we consider the following system in non-divergence form which is equivalent to the original system \eqref{NS}:
\begin{equation}\label{NS-0}
\begin{cases}
v_t-u_x=0,\\
u_t+p(v,\theta)_x=0,\\
\frac{R}{\gamma-1}\theta_t+p(v,\theta)u_x
=\left(\kappa\frac{\theta_x}{v}\right)_x.
\end{cases}
\end{equation}

\subsection{Approximate rarefaction wave}

We now recall important properties of the approximate 1-rarefaction wave $(v^{R}, u^{R}, \theta^{R})$ that approximate the 1-rarefaction wave  $(v^{r}, u^{r}, \theta^{r})$. Motivated by \cite{MN86}, we construct the approximate 1-rarefaction wave $(v^{R}, u^{R}, \theta^{R})$ as a smooth solution of the Euler system \eqref{E}, which can be defined by
\begin{align}\label{AR}
	\begin{aligned} 
		&\lambda_{1-}:=\lambda_1(v_-,\theta_-) = w_-,\ \lambda_{1*}:=\lambda_1(v_*, \theta_*) = w_*,\\
		&\lambda_1(v^R,\theta^R)(t, x) = w(1+t, x), \\
		&u^R = u_- - \int_{v_-}^{v^R} \lambda_1(s_-, v^\prime) \, dv^\prime,\\
		&s(v^R, \theta^R)(t, x)=s(v_-, \theta_-)=s(v_*,\theta_*),
	\end{aligned}
\end{align}
where the entropy function $s$ and $\lambda_1$ are defined as \eqref{entropy}, and $w(t, x)$ is the solution of the initial problem for the typical Burgers equation
\begin{equation} \label{ABE}
	\begin{cases}
		\displaystyle w_t + ww_x = 0, \\
		\displaystyle w(0, x) = w_0(x) = \frac{w_* + w_-}{2} + \frac{w_* - w_-}{2} \tanh x,
	\end{cases}
\end{equation}
with $w_- = \lambda_1(v_-,\theta_-), w_* = \lambda_1(v_*,\theta_*)$.

And it is easy to check that the above approximate rarefaction wave $(v^R, u^R, \theta^R)$ satisfies the Euler system:
\begin{equation}  \label{ARW}
	\begin{cases}
		\displaystyle v^R_t -u^R_x = 0, \\
		\displaystyle u^R_t + p(v^R, \theta^R)_x = 0,\\
		\di \frac{R}{\gamma-1}\theta^R_t+p(v^R, \theta^R) u^R_x=0. \\
	\end{cases}
\end{equation}

Set $\delta_R:=|v_*-v_-|$ denote the strength of the rarefaction wave and we have $\delta_R\sim |u_*-u_-|\sim |\theta_*-\theta_-|$. The following lemma comes from \cite{MN86}.
\begin{lemma} \label{lemma1.2}
	The smooth approximate 1-rarefaction wave $(v^R, u^R, \theta^R)(t, x)$ defined in \eqref{AR} satisfies the following properties:
	
	(1)~$ u^R_x = \frac{2 v^R}{(\gamma + 1) }w_x > 0,$ $ v^R_x = \frac{v^R}{\sqrt{R\gamma\theta^R}} u^R_x>0$ and $\theta^R_x=-\frac{(\gamma-1)\theta^R}{v^R}v^R_x<0,$  $\forall x \in \bbr,\ t \geq 0$.
	
	(2)~The following estimates hold for all $t \geq 0$ and $p \in [1, + \infty]$:
	\begin{align*}
		&\|(v^R_x, u^R_x, \theta^R_x)\|_{L^p} \leq C \min\{\delta_R, \delta_R^{1/p}(1+t)^{-1+1/p}\}, \\
		&\|(v^R_{xx}, u^R_{xx}, \theta^R_{xx})\|_{L^p} \leq C \min\{\delta_R, (1+t)^{-1}\}, \\
        &\|(v^R_{xxx}, u^R_{xxx}, \theta^R_{xxx})\|_{L^2} \leq C \min\{\delta_R, (1+t)^{-1}\}, \\
		& |u^R_{xx}| \leq C |u^R_{x}|,\quad |v^R_{xx}| \leq C |v^R_{x}|,\quad  |\theta^R_{xx}| \leq C |\theta^R_{x}| ,\\
		& |u^R_{xxx}| \leq C |u^R_{x}|,\quad |v^R_{xxx}| \leq C |v^R_{x}|,\quad  |\theta^R_{xxx}| \leq C |\theta^R_{x}| ,\quad \forall x\in\bbr.
	\end{align*}
	
	(3)~ For $x\geq \lambda_{1*}(1+t), t\geq 0,$ it holds that
	\begin{align*}
		&|(v^R, u^R, \theta^R)(t, x)-(v_*, u_*, \theta_*)| \leq C\delta_R \ e^{-2|x-\lambda_{1*}(1+t)|}, \\
		&|(v^R_x, u^R_x, \theta^R_x)(t, x)|\leq C\delta_R\ e^{-2|x-\lambda_{1*}(1+t)|}.
	\end{align*}
	
	(4)~ For $x\le\lambda_{1-}(1+t), t\geq 0,$ it holds that
	\begin{align*}
		&|(v^R, u^R, \theta^R)(t, x)-(v_-,u_-, \theta_-)| \leq C\delta_R \ e^{-2|x-\lambda_{1-}(1+t)|}, \\
		&|(v^R_x, u^R_x, \theta^R_x)(t, x)|\leq C\delta_R\ e^{-2|x-\lambda_{1-}(1+t)|}.
	\end{align*}

    (5)~The following estimates hold for all $t > 0$ and $p \in (1, + \infty)$:
    \begin{align*}
        &\|(v^R(t, \cdot) - v^r(\frac{\cdot}{t} ), u^R(t, \cdot) - u^r(\frac{\cdot}{t} ),\theta^R(t, \cdot) - \theta^r(\frac{\cdot}{t} ))\|_{L^p(\mathbb{R})} \le C \delta_R^{\frac{1}{p}} (1+t)^{-1+\frac{1}{p}},\\
        &\|(v^R(t,\cdot)-v^r(\frac{\cdot}{t}), u^R(t,\cdot)-u^r(\frac{\cdot}{t}), \theta^R(t,\cdot)-\theta^r(\frac{\cdot}{t}))\|_{L^\infty(\mathbb R)}\le C\delta_R \frac{1+\log(1+t\delta_R)}{1+t\delta_R}.
    \end{align*}
\end{lemma}

\subsection{Viscous shock wave}

For the viscous shock wave, we integrate system \eqref{VS} over 
 $(\pm \infty, \xi]$. Using the total energy $E^S=\frac{R}{\gamma-1}\theta^S+\frac{(u^S)^2}{2}$, we obtain:
\begin{equation}\label{VS2}
\begin{cases}
(p^S-p_+)+\sigma^2(v^S-v_+)
=(p^S-p^*)+\sigma^2(v^S-v^*)=0,\\
\begin{aligned}
-\kappa\frac{(\theta^S)'}{\sigma v^S}
&=\frac{R}{\gamma-1}(\theta^S-\theta_+)
+p_+(v^S-v_+)-\frac12\sigma^2(v^S-v_+)^2\\
&=\frac{R}{\gamma-1}(\theta^S-\theta^*)
+p^*(v^S-v^*)-\frac12\sigma^2(v^S-v^*)^2,
\end{aligned}
\end{cases}
\end{equation}
which means the existence of a shock profile  $(v^S, u^S, \theta^S)(\xi)$ is equivalent to the existence of solution to ODEs \eqref{VS2}. The properties of the 3-viscous shock wave $(v^S, u^S, \theta^S)$ are summarized below. And we mainly establish the explicit estimates \eqref{shock-vu}-\eqref{sharp-D}, while the remaining estimates can be found in \cite{FM}.

\begin{lemma} \label{lemma1.3}
	For any state  $(v^*,u^*,\theta^*)$, there exists a constant $C>0$ such that the following holds. For any end state such that $(v_+, u_+,\theta_+)\in S_3(v^*,u^*,\theta^*)$, there exists a unique (up to a constant shift) solution  $(v^S, u^S, \theta^S)(\xi)$ to \eqref{VS}. Let $\delta_S$ denote the strength of the shock as $\delta_S:=|v_+-v^*|\sim|u_+-u^*|\sim |\theta_+-\theta^*|$. Then, the following estimates hold 
	\begin{align}
		\begin{aligned}\label{shock-base}
			& u^S_\xi<0, \qquad v^S_\xi >0, \qquad \theta^S_\xi<0,\quad  \forall\xi\in\bbr,\\
			& | (v^S(\xi) -v^*, u^S(\xi) -u^*, \theta^S(\xi) -\theta^*)|\leq C\deltas\ e^{-C\delta_S |\xi|}, \quad \xi<0,\\[1mm]
			& | (v^S(\xi) -v_+, u^S(\xi) -u_+, \theta^S(\xi) -\theta_+)|\leq C\deltas\ e^{-C\delta_S |\xi|}, \quad \xi>0,\\[1mm]
			&|( v^S_\xi, u^S_\xi,\theta^S_\x)|\leq C\deltas^2\ e^{-C\delta_S |\xi|}, \quad\forall\xi\in\bbr,\\[1mm]
			&|( v^S_{\xi\xi}, u^S_{\xi\xi},\theta^S_{\x\x})|\leq C\delta_S |( v^S_\xi, u^S_\xi, \theta_\x^S)|, \quad \forall\xi\in\bbr,\\[1mm]
			&|( v^S_{\xi\xi\xi}, u^S_{\xi\xi\xi},\theta^S_{\x\x\x})|\leq C\delta_S |( v^S_\xi, u^S_\xi, \theta_\x^S)|, \quad \forall\xi\in\bbr.
		\end{aligned}
	\end{align}		
	In particular,  $|v^S_\xi|\sim  |u^S_\xi | \sim |\theta^S_\xi |$ for all $\xi\in\bbr$, more explicitly, we have
	\beq\label{shock-vu}
	\Big| (u^S)_\xi + \sigma^* (v^S)_\xi \Big| \le C \deltas|(v^S)_\xi|, \quad\forall \xi\in\bbr,
	\eeq
    and
	\beq\label{theta-s}
	\Big| (\theta^S)_\xi + \frac{(\gamma-1)p^*}{R}(v^S)_\xi \Big| \le C \deltas|(v^S)_\xi|,\quad\forall \xi\in\bbr.
	\eeq
    By employing the structure of the equations and the temperature dissipation, we have the following estimate
	\beq\label{sharp-D}
	\left| \frac{\kappa}{(v^S-v^*)(v_+-v^S)} \frac{v^S_{\xi}}{v^S}  - \alpha^*\frac{ R\gamma}{ (\gamma-1)^2} \right|\leq C\deltas,
	\eeq
	where 
	\beq\label{pstar}
	\begin{array}{ll}
	p^*:=p(v^*, \theta^*)=\frac{R\theta^*}{v^*},\quad \sigma^*:=\sqrt\frac{\gamma p^*}{v^*} = \frac{\sqrt{\gamma R\theta^*}}{v^*},\\
	p_+:=p(v_+, \theta_+)=\frac{R\theta_+}{v_+},\quad
	\sigma:=\sqrt{-\frac{p(v_+,\theta_+)-p(v^*,\theta^*)}{v_+-v^*}},
	\end{array}
	\eeq
	and 
	$$\alpha^*:= \frac{\gamma(\gamma+1) p^*}{2(v^*)^2\s^*},$$
	which satisfies
	\beq\label{sm1}
	|\sigma-\sigma^*|\le C\delta_S.
	\eeq
\end{lemma}

\begin{proof}
	We only prove the estimates \eqref{shock-vu}-\eqref{sharp-D} here, and the remaining estimates come from \cite{FM}. The estimate \eqref{shock-vu} is easily shown by \eqref{sm1} and the equation 
	$( u^S)_\xi=-\sigma (v^S)_\xi$.
	
	To prove \eqref{theta-s}, we use the following fraction from  \eqref{VS2} that
	\begin{equation*}
		\theta^S_\xi=\frac{p^S-\s ^2 v^S}{R}v^S_\xi,
	\end{equation*}
	so we only need to prove 
	\begin{equation}\label{ps}
		|p^S-\s ^2 v^S+(\gamma-1)p^*| \leq C \delta_S.
	\end{equation}
	Using \eqref{sm1} and $p^*-(\s^*)^2v^*=(1-\gamma)p^*$, we can obtain \eqref{theta-s}. 
	
	Next, we prove the estimate \eqref{sharp-D}. To do so, we first establish a useful equality
	\begin{equation}\label{beteq}
		\frac{1}{2}\sigma^2 v^*+\frac{\gamma}{\gamma-1} p^*=\frac{\gamma+1}{2(\gamma-1)}\s^2 v_+ .
	\end{equation}
	From $\eqref{VS2}$, let $\xi \to +\infty$ and $\xi \to -\infty$, it holds that
	\begin{align*}
			&\frac{R}{\gamma-1}(\theta_+-\theta^*)=p_+(v^* -v_+)-\frac12\sigma^2(v^*-v_+)^2,\\
			&\frac{R}{\gamma-1}(\theta_+-\theta^*)=-p^*(v_+ -v^*)+\frac12\sigma^2(v_+-v^*)^2,
	\end{align*}
	which implies 
	\begin{equation*}
		\frac{2R}{\gamma-1}(\theta_+-\theta^*)-(p_++p^*)(v^* -v_+)=0.
	\end{equation*}
	Using the fact that $p_+v_+= R\theta_+, ~p^*v^*= R\theta^*$ and $\s^2=\frac{p^*-p_+}{v_+-v^*}$, we get
	\begin{align*}
			0=&\frac{2R}{\gamma-1}(\theta_+-\theta^*)-(p_++p^*)(v^* -v_+)\\
			=&\frac{1}{\gamma-1}\left[(\gamma+1)p^*(v_+-v^*)+(\gamma-1)p_+(v_+-v^*)-2(p^*-p_+)v_+\right]\\
			=&\frac{1}{\gamma-1}(v_+-v^*)\left[(\gamma+1)p^*+(\gamma-1)p_+-2\s^2 v_+\right]\\
			=&\frac{1}{\gamma-1}(v_+-v^*)\left[2\gamma p^*+(\gamma-1)(p_+-p^*)-2\s^2 v_+\right]\\
			=&\frac{1}{\gamma-1}(v_+-v^*)\left[2\gamma p^*+(\gamma-1)\s^2(v^*-v_+)-2\s^2 v_+\right]\\
			=&2(v_+-v^*)\left[\frac{\gamma}{\gamma-1} p^*+\frac{1}{2}\s^2 v^*-\frac{\gamma+1}{2(\gamma-1)}\s^2 v_+\right].
	\end{align*}
	Thus, we get the equality \eqref{beteq}. From $\eqref{VS2}$, we have
$$
	0=-\s (v^S -v^*) - \frac{p^S-p^*}{\s} =-\frac{1}{\s}\left[ \s^2(v^S -v^*) + p^S-p^*\right],
$$
and then using the fact \eqref{pstar} and $p^S=\frac{R\theta^S}{v^S},$
it holds that 
\begin{equation}\label{v-theta}
	\theta^S-\theta^*=\frac{v^S}{R}\left(\frac{R\theta^*}{v^Sv^*}-\s^2\right) (v^S -v^*) .
\end{equation}
Along with $\eqref{VS2}$, $\eqref{beteq}$ and $\eqref{v-theta}$, we get
\begin{align*}
		-\frac{\kappa}{\sigma} \frac{\theta^S_\xi }{v^S} 
		&=\frac{R}{\gamma-1}(\theta^S-\theta^*)+p^*(v^S-v^*)-\frac{1}{2}\sigma^2(v^S-v^*)^2\\
		&=(v^S-v^*)\left(\frac{1}{2}\sigma^2 v^*+\frac{\gamma}{\gamma-1} p^*-\frac{\gamma+1}{2(\gamma-1)}\s^2 v^S\right)\\
		&=\frac{\gamma+1}{2(\gamma-1)}\s^2 (v^S-v^*)(v_+-v^S),
\end{align*}
which implies
\[
\frac{\kappa}{(v^S-v^*)(v_+-v^S)} \frac{\theta^S_\xi}{v^S} =-\frac{\gamma+1}{2(\gamma-1)}\sigma^3.
\]
Differentiating both sides with respect to $\x$ in $\eqref{v-theta}$, we obtain $$\theta^S_{\x}=\frac{1}{R}(p^S-\s^2 v^S)v^S_{\x}.$$
Consequently,
\[
\frac{\kappa}{(v^S-v^*)(v_+-v^S)} \frac{v^S_\xi}{v^S} =-\frac{R}{p^S-\s^2 v^S}\frac{\gamma+1}{2(\gamma-1)}\sigma^3.
\]
By means of equations \eqref{sm1} and \eqref{ps}, together with some direct calculations, we have
$$\left|\frac{\gamma+1}{2(\gamma-1)}\frac{R\s^3}{p^S-\s^2v^S}+\alpha^*\frac{R\gamma}{(\gamma-1)^2}\right|\leq C \delta_S,$$
so that we complete the proof.
\end{proof}

\begin{remark}
 In the proof of \eqref{sharp-D}, we exploit the structure of the equations \eqref{VS2} to convert the temperature-dependent viscosity into a density-dependent form, which facilitates its use in the subsequent relative entropy method to handle the difficulties arising from the associated nonlinear terms.
\end{remark}

\subsection{Viscous contact wave}

It is established that the inviscid contact discontinuity exhibits time-asymptotic instability for the compressible Euler equations \eqref{E}. In contrast, the viscous contact wave as the viscous version of the inviscid contact discontinuity, has been successfully constructed and shown to be time-asymptotically stable within both the ``artificial viscosity'' system \cite{Xin-c, Liu-Xin97} and the physical Navier-Stokes system \cite{HXY, HLM}. We recall the following lemma from \cite{HXY}.

\begin{lemma}\label{Lemma 2.2.}
The viscous contact wave  $\big(\vc, u^{C}, \theta^{C}\big)(t,x)$ defined in \eqref{vc} satisfies
	\beq\label{vc-pe}
	\begin{array}{ll}
		\di\big(v^C-v_*, u^C-u_*,\theta^{C}-\theta_* \big)= O(1)\delta_{_C}e^{ -\frac{\tilde C x^2}{1+t}}, \qquad \forall x<0;\\
		\di\big(v^C-v^*, u^C-u^*,\theta^{C}-\theta^* \big)= O(1)\delta_{_C}e^{ -\frac{\tilde C x^2}{1+t}}, \qquad \forall x>0;\\
		\di (\partial_x^nv^{C},\partial_x^n\theta^{C})(t,x)
		=O(1)\delta_{_C}(1+t)^{-\frac{n}{2}}e^{ -\frac{\tilde C x^2}{1+t}}, \qquad\forall x\in \mathbb{R}, \quad n=1,2,\cdots;\\
		\di \partial_x^n u^{C}(t, x)
		=O(1)\delta_{_C}(1+t)^{-\frac{1+n}{2}}e^{ -\frac{\tilde C x^2}{1+t}},\qquad\qquad\ \   \forall x\in \mathbb{R},\quad n=1,2,\cdots;\\
	\end{array}
	\eeq
	where $\delta_{_C}:=|v^*-v_*|\sim|\theta^{*}-\theta_{*}|$ is the amplitude of the viscous contact wave, and $\tilde C>0 $ is generic constant. 
\end{lemma}
Then the equations of the viscous contact wave  $\big(\vc, u^{C},
\theta^{C}\big)(t,x)$ defined in \eqref{vc} can be written as follows:
\begin{equation}\label{vc-equ}
	\begin{cases}
		v^{C}_t-u^{C}_{x} = 0, \cr
		u^{C}_t+ p^{C}_x
		=Q^C_1, \cr
		\frac{R}{\gamma-1}\theta^{C}_t+  p^{C} u^{C}_x
		=\kappa\big(\frac{\theta^{C}_{x}}{v^{C}} \big)_x,
	\end{cases}
\end{equation}
where $ p^{C}=p\big( v^{C}, \theta^{C}\big)$ and
\beq\label{QC1}
\begin{aligned}
	Q^C_1=u^C_t=O(1)\delta_{_C}(1+t)^{-\frac32}e^{-\frac{\tilde C x^2}{1+t}},
\end{aligned}
\eeq
as $x\rightarrow\pm \infty$ due to Lemma \ref{Lemma 2.2.}. Moreover, from \eqref{vc} and Lemma \ref{Lemma 2.2.}, it holds that for $\forall p\geq1,$
\beq
\|\big(\vc, u^{C},
\theta^{C}\big)(t,\cdot)-\big(v^c, u^{c},
\theta^{c}\big)(t,\cdot)\|_{L^p(\bbr)}=O(1)\kappa^{\frac{1}{2p}}(1+t)^{\frac{1}{2p}},
\eeq
which implies that viscous contact wave $(\vc, u^{C},
\theta^{C})(t,x)$ can converge to the inviscid contact discontinuity $(v^c, u^{c},
\theta^{c})(t,x)$ in $L^p$-norm $(\forall p\geq1)$ on any finite time interval as the heat conductivity coefficient $\kappa\rightarrow0+$, however, they could be far away at large time.

\section{Proof of main result}\label{sec-thm}
\setcounter{equation}{0}

\subsection{Reformulation of problem and local existence} 
 For simplicity, to align with the shock wave, we first carry out a coordinate transformation $(t,x)\mapsto (t, \xi=x-\s t)$. Thus, we can rewrite the system \eqref{NS-0} into the following system: 
\begin{equation}\label{NS-1}
\begin{cases}
v_t-\sigma v_\x-u_\x=0,\\
u_t-\sigma u_\x+p(v,\theta)_\x=0,\\
\frac{R}{\gamma-1}\theta_t-\frac{R\sigma}{\gamma-1}\theta_\x
+p(v,\theta)u_\x
=\left(\kappa\frac{\theta_\x}{v}\right)_\x .
\end{cases}
\end{equation}
Similarly, it follows from \eqref{ARW} that the approximate rarefaction wave $(v^R, u^R, \theta^R)(t,\xi+\sigma t) $ satisfies 
\begin{equation}  \label{rarexi}
	\begin{cases}
		\displaystyle v^R_t -\s  v^R_\xi-u^R_\xi = 0, \\
		\displaystyle u^R_t -\s  u^R_\xi+ p(v^R, \theta^R)_\xi = 0,\\
		\di \frac{R}{\gamma-1}\theta^R_t-\frac{R\s}{\gamma-1}\theta^R_\x+p(v^R, \theta^R) u^R_\x=0, \\
	\end{cases}
\end{equation}
and from \eqref{vc-equ} that
the viscous contact wave  $\big(\vc, u^{C},
\theta^{C}\big)(t,\xi+\sigma t)$ satisfies 
\begin{equation}\label{vcex}
	\begin{cases}
		v^C_t-\s v^C_\x-u^C_{\x} = 0, \cr
		u^C_t-\s u^C_\x+ p^C_\x
		=Q^C_1, \cr
		\frac{R}{\gamma-1}\theta^C_t-\frac{R\s}{\gamma-1}\theta^C_\x+  p^C u^C_\x
		=\kappa\big(\frac{\theta^C_{\x}}{v^C} \big)_\x,
	\end{cases}
\end{equation}
with the error terms $Q^C_1$ defined in \eqref{QC1}.
We will consider stability of the solution to \eqref{NS-1} around the superposition wave of the approximate 1-rarefaction wave, the 2-viscous contact wave and  the 3-viscous shock wave shifted by $\mb X(t)$ (to be defined in \eqref{X(t)}) : 
\beq\label{shwave}
\begin{array}{l}
	\di (\bar v, \bar u,\bar \theta) (t,\xi):= \Big(v^R(t,\xi+\sigma t)+v^C(\xi+\sigma t)+ v^S(\xi -\mb X(t))-v_*-v^*,\\[4mm]
	\di \qquad\qquad\qquad\qquad u^R(t,\xi+\sigma t)+u^C(t,\xi+\sigma t)+u^S(\xi-\mb X(t))-u_*-u^*,\\[4mm]
	\di \qquad\qquad\qquad\qquad  \theta^R(t,\x+\s t)+ \theta^C(\xi+\sigma t)+ \theta^S(\xi-\mb X(t))-\theta_*-\theta^* \Big).
\end{array}
\eeq
Then we can get the equations of the superposition wave $(\bar v, \bar u,\bar \theta) (t,\xi)$:
\begin{equation}\label{bar-system}
\begin{cases}
\bar v_t-\sigma \bar v_\xi+\dot{\mb X}(t)(v^S)^{-\mb X}_\xi-\bar u_\xi=0,\\[3mm]
\bar u_t-\sigma \bar u_\xi+\dot{\mb X}(t)(u^S)^{-\mb X}_\xi+\bar p_\xi=Q_1,\\[3mm]
\frac{R}{\gamma-1}\bar\theta_t-\frac{R\sigma}{\gamma-1}\bar\theta_\xi
+\frac{R}{\gamma-1}\dot{\mb X}(t)(\theta^S)^{-\mb X}_\xi
+\bar p\bar u_\xi
=\kappa\left(\frac{\bar\theta_\xi}{\bar v}\right)_\xi+Q_2,
\end{cases}
\end{equation}
where $\bar p=\frac{R\bar\theta}{\bar v}$ and the error terms
\beq\label{Q12}
\begin{aligned}
	&Q_1:=Q^I_1+Q^C_1, \\
	&Q_2:=Q^I_2+Q^R_2,
\end{aligned}
\eeq
with the wave interactions terms
\beq\label{QI1}
Q_1^I:= \left(\bar p-p^R-p^C-(p^S)^{-\mb X}\right)_\xi,
\eeq
\beq\label{QI2}
\begin{array}{ll}
	\di 
	Q_2^I:= \left(\bar p\bar u_\x-p^Ru^R_\xi-p^Cu^C_\xi-(p^S)^{-\mb X}(u^S)^{-\mb X}_\x\right)-\kappa\left(\frac{\bar \theta_\x}{\bar v}-\frac{\theta^R_\x}{v^R}-\frac{\theta^C_\x}{v^C}-\frac{(\theta^S)^{-\mb X}_\x}{(v^S)^{-\mb X}}\right)_\x,
\end{array}
\eeq
the error terms due to the inviscid rarefaction wave
\beq\label{QR}
Q^R_2:=-\kappa\left(\frac{\theta^R_\x}{v^R}\right)_\x ,
\eeq
and the error terms $Q_1^C$ due to the viscous contact wave are given in \eqref{QC1}.

For any given initial perturbation in $H^2$ of the composite wave profile \eqref{shwave}, we establish the existence of a global strong solution to the system \eqref{NS-1}. We will employ the standard continuation argument to derive uniform-in-time estimates for the perturbation. As a preliminary step, we recall the local-in-time well-posedness theory for strong solutions to \eqref{NS} (and also for \eqref{NS-1}); see \cite{Nash} for the general framework. More precisely, the local existence of strong solutions to \eqref{NS}, and hence to \eqref{NS-1}, is ensured as follows.
\begin{proposition} \label{prop:soln}
	Let $\underline v$, $\underline u$ and $\underline{\theta}$ be smooth functions such that
	\beq\label{sm}
	(\underline v(x) , \underline u(x), \underline \theta(x)) =(v_\pm, u_\pm, \theta_\pm), \quad\mbox{for } \pm x \ge 1.
	\eeq
	For any constants $M_0, M_1,  \underline \kappa_0,  \overline \kappa_0, \underline\kappa_1, \overline\kappa_1$ with $M_1>M_0>0$ and $ \overline \kappa_1>\overline \kappa_0>  \underline \kappa_0>\underline\kappa_1>0$, there exists a constant $T_0>0$ such that if 
	\begin{align*}
		\begin{aligned}
			&\|(v_0-\underline v, u_0 -\underline u, \theta_0-\underline \theta)\|_{H^2(\bbr)}  \le M_0,\\
			&0< \underline \kappa_0 \leq v_0(x), \theta_0(x)\leq  \overline \kappa_0, \qquad  \forall x\in \bbr, \\
		\end{aligned}
	\end{align*}
	then \eqref{NS-1} has a unique solution $(v, u, \theta)$ on $[0,T_0]$ such that 
	\begin{align*}
		\begin{aligned}
			&v -\underline v, u -\underline u\in C([0,T_0];H^2(\bbr)), \\
			&\theta-\underline{\theta}\in C([0,T_0];H^2(\bbr)) \cap  L^2(0,T_0;H^3(\bbr)),
		\end{aligned}
	\end{align*}
	and
	\[
	\|(v-\underline v, u-\underline u, \theta-\underline\theta)\|_{L^\infty(0,T_0;H^2(\bbr))} \le M_1.
	\]
	Moreover:
	\beq\label{bddab}
	\underline  \kappa_1 \le v(t,x),~\theta(t,x) \le \overline  \kappa_1,\qquad \forall (t,x)\in [0,T_0]\times \bbr.
	\eeq
\end{proposition}

\subsection{Construction of shift}

The continuation argument is fundamentally based on the {\it a priori} estimates established in Proposition \ref{prop2}. As these estimates explicitly depend on the shift function, we now introduce its definition. The shift is constructed by means of a weight function $a:\mathbb{R}\to\mathbb{R}$ prescribed in \eqref{weight}. For the purposes of the subsequent analysis, it suffices to note that $a$ satisfies the uniform bound $\|a\|_{C^1(\mathbb{R})}\leq 2.$ Having fixed this weight function, we define the shift $\mathbf{X}$ as the solution to the following ODE:
\begin{equation}\label{X(t)}
\begin{cases}
\begin{aligned}
\dot{\mb{X}}(t)
={}&-\frac{M}{\delta_S}\Big[\int_{\mathbb{R}}a(\xi-\mb{X})
\Big[u^S_\xi(\xi-\mb{X})(u-\bar u)
+\frac{R\theta^S_\xi(\xi-\mb{X})}{(\gamma-1)\bar\theta}(\theta-\bar\theta)\\[4mm]
&\qquad\qquad\qquad\qquad\qquad
+\frac{\bar p v^S_\x(\xi-\mb{X})}{\bar v}(v-\bar v)\Big]d\xi,
\end{aligned}\\[4mm]
\mb X(0)=0,
\end{cases}
\end{equation}
where $M$ is the specific constant chosen as $
M:=\frac{3\gamma(\gamma+1)p^*}{2(v^*)^2(\sigma^*)^3}$, which will be used in the proof of Lemma \ref{lem-sharp1}.

The following lemma ensures that \eqref{X(t)} has a unique absolutely continuous solution defined on any interval in time $[0,T]$ for which   \eqref{bddab} is verified.

\begin{lemma} \cite{CKKV20} \label{lem_ckkv}
	Let $p>1$ and $T>0$. Suppose that a function 
	$F:[0,T]\times\bbr\rightarrow\bbr$  satisfies 
	$$\sup_{x\in\bbr }|F(t,x)|\leq f(t)\ \ \mbox{and}\ \ 
	\sup_{x,y\in\bbr,x\neq y }\Big|\frac{F(t,x)-F(t,y)}{x-y}\Big|\leq g(t),
	\quad \mbox{for } t\in[0,T] $$ for some functions $f \in L^1(0,T)$ and $\, g\in L^p(0,T)$. Then for any $x_0\in\bbr$, there exists a unique absolutely continuous function $\mb{X}:[0,T]\rightarrow \bbr$ satisfying
	\begin{equation}\label{ode_eq}
\begin{cases}
\begin{aligned}
\dot{\mb{X}}(t)&=F(t,\mb{X}(t)),\quad\mbox{for \textit{a.e.} }t\in[0,T],\\
\mb{X}(0)&=x_0.
\end{aligned}
\end{cases}
\end{equation}
\end{lemma}

\begin{lemma}\label{lem:xex}
	For any $c_1,c_2>0$, there exists a constant $C>0$ such that the following is true.  For any $T>0$, and any   function $v,\theta\in L^\infty ((0,T)\times \R)$ 
	verifying
	\beq\label{odes}
	c_1 \le v(t,x), \theta(t,x)\le c_2,\qquad \forall (t,x)\in [0,T]\times \bbr,
	\eeq
	the ODE \eqref{X(t)} has a unique absolutely continuous solution $\mb X$ on $[0,T]$. Moreover, 
	\beq\label{roughx}
	|{\mb X}(t)| \le Ct,\quad \forall t\le T.
	\eeq
\end{lemma}
\begin{proof}
	We will use the above lemma as a simple adaptation of the well-known Cauchy-Lipschitz theorem.  

	To apply Lemma \ref{lem_ckkv}, let $F(t,\mb{X})$ denote the right-hand side of the ODE \eqref{X(t)}. \\
	Then the sufficient conditions of the above lemma are verified thanks to the facts that $\|a\|_{C^1(\bbr)}\leq 2$, $\|(v^S,u^S,\theta^S)\|_{C^2(\bbr)}\leq C$, and $\|( v^S_\xi, u^S_\x,\theta^S_\x)\|_{L^1} \le C\delta_S$.
	Indeed, using \eqref{odes}, we find that for some constant $C>0$,
	\beq\label{f1t}
	\begin{array}{l}
		\di \sup_{\mb{X}\in\bbr}|F(t,\mb{X})| \le \frac{C}{\deltas}  \|\big(v-\bar v, u-\bar u, \theta-\bar\theta) \|_{L^\infty(\bbr)} \int_\bbr |( v^S_\xi, u^S_\x,\theta^S_\x)| d\xi\le C,
	\end{array}
	\eeq
	and
	\begin{align*}
		\begin{aligned}
			\sup_{\mb{X}\in \bbr} |\partial_{\mb{X}}F(t,\mb{X})| &\le  \frac{C}{\deltas} \|a\|_{C^1}\|\big(v-\bar v, u-\bar u, \theta-\bar\theta) \|_{L^\infty(\bbr)} \int_\bbr |( v^S_\xi, u^S_\x,\theta^S_\x)|d\xi  \le C.
		\end{aligned}
	\end{align*}
	Especially, since $|\dot{\mb X}(t)| \le C$ by \eqref{f1t}, we have \eqref{roughx}.
\end{proof}

\subsection{Global existence and asymptotic stability}

To prove our main result Theorem \ref{thm:main}, it is sufficient to prove the following global existence and uniform estimates.
\begin{proposition}[global existence and uniform estimates]\label{prop3}
	Let $(v_-,u_-,\theta_-)\in\mathbb{R}_+\times\mathbb{R}\times\mathbb{R}_+$ be a given state. Denote by $(\bar v,\bar u,\bar\theta)$ the composite wave profile constructed in \eqref{shwave}, in which the shift $\mathbf{X}$ is only performed in the viscous shock and is determined as the absolutely continuous solution to \eqref{X(t)} with respect to the weight function $a$ specified in \eqref{weight}.
    Then, there exist small, positive constants $\delta_0$ and $\eps_0$ such that if the wave strength satisfies
\[
    |v_+-v^*|+|v^*-v_*|+|v_*-v_-|\leq \delta_0,
\]
and the initial perturbation satisfies \eqref{i-p}, the Cauchy problem \eqref{NS}-\eqref{state} admits a unique global-in-time solution $(v,u,\theta)(t,x)$. Furthermore, there exists an absolutely continuous time-dependent shift $\mb{X}(t)$ (defined in \eqref{X(t)}) such that
	\begin{align}
		\begin{aligned}\label{infinest}
			&\sup_{t\geq0}\Big[\|(v-\bar v, u-\bar u, \theta-\bar \theta)(t,\cdot)\|^2_{H^2(\bbr)}\Big] +\deltas\int_0^{\infty}|\dot{\mb{X}}(\tau)|^2 d\tau+ \int_0^{\infty} ( \mathcal{G}^R(U) + \mathcal{G}^S(U))d\tau\\
			&\quad+\int_0^{\infty}\|(v-\bar v, u-\bar u)_x(\tau,\cdot)\|_{H^1(\bbr)}^2 d\tau+\int_0^{\infty}\|(\theta-\bar \theta)_x(\tau,\cdot)\|^2_{H^2(\bbr)} d\tau\\
			&\le C_0^2 \|\big(v_0(\cdot)-\bar v(0,\cdot),u_0(\cdot) -\bar u(0,\cdot),\theta_0(\cdot)-\bar\theta(0,\cdot)\big)\|^2_{H^2(\bbr)}  + C_0^2 \delta_0^{1/2} ,
		\end{aligned}
	\end{align}
	where 
	\begin{align}
		\begin{aligned}\label{maingood}
			&\mathcal{G}^R(U):= \int_\bbr|v^R_\x| |(v-\bar v, \theta-\bar \theta)|^2 d\x ,\\
			&\mathcal{G}^S(U):=\int_\bbr|(v^S)^{-\mb X}_\x| |(v-\bar v, u-\bar u, \theta-\bar \theta)|^2 d\x.
		\end{aligned}
	\end{align}
\end{proposition}

Once we prove Proposition \ref{prop3}, we can use it to prove \eqref{con} for the long-time behavior as follows. In fact, from the estimates \eqref{infinest}, it holds that
 \begin{equation*}
     \int_0^{\infty} \left(\|(v-\bar v, u-\bar u, \theta-\bar \theta)_x\|^2_{L^2(\bbr)} + \left| \frac{d}{d\tau} \|(v-\bar v, u-\bar u, \theta-\bar \theta)_x\|^2_{L^2(\bbr)}\right|\right) d\tau <\infty,
 \end{equation*}
 which implies 
 \begin{equation}\label{longtime}
     \lim_{t\to\infty}\|(v-\bar v, u-\bar u, \theta-\bar \theta)_x\|^2_{L^2(\bbr)}=0.
 \end{equation}
 By Sobolev's inequality, we have
  \begin{equation*}
     \|(v-\bar v, u-\bar u, \theta-\bar \theta)\|^2_{L^{\infty}(\bbr)}\leq C \|(v-\bar v, u-\bar u, \theta-\bar \theta)\|_{L^{2}(\bbr)}\|(v-\bar v, u-\bar u, \theta-\bar \theta)_x\|_{L^{2}(\bbr)},
 \end{equation*}
 which, together with \eqref{longtime}, implies
  \begin{equation*}
     \lim_{t\to\infty}\|(v-\bar v, u-\bar u, \theta-\bar \theta)\|_{L^\infty(\bbr)}=0.
 \end{equation*}
 Then, using Lemma \ref{lemma1.2} (5), \eqref{xprop} and above equality, we can prove \eqref{con}.  and \eqref{as}. As for \eqref{ext-main2}, using Lemma \ref{lemma1.2} (5) and Proposition \ref{prop3}, we can prove the estimate \eqref{ext-main2}. Therefore, we complete the proof of Theorem \ref{thm:main}.

\subsection{A priori estimates}
To prove Proposition \ref{prop3}, it suffices to show the following {\it a priori} estimates.

\begin{proposition}[{\it a priori} estimates] \label{prop2}	
	Let $(v_-,u_-,\theta_-)\in\mathbb{R}_+\times\mathbb{R}\times\mathbb{R}_+$ be a given state. Then there exist positive constants $C_0$, $\delta_1$, and $\varepsilon_1$ such that the following statement holds.	
	Assume that $(v,u,\theta)$ is a solution to the system \eqref{NS-1} on $[0,T]$ for some $T>0$. Denote by $(\bar v,\bar u,\bar\theta)$ the composite wave profile constructed in \eqref{shwave}, in which the shift $\mathbf{X}$ is only performed in the viscous shock and is determined as the absolutely continuous solution to \eqref{X(t)} with respect to the weight function $a$ specified in \eqref{weight}. Suppose that the composite wave strength satisfy $\delta_0=\deltar+\delta_C+\deltas \leq \delta_1$ 
	and that 
	\begin{align*}
		\begin{aligned}
			&(v -\bar v, u-\bar u)\in C([0,T];H^2(\bbr)), \\
			&\theta-\bar\theta\in C([0,T];H^2(\bbr)) \cap  L^2(0,T;H^3(\bbr)),
		\end{aligned}
	\end{align*}
	and {\it a priori} assumption
	\beq\label{apri-ass}
	\|(v-\bar v, u-\bar u, \theta-\bar\theta)\|_{L^\infty(0,T;H^2(\bbr))} \le \eps_1,
	\eeq
    where $\eps_1$ and $\delta_1$ are dependent on the states $v_{\pm}, u_{\pm}, \theta_{\pm},$ and coefficient‌s $R, \gamma, \kappa$.
	Then,  for all $t\le T$,
	\begin{align}
		\begin{aligned}\label{finest}
			&\sup_{t\in[0,T]}\Big[\|(v-\bar v, u-\bar u, \theta-\bar \theta)(t,\cdot)\|^2_{H^2(\bbr)}\Big] +\deltas\int_0^t|\dot{\mb{X}}(\tau)|^2 d\tau+\int_0^t ( \mathcal{G}^R(U) + \mathcal{G}^S(U))d\tau \\
			&\quad+\int_0^t\|(v-\bar v, u-\bar u)_x(\tau,\cdot)\|_{H^1(\bbr)}^2 d\tau+\int_0^t\|(\theta-\bar \theta)_x(\tau,\cdot)\|^2_{H^2(\bbr)} d\tau\\
			&\le C_0^2 \|\big(v_0(\cdot)-\bar v(0,\cdot),u_0(\cdot) -\bar u(0,\cdot),\theta_0(\cdot)-\bar\theta(0,\cdot)\big)\|^2_{H^2(\bbr)}  + C_0^2 \delta_0^{1/2} ,
		\end{aligned}
	\end{align}
	where $ \mathcal{G}^R(U),  \mathcal{G}^S(U)$ are as defined in \eqref{maingood}.
	In addition, by \eqref{X(t)},
	\beq\label{xprop}
	|\dot{\mb{X}}(t)|\leq C_0\|(v-\bar v,u-\bar u, \theta-\bar \theta)(t,\cdot)\|_{L^\infty(\bbr)},\qquad t\le T.
	\eeq
\end{proposition}

We postpone the proof of this key proposition to Sections \ref{sec-acontraction} and \ref{sec-vu}. 

\section{Relative entropy estimates}\label{sec-acontraction}
\setcounter{equation}{0}
This section is devoted to the proof of the following lemma by using the relative entropy method and the $a$-contraction with shifts.
\begin{lemma}\label{lem-zvh}
	Under the hypotheses of Proposition \ref{prop2}, there exists $C>0$ (independent of $\delta_1, \eps_1, T$) such that for all $t\in [0,T]$,
	\begin{align}
		\begin{aligned}\label{esthv}
			&\sup_{t\in[0,T]}\Big[\|(v-\bar v, u-\bar u, \theta-\bar \theta)(t,\cdot)\|^2_{L^2(\bbr)}\Big] +\deltas\int_0^t|\dot{\mb{X}}(\tau)|^2 d\tau \\
			&\quad+\int_0^t ( \mathcal{G}^R(U) + \mathcal{G}^S(U)) d\tau +\int_0^t\|( \theta-\bar \theta)_x(\tau,\cdot)\|^2_{L^2(\bbr)} d\tau\\
			&\leq C \|\big(v_0(\cdot)-\bar v(0,\cdot),u_0(\cdot) -\bar u(0,\cdot),\theta_0(\cdot)-\bar\theta(0,\cdot)\big)\|^2_{L^2(\bbr)} \\
            &\quad+C(\eps_1+\delta_0)\int_0^t\|(v-\bar v,u-\bar u)_x(\tau,\cdot)\|^2_{L^2(\bbr)} d\tau+ C\delta_0^{1/2},
		\end{aligned}
	\end{align}
	where the good terms $ \mathcal{G}^R(U),  \mathcal{G}^S(U)$ are defined in \eqref{maingood}.
\end{lemma}

We begin by introducing some preliminary results which will be instrumental in the proof of Lemma \ref{lem-zvh}.

\begin{lemma}[Poincar\'e-type inequality] (see \cite[Lemma 2.9]{KVJEMS21}). \label{lem-poin}
For any $f:[0,1]\to\bbr$ satisfying $\int_0^1 y(1-y)|f'|^2 dy<\infty$, 
\beq\label{poincare}
\int_0^1\Big|f-\int_0^1 f dy \Big|^2 dy\le \frac{1}{2}\int_0^1 y(1-y)|f'|^2 dy.
\eeq
\end{lemma}

\subsection{Wave interaction estimates}
We first present useful estimates for the wave interaction terms such as $Q_i^I$ in \eqref{QI1}. Combining the {\it a priori} assumption \eqref{apri-ass} with Sobolev's inequality, we obtain
\beq\label{smp1}
\|(v-\bar v, u-\bar u, \theta-\bar \theta)\|_{L^\infty((0,T)\times\bbr)}\le C\eps_1, \quad\mbox{and so}\quad v, u, \theta \in L^\infty((0,T)\times\bbr).
\eeq
Then, the ODE \eqref{X(t)} together with Lemma \ref{lemma1.3} yields that
\beq\label{dxbound}
\begin{aligned}
	|\dot{\mb X}(t)| &\le \frac{C}{\deltas}  \|(v-\bar v, u-\bar u, \theta-\bar \theta)(t,\cdot) \|_{L^\infty(\bbr)} \int_\bbr ( v^S)^{-\mb X}_\xi d\xi \\
	&\le C \|(v-\bar v, u-\bar u, \theta-\bar \theta)(t,\cdot) \|_{L^\infty(\bbr)}.
\end{aligned}
\eeq
This estimate proves \eqref{xprop} and will be used to derive the wave interaction estimates in Lemma \ref{lemma2.2}. Compared with the existing results in the reference \cite{KVW25}, the terms in $Q_i^I$ are fewer, so we state the following lemma without proof.

\begin{lemma}\label{lemma2.2}
	Let $\mb X$ be the shift defined by \eqref{X(t)} and $Q_i^I(~i=1,2)$, $Q_{i,\xi}^I(~i=1,2)$ and $Q_{1,\xi\xi}^I$ are defined in \eqref{QI1} and \eqref{QI2}. Under the same hypotheses as in Proposition \ref{prop2}, the following holds: for $i=1,2,$ and $\forall t\le T,$ $\exists$ a  constant $C>0$ independent of $T, \delta_R, \deltas$ and $\delta_C$ such that
	\beq\label{wave-interactions}
	\begin{aligned}
		&\|(Q_i^I,Q_{i,\xi}^I,Q_{1,\xi\xi}^I)\|_{L^2(\bbr)} \le C\deltas (\deltar+\delta_C) e^{-C \deltas t}+C\deltar\delta_Ce^{-Ct},\\[3mm]
		& \| |v^S_\xi| |\big((v^R)^{\mb X} -v_*,(\theta^R)^{\mb X} -\theta_*\big) |\|_{L^2(\bbr)}+\| |v^S_\xi| |\big((v^C)^{\mb X} -v^*,(\theta^C)^{\mb X} -\theta^*\big) |\|_{L^2(\bbr)}\\
		&\leq C\deltas^{3/2} (\deltar+\delta_C) e^{-C \deltas t} .
	\end{aligned}
	\eeq
\end{lemma}

\subsection{Construction of weight function}
‌To apply $a$-contraction with shifts method, we first introduce the weight function $a$. Define the weight function $a$ by
\begin{equation}\label{weight}
	a(\xi):=1+\frac{\lam}{\deltas}(v^S(\xi)-v^*),
\end{equation}
where the constant $\lam$ is chosen to be so small but far bigger than $\deltas:=|v_+-v^*|=v_+-v^*$ such that 
\beq\label{lamsmall}
\deltas\ll \lam \le C\sqrt{\deltas}.
\eeq
Thus, we have
\begin{equation}\label{a-bound}
	1<a(\xi)<1+\lam,
\end{equation}
and 
\begin{equation}\label{a-prime}
	a^\prime(\xi)=\frac{\lam}{\deltas}v^S_\xi>0,
\end{equation}
and so,
\beq\label{d-weight}
|a'|\sim\frac{\lam}{\deltas} |u^S_\xi|\sim \frac{\lam}{\deltas} |\theta^S_\xi|.
\eeq

\subsection{Relative entropy method} \label{ssec:ent}
Then, we compute the relative entropy defined by the entropy 
\[
s(V):=R\log v + \frac{R}{\gamma-1} \log\theta+const.,
\]
where $V=(v,u,E)$. Using the Gibbs relation and $E=e+\frac{u^2}{2}$, we obtain
\[
\theta ds = dE - udu + p dv,
\]
and so
\[
\nabla_V s(V) = \left( \frac{p}{\theta}, -\frac{u}{\theta}, \frac{1}{\theta} \right).
\]
Thus, for any $\bar V:=(\bar v,\bar u,\bar E)$ with $\bar E:=\bar e+\frac{\bar u^2}{2}$, $\bar e=\frac{R}{\gamma-1}\bar \theta + \mbox{const.}$, and $\bar p = \frac{R\bar\theta}{\bar v}$, and we can calculate the relative entropy
\begin{align}\label{ent}
    \begin{aligned}
	(-s)(V|\bar V) &= - s(V) + s(\bar V) + \nabla_V s (\bar V)\cdot (V-\bar V)\\
	&=  R  \left(\frac{v}{\bar v} -1- \log\frac{v}{\bar v}\right) +\frac{R}{\gamma-1} \left(\frac{\theta}{\bar\theta} -1- \log\frac{\theta}{\bar\theta}\right)  + \frac{(u-\bar u)^2}{2\bar\theta}.
    \end{aligned}
\end{align}

	To simplify the analysis of the relative entropy's evolution, we employ the non-conserved variables. 
Let $U = (v, u, \theta)^t$ and set $\bar{U} = (\bar{v}, \bar{u}, \bar{\theta})^t$ denote the superposition of a 1-rarefaction wave, a 2-viscous contact wave, and a 3-viscous shock wave shifted by $\mathbf{X}$, which is defined in \eqref{shwave}. 

As established in \eqref{ent}, for the mathematical entropy $\eta := -s$, the relative entropy between the solution $U$ and the profile $\bar{U}$ is given by
\[
\eta(U|\bar{U}) = R \left( \frac{v}{\bar{v}} - 1 - \log\frac{v}{\bar{v}} \right) + \frac{R}{\gamma-1} \left( \frac{\theta}{\bar{\theta}} - 1 - \log\frac{\theta}{\bar{\theta}} \right) + \frac{(u - \bar{u})^2}{2\bar{\theta}}.
\]
By introducing the convex function $\Phi(z) := z - 1 - \ln z$, this can be rewritten compactly as
\[
\eta(U|\bar{U}) = R \Phi\left( \frac{v}{\bar{v}} \right) + \frac{R}{\gamma-1} \Phi\left( \frac{\theta}{\bar{\theta}} \right) + \frac{(u - \bar{u})^2}{2\bar{\theta}}.
\]

Accordingly, the relative entropy weighted by $\bar{\theta}$ takes the form
\begin{equation}\label{def-rent}
	\bar{\theta} \eta(U|\bar{U}) = R \bar{\theta} \Phi\left( \frac{v}{\bar{v}} \right) + \frac{R \bar{\theta}}{\gamma-1} \Phi\left( \frac{\theta}{\bar{\theta}} \right) + \frac{(u - \bar{u})^2}{2}.
\end{equation}

In the subsequent analysis, we will study the evolution of the relative entropy weighted by the function $a(\xi - \mathbf{X})\bar{\theta}(t,\xi)$. Specifically, we consider the following integral
\begin{equation}\label{entropy-int}
    \int_{\mathbb{R}} a^{-\mathbf{X}}(\xi) \bar{\theta}(t,\xi) \eta\big(U(t,\xi)|\bar{U}(t,\xi)\big)  d\xi.
\end{equation}

‌For the subsequent estimates, we can decompose the relative entropy \eqref{entropy-int} into several parts by the following lemma.

\begin{lemma}\label{lem-relative}
	Let $a$ be the weight function defined by \eqref{weight}. Let $U$ be a solution to \eqref{NS-1}, and $\bar U$ the shifted wave given by \eqref{shwave}.
	Then,
	\begin{align*}
			\frac{d}{dt}\int_{\bbr} a^{-\mb X}(\xi)\bar \theta(t,\xi) \eta\big(U(t,\xi)|\bar U(t,\xi) \big) d\xi =\dot{\mb X} (t) \mb{Y}(U) +\mathcal{J}^{bad}(U) - \mathcal{J}^{good}(U),
	\end{align*}
    where
	\begin{align*}
			&\mb Y(U):=-\int_{\bbr} \! a_\xi^{-\mb X} \bar \theta\eta(U|\bar U ) d\xi +\int_\bbr a^{-\mb X}\Big[-R (\theta^S)^{-\mb{X}}_\xi\Phi\left(\frac{v}{\bar v}\right) - \frac{R }{\gamma-1}(\theta^S)^{-\mb{X}}_\xi\Phi\left(\frac{\theta}{\bar \theta}\right) \Big]d\xi\\
			&\quad\qquad\quad+\int_\bbr a^{-\mb X} \Big[(u^S)^{-\mb{X}}_\xi(u-\bar u)+\frac{(v^S)^{-\mb{X}}_\xi \bar p}{\bar v}(v-\bar v)+\frac{R}{\gamma-1}\frac{(\theta^S)^{-\mb{X}}_\xi}{\bar\theta}(\theta-\bar\theta)\Big]d\xi,
	\end{align*}
	\begin{align*}
			\mathcal{J}^{bad}(U)&:= \int_\bbr a_\xi^{-\mb X} (u-\bar u)(p-\bar p) d\xi \\
			&\quad+ \int_\bbr \!a^{-\mb X} \Big[R\Big((\theta_t^R-\sigma \theta_\x^R)+(\theta_t^C-\sigma \theta_\x^C)-\sigma (\theta^S)^{-\mb X}_\xi \Big)\Phi(\frac{v}{\bar v})-\frac{\bar p\bar u_\xi}{v\bar v}(v-\bar v)^2 \Big]d\x\\
			& \quad + \int_\bbr \!a^{-\mb X} \Big[\frac{R}{\gamma-1}\Big((\theta_t^R-\sigma \theta_\x^R)+(\theta_t^C-\sigma \theta_\x^C)-\sigma (\theta^S)^{-\mb X}_\xi \Big)\Phi\left(\frac{\theta}{\bar \theta}\right) \\
            &\quad\qquad - \frac{\bar u_\xi}{\theta}(\theta-\bar \theta)(p-\bar p)+\bar p \bar u_\xi \frac{(\theta-\bar\theta)^2}{\theta\bar\theta} \Big]d\x + \int_\bbr \!a^{-\mb X} \Big[  \kappa \frac{\theta-\bar\theta}{\theta} \Big(\frac{\theta_\xi}{v}- \frac{\bar \theta_\xi}{\bar v}  \Big) \Big]_\xi d\xi \\
			&\quad +\int_\bbr \!a^{-\mb X}  \left\{\kappa \frac{\theta-\bar\theta}{\theta^2}\theta_\x(\frac{\theta_\x}{v}-\frac{\bar\theta_\x}{\bar v})-\kappa\frac{(\theta-\bar\theta)_\xi}{\theta}\bar\theta_\x(\frac1v-\frac 1{\bar v})\right.\\
            &\quad\qquad \left.-\frac{(\theta-\bar\theta)^{2}}{\theta\bar\theta}\Big[\kappa(\frac{\bar\theta_\xi}{\bar v})_\xi+Q_2\Big] \right\}d\xi-\int_\bbr \!a^{-\mb X}\Big[Q_1(u-\bar u)+Q_2(1-\frac{\bar\theta}{\theta})\Big] d\xi,
	\end{align*}
	and
	\begin{align*}
			&\mathcal{J}^{good}(U):= \s\int_{\bbr} \! a_\xi^{-\mb X} \bar \theta\eta(U|\bar U ) d\xi  +\int_\bbr \!a^{-\mb X}\Big[\frac{\kappa}{v\theta}|(\theta-\bar \theta)_\xi|^2\Big]d\x.
	\end{align*}
\end{lemma}
\begin{remark}
	Since $\s a'(\xi) >0$, $\mathcal{J}^{good}$ consists of nonnegative terms. 
\end{remark}

\begin{proof}
	First, from \eqref{def-rent} we have
	\begin{align}
		\begin{aligned}\label{fieq}
			&\frac{d}{dt}\int_{\bbr} a^{-\mb X}(\xi)\bar \theta(t,\xi) \eta\big(U(t,\xi)|\bar U(t,\xi) \big) d\xi \\
			&= -\dot{\mb X} (t)  \int_{\bbr} \! a_\xi^{-\mb X} \bar \theta\eta(U|\bar U ) d\xi + \int_{\bbr} \! a^{-\mb X} \partial_t\bigg[R\bar\theta\Phi\left(\frac{v}{\bar v}\right)+\frac{R\bar\theta}{\gamma-1}\Phi\left(\frac{\theta}{\bar \theta}\right) + \frac{(u-\bar u)^2}{2}\bigg] d\xi.
		\end{aligned} 
	\end{align}
	Then,  using the two systems \eqref{NS-1} and \eqref{bar-system} (satisfied by $\bar U$), we get the perturbation equations
	\begin{equation}\label{pereq1}
\begin{cases}
\begin{aligned}
&(v-\bar v)_t-\s (v-\bar v)_\xi-\dot{\mb X}(t)(v^S)^{-\mb X}_\xi
-(u-\bar u)_\xi=0,\\[3mm]
&(u-\bar u)_t-\s (u-\bar u)_\xi-\dot{\mb X}(t)(u^S)^{-\mb X}_\xi
+(p-\bar p)_\xi=-Q_1,\\[3mm]
&\frac{R}{\gamma-1}(\theta-\bar\theta)_t
-\frac{R\s}{\gamma-1}(\theta-\bar\theta)_\xi
-\frac{R}{\gamma-1}\dot{\mb X}(t)(\theta^S)^{-\mb X}_\xi
+(pu_\xi-\bar p\bar u_\xi)\\
&\qquad\qquad
=\kappa\left(\frac{\theta_\xi}{v}
-\frac{\bar\theta_\xi}{\bar v}\right)_\xi-Q_2.
\end{aligned}
\end{cases}
\end{equation}
	Using $\Phi'(z)=1-1/z$ and $\eqref{pereq1}_1$ and $\eqref{bar-system}_1$, we have
	\begin{align*}
			\partial_t \bigg[\bar\theta\Phi\left(\frac{v}{\bar v}\right)\bigg] &= \bar\theta_t \Phi\left(\frac{v}{\bar v}\right) + \bar\theta \Phi' \left(\frac{v}{\bar v}\right) \left(\frac{v}{\bar v}\right)_t \\
			&=  \bar\theta_t \Phi\left(\frac{v}{\bar v}\right) + \bar\theta\left(\frac{1}{\bar v}-\frac{1}{v} \right) \bigg[ (v-\bar v)_t +\bar v_t \left(1- \frac{v}{\bar v}\right)  \bigg]  \\
			&=  \bar\theta_t \Phi\left(\frac{v}{\bar v}\right) + \bar\theta\left(\frac{1}{\bar v}-\frac{1}{v} \right) \bigg[ \s (v-\bar v)_\xi +\dot{\mb X}(t)(v^S)^{-\mb X}_\xi +(u-\bar u)_\x  \bigg] \\
			&\quad +  \bar\theta\left(\frac{1}{\bar v}-\frac{1}{v} \right)  \left(1- \frac{v}{\bar v}\right) \bigg[ \s \bar v_\xi -\dot{\mb X}(t)(v^S)^{-\mb X}_\xi +\bar u_\x  \bigg]. 
	\end{align*}
	In addition, since
	\[
	\partial_\x \bigg[\s\bar\theta\Phi\left(\frac{v}{\bar v}\right)\bigg] = \s\bar\theta_\x \Phi\left(\frac{v}{\bar v}\right)  +\s \bar\theta\left(\frac{1}{\bar v}-\frac{1}{v} \right) \bigg[  (v-\bar v)_\xi  +\bar v_\x \left(1- \frac{v}{\bar v}\right)  \bigg] ,
	\]
	we have
	\begin{align}
		\begin{aligned} \label{tvest}
			&\partial_t \bigg[R\bar\theta\Phi\left(\frac{v}{\bar v}\right)\bigg] -\partial_\x \bigg[R\s\bar\theta\Phi\left(\frac{v}{\bar v}\right)\bigg]  \\
			&= R( \bar\theta_t -\s\bar\theta_\x) \Phi\left(\frac{v}{\bar v}\right) + \dot{\mb X}(t)(v^S)^{-\mb X}_\xi \frac{\bar p}{\bar v} (v-\bar v)  -\frac{\bar p\bar u_\xi}{v\bar v}(v-\bar v)^2 + R\bar\theta \left(\frac{1}{\bar v} -\frac{1}{v} \right) (u-\bar u)_\x .
		\end{aligned} 
	\end{align}
	Likewise, using $\eqref{pereq1}_3$ and $\eqref{bar-system}_3$, we can obtain
	\begin{align}
		\begin{aligned} \label{tuest} 
			&\partial_t \bigg[\frac{R\bar\theta}{\gamma-1}\Phi\left(\frac{\theta}{\bar \theta}\right)  \bigg] -\partial_\x \bigg[\frac{R\s\bar\theta}{\gamma-1}\Phi\left(\frac{\theta}{\bar \theta}\right)  \bigg]  \\
			&= \frac{R}{\gamma-1}( \bar\theta_t -\s\bar\theta_\x)\Phi\left(\frac{\theta}{\bar \theta}\right)  + \dot{\mb X}(t)\frac{R}{\gamma-1}\frac{(\theta^S)^{-\mb{X}}_\xi}{\bar\theta}(\theta-\bar\theta) \\
			&\quad + \frac{R}{v} (\bar\theta - \theta) (u-\bar u)_\x -\frac{\bar u_\x}{\theta}(\theta-\bar\theta) (p-\bar p) + \bar p \bar u_\xi \frac{(\theta-\bar\theta)^2}{\theta\bar\theta}   \\
			&\quad +\Big[\kappa \frac{\theta-\bar\theta}{\theta} \Big(\frac{\theta_\xi}{v}- \frac{\bar \theta_\xi}{\bar v}  \Big) \Big]_\xi -\frac{\kappa}{v\theta}|(\theta-\bar \theta)_\xi|^2 +\kappa \frac{\theta-\bar\theta}{\theta^2}\theta_\x(\frac{\theta_\x}{v}-\frac{\bar\theta_\x}{\bar v}) \\
			&\quad -\kappa\frac{(\theta-\bar\theta)_\xi}{\theta}\bar\theta_\x(\frac1v-\frac 1{\bar v}) -Q_2(1-\frac{\bar\theta}{\theta}) -\frac{(\theta-\bar\theta)^{2}}{\theta\bar\theta}\Big[\kappa(\frac{\bar\theta_\xi}{\bar v})_\xi+Q_2\Big].
		\end{aligned} 
	\end{align}
	Using  $\eqref{pereq1}_2$, we have
	\begin{align}
		\begin{aligned}  \label{ttest}
			&\partial_t \bigg[\frac{(u-\bar u)^2}{2}\bigg] - \partial_\x \bigg[ \s \frac{(u-\bar u)^2}{2}\bigg] \\
			&= \dot{\mb X}(t)(u^S)^{-\mb X}_\xi (u-\bar u) -\Big[ (p-\bar p) (u-\bar u)\Big]_\x + (p-\bar p) (u-\bar u)_\x  -Q_1 (u-\bar u).
		\end{aligned} 
	\end{align}
	Therefore, combining \eqref{tvest}, \eqref{tuest} and \eqref{ttest}, we have
	\begin{align*}
			&\partial_t\bigg[R\bar\theta\Phi\left(\frac{v}{\bar v}\right)+\frac{R\bar\theta}{\gamma-1}\Phi\left(\frac{\theta}{\bar \theta}\right) + \frac{(u-\bar u)^2}{2}\bigg]  \\
			& = \s  \partial_\x \bigg[ \bar\theta \eta (U |\bar U)\bigg] + ( \bar\theta_t -\s\bar\theta_\x) \bigg[ R\Phi\left(\frac{v}{\bar v}\right) +\frac{R}{\gamma-1}\Phi\left(\frac{\theta}{\bar \theta}\right) \bigg] \\
			&\quad + \dot{\mb X}(t) \Big[\frac{(v^S)^{-\mb{X}}_\xi \bar p}{\bar v}(v-\bar v)+\frac{R}{\gamma-1}\frac{(\theta^S)^{-\mb{X}}_\xi}{\bar\theta}(\theta-\bar\theta)+(u^S)^{-\mb{X}}_\xi(u-\bar u)\Big] \\
			&\quad -\frac{\bar p\bar u_\xi}{v\bar v}(v-\bar v)^2 -\frac{\bar u_\x}{\theta}(\theta-\bar\theta) (p-\bar p) +\partial_\x \Big[  \kappa \frac{\theta-\bar\theta}{\theta} \Big(\frac{\theta_\xi}{v}- \frac{\bar \theta_\xi}{\bar v}  \Big) \Big] \\
			&\quad +\kappa \frac{\theta-\bar\theta}{\theta^2}\theta_\x(\frac{\theta_\x}{v}-\frac{\bar\theta_\x}{\bar v})-\kappa\frac{(\theta-\bar\theta)_\xi}{\theta}\bar\theta_\x(\frac1v-\frac 1{\bar v}) -Q_1(u-\bar u)-Q_2(1-\frac{\bar\theta}{\theta}) \\
            &\quad -\frac{\kappa}{v\theta}|(\theta-\bar \theta)_\xi|^2 -\frac{(\theta-\bar\theta)^{2}}{\theta\bar\theta}\Big[\kappa(\frac{\bar\theta_\xi}{\bar v})_\xi+Q_2\Big].
	\end{align*}
	In addition, since
	\[
	\bar\theta_t -\s\bar\theta_\x = (\theta_t^R-\sigma \theta_\x^R)+(\theta_t^C-\sigma \theta_\x^C)-\sigma (\theta^S)^{-\mb X}_\xi - \dot{\mb X}(t) (\theta^S)^{-\mb X}_\xi,
	\]
	we substitute the above relations into \eqref{fieq} to get the desired representation.
	
\end{proof}

\subsection{Decompositions}
In order to prove \eqref{esthv}, Lemma \ref{lem-relative} allows us to decompose the terms $\mathcal{J}^{good}$ and $\mathcal{J}^{bad}$ as follows:
\begin{align}
	\begin{aligned}\label{eq-1}
		&\frac{d}{dt}\int_{\bbr} a^{-\mb X}(\xi)\bar \theta(t,\xi) \eta\big(U(t,\xi)|\bar U(t,\xi) \big) d\xi \\
		&\qquad = \dot{\mb X} (t) \mb{Y}(U) +\sum_{i=1}^6 \mb B_i(U) +\mb S_1(U)+\mb S_2(U)  -\mb G(U)  -\mb D(U),
	\end{aligned}
\end{align}
where
\begin{align*}
		&\mb B_1(U):=  \int_\bbr a^{-\mb X} (v^S)_\xi^{-\mb X} \Big[ \frac{(\gamma+1)p^*\sigma^*}{2(v^*)^2} (v-\bar v)^2+\frac{R\sigma^*}{2v^*\theta^*}(\theta-\bar\theta)^2-\frac{p^*\sigma^*}{v^*\theta^*}(v-\bar v)(\theta-\bar\theta)\Big]d\x,\\
		&\mb B_2(U):=  \int_\bbr a_\xi^{-\mb X} (u-\bar u)(p-\bar p) d\xi, \\
		&\mb B_3(U):= - \int_\bbr \!a_\xi^{-\mb X} \Big[  \kappa \frac{\theta-\bar\theta}{\theta} \Big(\frac{\theta_\xi}{v}- \frac{\bar \theta_\xi}{\bar v}  \Big) \Big] d\xi, \\
		&\mb B_4(U):= \int_\bbr \!a^{-\mb X}  \Big\{ \kappa \frac{\theta-\bar\theta}{\theta^2}\theta_\x(\frac{\theta_\x}{v}-\frac{\bar\theta_\x}{\bar v})-\kappa\frac{(\theta-\bar\theta)_\xi}{\theta}\bar\theta_\x(\frac1v-\frac 1{\bar v})-\kappa\frac{(\theta-\bar\theta)^{2}}{\theta\bar\theta}(\frac{\bar\theta_\xi}{\bar v})_\xi\Big\} d\xi,\\
		&\mb B_5(U) := \int_\bbr \!a^{-\mb X} \Big\{\frac{R}{\gamma-1}\Big((\theta_t^R-\sigma \theta_\x^R)+(\theta_t^C-\sigma \theta_\x^C)-\sigma (\theta^S)^{-\mb X}_\xi \Big)\Phi\left(\frac{\theta}{\bar \theta}\right) - \frac{\bar u_\xi}{\theta}(\theta-\bar \theta)(p-\bar p) \\
        &\qquad\qquad+\bar p \bar u_\xi \frac{(\theta-\bar\theta)^2}{\theta\bar\theta} - (v^S)_\xi^{-\mb X} \Big[\frac{R\sigma^*}{2v^*\theta^*}(\theta-\bar\theta)^2-\frac{p^*\sigma^*}{v^*\theta^*}(v-\bar v)(\theta-\bar\theta)\Big]\Big\}d\x,\\
        &\mb B_6(U):=\int_\bbr \!a^{-\mb X} \Big[R\Big((\theta_t^R-\sigma \theta_\x^R)+(\theta_t^C-\sigma \theta_\x^C)-\sigma (\theta^S)^{-\mb X}_\xi \Big)\Phi\left(\frac{v}{\bar v}\right)-\frac{\bar p\bar u_\xi}{v\bar v}(v-\bar v)^2\\
        &\qquad\qquad-(v^S)_\xi^{-\mb X} \frac{(\gamma+1)p^*\sigma^*}{2(v^*)^2} (v-\bar v)^2\Big]d\x,\\
		&\mb S_1(U):= - \int_\bbr \!a^{-\mb X} Q_1(u-\bar u) d\xi, \quad\qquad  \mb S_2(U):= - \int_\bbr \!a^{-\mb X} Q_2\Big[(1-\frac{\bar\theta}{\theta}) +\frac{(\theta-\bar\theta)^{2}}{\theta\bar\theta}\Big]d\xi, 
\end{align*}
and 
\begin{align*}
		& \mb G(U) := \s\int_{\bbr} \! a_\xi^{-\mb X} \bar \theta\eta(U|\bar U ) d\xi , \\
		&\mb D(U) := \int_\bbr \!a^{-\mb X}\Big[ \frac{\kappa}{v\theta}|(\theta-\bar \theta)_\xi|^2\Big]d\x.
\end{align*}

We decompose the functional $\mb Y$ as follows:
\[
\mb Y:= \sum_{i=1}^6\mb{Y}_i,
\]
where
\begin{align*}
		&\mb Y_1(U):= \int_\bbr a^{-\mb X} (u^S)^{-\mb{X}}_\xi(u-\bar u) d\xi,\\
		&\mb Y_2(U):= \int_{\bbr} a^{-\mb X}\frac{(v^S)^{-\mb{X}}_\xi \bar p}{\bar v}(v-\bar v) d\xi,\\
		&\mb Y_3(U):=\int_{\bbr}a^{-\mb X} \frac{R}{\gamma-1}\frac{(\theta^S)^{-\mb{X}}_\xi}{\bar\theta}(\theta-\bar\theta) d\xi,\\
		&\mb Y_4(U):= - \int_\bbr a^{-\mb X} R (\theta^S)^{-\mb{X}}_\xi\Phi(\frac{v}{\bar v}) d\xi,\\
		&\mb Y_5(U):= -\int_\bbr a^{-\mb X} \frac{R }{\gamma-1}(\theta^S)^{-\mb{X}}_\xi\Phi(\frac{ \theta}{\bar \theta}) d\xi , \\
		&\mb Y_6(U):= -\int_{\bbr} \! a_\xi^{-\mb X} \bar \theta\eta(U|\bar U ) d\xi .
\end{align*}

Notice from \eqref{X(t)} that 
\beq\label{defxy}
\dot{\mb{X}}(t)=-\frac{M}{\delta_S} (\mb{Y}_1+\mb{Y}_2+\mb{Y}_3),
\eeq
and so,
\begin{equation}\label{XY}
	\dot{\mb{X}}(t)\mb{Y}= -\frac{\deltas}{M} |\dot{\mb X} (t) |^2 +\dot{\mb{X}}(t)\sum_{i=4}^6\mb{Y}_i.
\end{equation}

\subsection{Leading order estimates}
First of all, it follows from \eqref{eq-1} and \eqref{XY} to have
\begin{align*}
	\begin{aligned}
		\frac{d}{dt}\int_{\bbr} a^{-\mb X}\bar \theta \eta\big(U|\bar U \big) d\xi  
		&= -\frac{\deltas}{2M} |\dot{\mb X}|^2 +\mb B_1+\mb B_2 - \mb{G} -\frac{3}{4}\mb D \\
		&\quad -\frac{\deltas}{2M} |\dot{\mb X}|^2 +\dot{\mb{X}}\sum_{i=4}^6\mb{Y}_i  +\sum_{i=3}^6 \mb B_i +\mb S_1+\mb S_2   -\frac{1}{4} \mb D.
	\end{aligned}
\end{align*}
Then, using Young's inequality, there exists a constant $C>0$ such that
\begin{align}
	\begin{aligned}\label{ineq-2}
		\frac{d}{dt}\int_{\bbr} a^{-\mb X}\bar \theta \eta\big(U|\bar U \big) d\xi  
		&\leq -\frac{\deltas}{2M} |\dot{\mb X}|^2 +\mb B_1+\mb B_2 - \mb{G} -\frac{3}{4}\mb D \\
		&\quad -\frac{\deltas}{4M} |\dot{\mb X}|^2 +\frac{C}{\deltas}\sum_{i=4}^6 |\mb{Y}_i|^2  +\sum_{i=3}^6 \mb B_i +\mb S_1+\mb S_2   -\frac{1}{4} \mb D.
	\end{aligned}
\end{align}

Therefore, in order to prove Lemma \ref{lem-zvh}, we need to estimate the terms in \eqref{ineq-2} and we first give the leading order estimates in the following lemma.
\begin{lemma}\label{lem-sharp1}
	There exists $C>0$ such that 
	\begin{align}\label{eq-sharp1}
		\begin{aligned}
			&-\frac{\deltas}{2M} |\dot{\mb X}|^2 +\mb B_1+\mb B_2 - \mb{G} -\frac{3}{4}\mb D \\&\le -C\mathcal{G}^S(U) +C (\eps_1+\delta_0)\| (u-\bar u)_\x\|^2_{L^2(\bbr)} + C ( \deltar^{4/3}+ \delta_C^{4/3})  \deltas^{4/3} e^{-C\deltas t},
		\end{aligned}
	\end{align}
\end{lemma}
\begin{proof}
	To exploit the good structure in the integrands, we need to rewrite the main terms in terms of the new variables $y$ and $w$:
    \begin{equation}\label{omega}
		w :=u^{\mb X}-\bar u^{\mb X}
	\end{equation}
	and
	\begin{equation}\label{y-xi}
		y:=\frac{ v^S(\xi)-v^*}{\deltas}.
	\end{equation}
    Since  $\deltas:=v_+-v^*>0$ and $v^S_\xi>0$, the change of variable $\xi\in\bbr\mapsto y\in (0,1)$ is well-defined, with
	\begin{equation}\label{dery}
		\frac{dy}{d\xi} =\frac{v^S_\xi}{\delta_S}>0.
	\end{equation}
	Note also that $a(\xi)=1+\lam y$ and  $a'(\xi)=\lam (dy/d\xi)>0$.\\
    
	To find the sharp estimates, we will also use the $O(1)$-constants $p^*, \s^*$ defined in \eqref{pstar},
	which are indeed independent of the small constants $\delta_1, \eps_1$, since $\frac{v_+}{2}\le v^*\le v_+$ and $\frac{\theta_+}{2}\le \theta^*\le \theta_+$. \\
	
	First, there are some good terms in the combined term, so we first estimate $\mb B_2 - \mb G$ to get the good terms $\mb G_1, \mb G_2$, which are defined in \eqref{b2g}. And then, we estimate the remaining terms one by one.
	
	\noindent$\bullet$ {\bf Estimates on $\mb B_2 - \mb G$: } \\
    First, by \eqref{def-rent}, 
\begin{align*}
\begin{aligned}
&(u-\bar u)(p-\bar p) -\sigma \bar \theta \eta\big(U(t,\xi)|\bar U(t,\xi) \big) \\
&\qquad = (u-\bar u)(p-\bar p) -\sigma \Big[R\bar\theta\Phi(\frac{v}{\bar v})+\frac{R\bar\theta}{\gamma-1}\Phi(\frac{\theta}{\bar \theta}) + \frac{(u-\bar u)^2}{2} \Big].
\end{aligned}
\end{align*}
Using  $p-\bar p =\frac{R}{v} (\theta-\bar\theta) -\frac{\bar p}{v} (v-\bar v) $, and \eqref{sm1},  we have
\begin{align*}
&(u-\bar u)(p-\bar p) -\sigma \bar \theta \eta\big(U(t,\xi)|\bar U(t,\xi) \big) \\
&\quad = (u-\bar u)\Big(\frac{R}{v} (\theta-\bar\theta) -\frac{\bar p}{v} (v-\bar v)  \Big) -\sigma \Big[R\bar\theta\Phi(\frac{v}{\bar v})+\frac{R\bar\theta}{\gamma-1}\Phi(\frac{\theta}{\bar \theta}) + \frac{(u-\bar u)^2}{2} \Big]\\
&\quad\le (u-\bar u)\Big[\frac{R}{v^*}(\theta-\bar\theta)-\frac{p^*}{v^*}(v-\bar v)\Big]  -\frac{R\sigma^*\theta^*}{2(v^*)^2}(v-\bar v)^2 - \frac{R\sigma^*}{2(\gamma-1)\theta^*}(\theta-\bar\theta)^2 \\
&\quad\quad -\frac{\sigma^*}{2}(u-\bar u)^2 +C\big(|v-\bar v| + |\bar v - v^*| + |\bar\theta - \theta^*| \big) |(v-\bar v, u-\bar u,\theta-\bar\theta)|^2\\
&\quad=  -\frac{R\sigma^*\theta^*}{2(v^*)^2} \Big[(v-\bar v)+\frac{u-\bar u}{\sigma^*}  \Big]^2 - \frac{R\sigma^*}{2(\gamma-1)\theta^*} \Big[(\theta-\bar\theta)-\frac{(\gamma-1)\theta^*}{v^*\s^*}(u-\bar u)\Big]^2\\
&\quad\quad +C\big(|v-\bar v|+ \deltas + |(v^R - v_* ,\theta^R - \theta_*)| +|(v^C - v^*,\theta^C - \theta^*)|  \big) |(v-\bar v, u-\bar u,\theta-\bar\theta)|^2,
\end{align*}
where the last equality is obtained by using
\[
\frac{R\theta^*}{2(v^*)^2 \s^*} + \frac{R(\gamma-1)\theta^*}{2(v^*)^2 \s^*} = \frac{R\gamma\theta^*}{2(v^*)^2 \s^*} = \frac{\s^*}{2} \quad \mbox{by }  \sigma^*= \frac{\sqrt{\gamma R\theta^*}}{v^*}.
\]
	Therefore, we can obtain
	\beq\label{b2g}
    \begin{aligned}
        \mb B_2(U) - \mb G (U) &= \int_\bbr a_\xi^{-\mb X} \left[(u-\bar u)(p-\bar p) -\s \bar \theta\eta(U|\bar U )\right]d\xi\\
        &\le -\frac{R\sigma^*\theta^*}{2(v^*)^2} \int_\bbr a_\xi^{-\mb X}  \Big[(v-\bar v)+\frac{u-\bar u}{\sigma^*}  \Big]^2 d\xi \\
        &\quad - \frac{R\sigma^*}{2(\gamma-1)\theta^*} \int_\bbr a_\xi^{-\mb X}  \Big[(\theta-\bar\theta)-\frac{(\gamma-1)\theta^*}{v^*\s^*}(u-\bar u)\Big]^2 d\x \\
        &\quad+\Big(C\deltas \int_\bbr a_\xi^{-\mb X}  |(v-\bar v, u-\bar u,\theta-\bar\theta)|^2 d\x \\
		&\quad + C  \int_\bbr a_\xi^{-\mb X}  |(v-\bar v, u-\bar u,\theta-\bar\theta)|^3 d\x \\
		&\quad  + C  \int_\bbr a_\xi^{-\mb X} \big( |(v^R - v_* ,\theta^R - \theta_*)| +|(v^C - v^*,\theta^C - \theta^*)|  \big) \\
        &\qquad \cdot|(v-\bar v, u-\bar u,\theta-\bar\theta)|^2 d\x \Big)\\
        &:=-\mb G_1 (U) -  \mb G_2 (U) + \mb B_{new} (U).
    \end{aligned}
	\eeq

	The two good terms $\mb G_1, \mb G_2$ will be utilized in the remaining estimates, while
	the bad term $\mb B_{new}$ can be controlled by these good terms $\mb G_1, \mb G_2$ and $\mathcal{G}^S$,
	as follows.
    
    Using \eqref{a-prime} and the interpolation inequality, and $\lam\le C\sqrt{\deltas}$ by \eqref{lamsmall}, together with 
    $$
    |(v-\bar v,u-\bar u, \theta-\bar\theta)|\leq C\eps_1,
    $$
    we have

\begin{align*}
&  \int_\bbr a_\xi^{-\mb X}  |(v-\bar v, u-\bar u,\theta-\bar\theta)|^3 d\x \\
&\quad \le C \int_\bbr  |a_\xi^{-\mb X} |\Big|(v-\bar v)+\frac{u-\bar u}{\sigma^*}\Big| ^3d\xi+C\int_\bbr |a_\xi^{-\mb X}||u-\bar u|^3d\xi\\
&\qquad+C \int_\bbr |a_\xi^{-\mb X} |\Big| (\theta-\bar\theta)-\frac{(\gamma-1)\theta^*}{v^*\s^*}(u-\bar u) \Big|^3 d\xi \\
&\quad \le C\eps_1 (\mb{G}_1+\mb{G}_2)+C \frac{\l}{\deltas} \int_\bbr |(v^S)_\xi^{-\mb X} |  \|u-\bar u\|_{L^\infty(\bbr)}^2 |u-\bar u| d\xi \\
&\quad \le C \eps_1(\mb{G}_1+\mb{G}_2)+C\frac{\lam}{\deltas}  \|(u-\bar u)_\xi\|_{L^2(\bbr)} \|u-\bar u\|_{L^2(\bbr)}  \\
&\qquad\cdot\sqrt{\int_\bbr   |(v^S)_\xi^{-\mb X} | |u-\bar u|^2 d\xi} \sqrt{\int_\bbr  |(v^S)_\xi^{-\mb X} |  d\xi} \\
&\quad \le C \eps_1(\mb{G}_1+\mb{G}_2)+C\eps_1 \frac{\lam}{\sqrt{\deltas}}  \|(u-\bar u)_\xi\|_{L^2(\bbr)}  \sqrt{\mathcal{G}^S}\\
&\quad   \le C \eps_1(\mb{G}_1+\mb{G}_2 + \|(u-\bar u)_\xi\|_{L^2(\bbr)}^2  +\mathcal{G}^S).
\end{align*}
Likewise, using the interpolation inequality, Lemma \ref{lemma2.2} and {\it a priori} assumption‌ \eqref{apri-ass}, we have
\begin{align*}
& \int_\bbr a_\xi^{-\mb X} \big( |(v^R - v_* ,\theta^R - \theta_*)| +|(v^C - v^*,\theta^C - \theta^*)|  \big) |(v-\bar v, u-\bar u,\theta-\bar\theta)|^2 d\x \\
&\quad \le C(\delta_R +\delta_C)  (\mb{G}_1+\mb{G}_2) \\
&\qquad+C \frac{\l}{\deltas} \int_\bbr  |(v^S)_\xi^{-\mb X} |  \big( |(v^R - v_* ,\theta^R - \theta_*)| +|(v^C - v^*,\theta^C - \theta^*)|  \big) |u-\bar u|^2 d\xi  \\
&\quad\le C(\delta_R +\delta_C)  (\mb{G}_1+\mb{G}_2)+C\frac{\lam}{\deltas} \|u-\bar u\|_{L^4(\bbr)}^2 \Big[\| |(v^S)_\xi^{-\mb X}| | (v^R - v_* ,\theta^R - \theta_*)|\|_{L^2(\bbr)}\\
 &\qquad \quad +\| |(v^S)_\xi^{-\mb X} | | (v^C - v^* ,\theta^C - \theta^*)|\|_{L^2(\bbr)}  \Big]  \\
 &\quad \le C \delta_0(\mb{G}_1+\mb{G}_2)+C\frac{1}{\sqrt\deltas} \|(u-\bar u)_\xi\|_{L^2(\bbr)}^{1/2}  \|u-\bar u\|_{L^2(\bbr)}^{3/2}  \deltas^{\frac32}(\deltar+\delta_C)e^{-C\deltas t}\\
&\quad\leq  C \delta_0(\mb{G}_1+\mb{G}_2)+C\|(u-\bar u)_\xi\|_{L^2(\bbr)}^{1/2} \eps_1^{\frac32}\deltas (\deltar+\delta_C)e^{-C\deltas t}\\
&\quad\leq  C \delta_0(\mb{G}_1+\mb{G}_2)+C\eps_1 \|(u-\bar u)_\xi\|_{L^2(\bbr)}^{2} +C \deltas^{4/3} ( \deltar^{4/3}+ \delta_C^{4/3}) e^{-C\deltas t}.
\end{align*}
	Thus, 
	\beq\label{bnew}
    \begin{aligned}
	\mb B_{new} \le& C (\l+\eps_1+\delta_0) ( \mb{G}_1+\mb{G}_2 +\mathcal{G}^S)+C (\eps_1+\delta_0) \| (u-\bar u)_\x\|^2_{L^2(\bbr)}\\
    &+C \deltas^{4/3} ( \deltar^{4/3}+ \delta_C^{4/3}) e^{-C\deltas t}.
    \end{aligned}
	\eeq
	
\noindent$\bullet$ {\bf Estimates on $-\frac{\deltas}{2M} |\dot{\mb X}|^2$:}\\
    
First, to estimate the term $-\frac{\deltas}{2M} |\dot{\mb X}|^2$, we will estimate $\mb Y_1, \mb Y_2, \mb Y_3$ due to \eqref{defxy}. By the change of variable \eqref{dery} and $\eqref{VS}_1$, together with the change of variable $\xi\mapsto\x+\mb X(t)$, we have
\[
\mb Y_1 = \int_\bbr a u^S_\xi(u^{\mb{X}}-\bar u^{\mb{X}}) d\xi =-\deltas \s\int_0^1 a w  dy.
\]
Using \eqref{sm1} and $|a-1|\le \lam$, we have
\beq\label{y1}
\left|\mb Y_1 + \deltas \s^*\int_0^1 w dy \right| \le C\deltas \l \int_0^1|w| dy.
\eeq
Similarly, for 
\begin{align*}
 \mb Y_2 &= \int_{\bbr} a \frac{v^S_\xi \bar p^{\mb X}}{\bar v^{\mb X}}(v^{\mb X}-\bar v^{\mb X}) d\xi \\
 &=-\int_{\bbr}a \frac{v^S_\xi \bar p^{\mb X} }{\bar v^{\mb X} \s^*}(u^{\mb X}-\bar u^{\mb X}) d\x +\int_{\bbr} a \frac{v^S_\xi \bar p^{\mb X}}{\bar v^{\mb X}}\Big[(v^{\mb X}-\bar v^{\mb X})+\frac{u^{\mb X}-\bar u^{\mb X}}{\sigma^*}\Big] d\x,
\end{align*} 
using 
\begin{align*} 
|\bar p^{\mb X} - p^*| & \le C( |\bar v^{\mb X} - v^*| +|\bar \theta^{\mb X} - \theta^*| ) \\
& \le C(\deltar+\deltas +\delta_C) \le C\delta_0, 
\end{align*} 
and \eqref{a-prime}, we have
\beq\label{y2}
\begin{aligned}
\left|\mb Y_2 + \frac{p^*\deltas}{v^*\sigma^*}\int_0^1 w dy\right| &\le C\deltas (\lam+\delta_0)\int_0^1|w| dy+ C\frac{\deltas}{\l}\int_{\bbr} a_\xi \left|(v^{\mb X}-\bar v^{\mb X})+\frac{u^{\mb X}-\bar u^{\mb X}}{\sigma^*}\right| d\x\\
&\le C\deltas (\lam+\delta_0)\int_0^1|w| dy+ C\frac{\deltas}{\l}\int_{\bbr} a^{-\mb X}_\xi \left|(v-\bar v)+\frac{u-\bar u}{\sigma^*}\right| d\x.
\end{aligned}
\eeq
Likewise, since
\begin{align*}
\mb Y_3 &= \int_{\bbr}a^{-\mb X} \frac{R}{\gamma-1}\frac{(\theta^S)^{-\mb{X}}_\xi}{\bar\theta}(\theta-\bar\theta) d\xi  \\
  &= \frac{R \theta^*}{v^*\s^*} \int_{\bbr} a \frac{\theta^S_\xi}{\bar\theta^{\mb X}}  (u^{\mb X}-\bar u^{\mb X}) d\x \\
&\quad+\int_{\bbr}a^{-\mb X} \frac{R}{\gamma-1}\frac{(\theta^S)^{-\mb{X}}_\xi}{\bar\theta}\Big[(\theta-\bar\theta)-\frac{(\gamma-1)\theta^*}{v^*\s^*}(u-\bar u)\Big] d\x,
\end{align*}
using $|\theta-\theta^*| \le C(\eps_1 +\delta_0)$ and \eqref{theta-s}, \eqref{d-weight}, we have
\beq\label{y3}
\begin{aligned}
\left|\mb Y_3 + \frac{(\gamma-1)p^*\deltas}{v^*\sigma^*}\int_0^1 w dy \right| \le & C\deltas (\lam+\delta_0+\varepsilon_1)\int_0^1|w| dy\\
  &+ C\frac{\deltas}{\l}\int_{\bbr} a^{-\mb{X}}_\xi \left| (\theta-\bar\theta)-\frac{(\gamma-1)\theta^*}{v^*\s^*}(u-\bar u)  \right|  d\x.
\end{aligned}
\eeq
Therefore, using \eqref{defxy}, \eqref{y1}, \eqref{y2} and \eqref{y3} together with $\s^* =\sqrt\frac{\gamma p^*}{v^*}$, we have
\begin{align*}
\di\left| \dot{\mb X} -2\sigma^*M\int_0^1 w dy\right| & \le C (\lam+\delta_0+\eps_1)\int_0^1|w| dy+ \frac{C}{\l}\int_{\bbr} a^{-\mb{X}}_\xi \left|(v-\bar v)+\frac{u-\bar u}{\sigma^*}\right| d\x\\
&\quad + \frac{C}{\l}\int_{\bbr} a^{-\mb{X}}_\xi \left| (\theta-\bar\theta)-\frac{(\gamma-1)\theta^*}{v^*\s^*}(u-\bar u) \right| d\x,
\end{align*}
which implies
\[
 \left( \left|2\sigma^*M \int_0^1 w dy\right| - |\dot{\mb X} | \right)^2  \le  C(\lam+\delta_0+\eps_1)^2 \int_0^1|w|^2 dy + \frac{C}{\l^2}  (\mb G_1 +\mb G_2)\int_{\bbr} a^{-\mb{X}}_\xi d\x.
\]
This and the algebraic inequality $\frac{p^2}{2}-q^2 \le (p-q)^2$ for all $p,q\ge 0$ imply
\[
2(\sigma^*)^2M^2\left(\int_0^1 w dy\right)^2 - |\dot{\mb X}|^2 \le  C(\lam+\delta_0+\eps_1)^2 \int_0^1|w|^2 dy + \frac{C}{\l}  (\mb G_1 +\mb G_2).
\]
Thus,
\beq\label{gxest}
 -\frac{\deltas}{2M} |\dot{\mb X}|^2 \le -(\sigma^*)^2M\deltas\left(\int_0^1 w dy\right)^2 +C\ds(\lam+\delta_0+\eps_1)^2 \int_0^1 w^2 dy
+ \frac{C\delta_S}{\l}  (\mb G_1 +\mb G_2). 
\eeq
	\noindent$\bullet$ {\bf Estimates on $\mb B_1$:}\\
    Set
\beq\label{b1*}
\begin{aligned}
 \mb B_1(U)&=  \int_\bbr a^{-\mb X} (v^S)_\xi^{-\mb X} \Big[ \frac{(\gamma+1)p^*\sigma^*}{2(v^*)^2} (v-\bar v)^2+\frac{R\sigma^*}{2v^*\theta^*}(\theta-\bar\theta)^2-\frac{p^*\sigma^*}{v^*\theta^*}(v-\bar v)(\theta-\bar\theta)\Big]d\x\\
& := \mb B_{11}(U)+ \mb B_{12}(U)+\mb B_{13}(U).
\end{aligned}
\eeq
First, using Young's inequality, 
\begin{align*}
\mb B_{11} & = \frac{(\gamma+1)p^*\sigma^*}{2(v^*)^2}  \int_\bbr a^{-\mb X} (v^S)_\xi^{-\mb X} \Big|\Big((v-\bar v)+\frac{u-\bar u}{\sigma^*}\Big) - \frac{u-\bar u}{\sigma^*} \Big|^2 d\xi  \\
&\le  \frac{(\gamma+1)p^*\sigma^*}{2(v^*)^2} \left(1+\l+\Big(\frac{\deltas}{\l}\Big)^\frac12 \right) \int_\bbr (v^S)_\xi^{-\mb X} \Big| \frac{u-\bar u}{\sigma^*} \Big|^2 d\xi \\
&\quad + C\Big(\frac{\deltas}{\l}\Big)^{-\frac12} \int_\bbr  (v^S)_\xi^{-\mb X} \Big|(v-\bar v)+\frac{u-\bar u}{\sigma^*} \Big|^2 d\xi\\
&\le  \frac{(\gamma+1)p^*\sigma^*}{2(v^*)^2} \left(1+\l+\Big(\frac{\deltas}{\l}\Big)^\frac12 \right) \int_\bbr (v^S)_\xi \Big| \frac{u^{\mb X}-\bar u^{\mb X}}{\sigma^*} \Big|^2 d\xi \\
&\quad + C\Big(\frac{\deltas}{\l}\Big)^{-\frac12} \int_\bbr  (v^S)_\xi^{-\mb X} \Big|(v-\bar v)+\frac{u-\bar u}{\sigma^*} \Big|^2 d\xi,
\end{align*}
which together with \eqref{a-prime} and the change of variable $\xi\mapsto\x+\mb X(t)$ yields
\beq\label{b11}
\mb B_{11}  \le \frac{(\gamma+1)p^*}{2(v^*)^2\s^*} \left(1+\l+\Big(\frac{\deltas}{\l}\Big)^\frac12 \right) \deltas \int_0^1 w^2 dy + C\Big(\frac{\deltas}{\l}\Big)^{\frac12} \mb G_1.
\eeq
Similarly, we estimate (using $p^*=R\theta^*/v^*$)
\begin{align}\label{b12}
\begin{aligned}
\mb B_{12} & =\frac{R\sigma^*}{2v^*\theta^*}  \int_\bbr a^{-\mb X} (v^S)_\xi^{-\mb X}  \left| \Big( (\theta-\bar\theta)-\frac{(\gamma-1)\theta^*}{v^*\s^*}(u-\bar u) \Big) + \frac{(\gamma-1)\theta^*}{v^*\s^*}(u-\bar u) \right|^2 d\xi  \\
&\le  \frac{(\gamma-1)^2p^*}{2(v^*)^2\s^*} \left(1+\l+\Big(\frac{\deltas}{\l}\Big)^\frac12 \right) \deltas \int_0^1 w^2 dy + C\Big(\frac{\deltas}{\l}\Big)^{\frac12} \mb G_2,
\end{aligned}
\end{align}
and 
\begin{align}\label{b13}
\begin{aligned}
\mb B_{13} & \le \frac{p^*\sigma^*}{v^*\theta^*} \int_\bbr a^{-\mb X} (v^S)_\xi^{-\mb X} \Big|\Big((v-\bar v)+\frac{u-\bar u}{\sigma^*}\Big) - \frac{u-\bar u}{\sigma^*} \Big| \\
&\qquad\qquad\quad \cdot \left| \Big( (\theta-\bar\theta)-\frac{(\gamma-1)\theta^*}{v^*\s^*}(u-\bar u) \Big) + \frac{(\gamma-1)\theta^*}{v^*\s^*}(u-\bar u) \right| d\xi  \\
&\le  \frac{(\gamma-1)p^*}{(v^*)^2\s^*} \left(1+\l+\Big(\frac{\deltas}{\l}\Big)^\frac12 \right) \deltas \int_0^1 w^2 dy + C\Big(\frac{\deltas}{\l}\Big)^{\frac12} (\mb G_1 + \mb G_2).
\end{aligned}
\end{align}
Let $\alpha^*$ be the $O(1)$-constant defined by
\beq\label{alpha*}
\alpha^*:= \frac{\gamma(\gamma+1) p^*}{2(v^*)^2\s^*}.
\eeq
Substituting \eqref{b11}, \eqref{b12} and \eqref{b13} into \eqref{b1*}, we have
\beq\label{b1-est}
\mb B_1 \le \alpha^* \left(1+\l+\Big(\frac{\deltas}{\l}\Big)^\frac12 \right) \deltas \int_0^1 w^2 dy+ C\Big(\frac{\deltas}{\l}\Big)^{\frac12} (\mb G_1 + \mb G_2).
\eeq
	
	\noindent$\bullet$ {\bf Estimates on $\mb D$:} \\
	We first recall that
	\beq\label{D-*}
	\mb D(U)= \int_\bbr \!a^{-\mb X}\Big[ \frac{\kappa}{v\theta}|(\theta-\bar \theta)_\xi|^2\Big]d\x .
	\eeq
	To estimate $\mb D$, we denote
$p_+:=p(v_+,\theta_+)$ and $p^*:=p(v^*,\theta^*)$
for simplicity. Since $y=\frac{v^S-v^*}{\deltas}$ and $1-y=\frac{v_+-v^S}{\deltas}$, it follows from \eqref{sharp-D} that
	\beq\label{sharp-d}
	\left| \frac{1}{y(1-y)} \frac{\kappa}{v^S} \frac{dy}{d\xi} -\deltas \alpha^*\frac{ R\gamma}{ (\gamma-1)^2} \right|\leq C\deltas^2.
	\eeq
	
	Then, using \eqref{sharp-d}, the fact that $a\geq1$, the change of variables in \eqref{y-xi}, and the translation $\xi\mapsto\xi+\mb X(t)$, we obtain
	\beq\label{D2e}
	\begin{aligned}
		\mb D&\geq\int_\bbr \frac{\kappa}{v^{\mb X}\theta^{\mb X}}|(\theta^{\mb X}-\bar \theta^{\mb X})_\xi|^2d\x=\int_0^1\frac{\kappa}{v^{\mb X}\theta^{\mb X}}|(\theta^{\mb X}-\bar \theta^{\mb X})_y|^2 \Big(\frac{dy}{d\xi}\Big) dy\\
		& =\int_0^1\frac{ v^S}{ v^{\mb X}\theta^{\mb X}}|(\theta^{\mb X}-\bar \theta^{\mb X})_y|^2 \frac{\kappa}{v^S}\Big(\frac{dy}{d\xi}\Big)  dy\\
		& =\int_0^1\frac{1 }{\theta^*}|(\theta^{\mb X}-\bar \theta^{\mb X})_y|^2 \frac{\kappa}{v^S}\Big(\frac{dy}{d\xi}\Big)  dy+\int_0^1\left(\frac{v^S}{v^{\mb X}\theta^{\mb X}}-\frac{1}{\theta^*}\right)|(\theta^{\mb X}-\bar \theta^{\mb X})_y|^2 \frac{\kappa}{v^S}\Big(\frac{dy}{d\xi}\Big)  dy\\
		& \ge  \alpha^*\frac{R\gamma}{(\gamma-1)^2} \frac{1}{ \theta^*}\big(1-C(\delta_0+\eps_1)\big)\deltas  \int_0^1y(1-y)  |(\theta^{\mb X}-\bar \theta^{\mb X})_y|^2   dy,
	\end{aligned}
	\eeq
	together with $|v^{\mb X}-v^S|\le C(\eps_1+\delta_0)$.	
	
	The goal is to use the Poincar\'{e}-type inequality in Lemma \ref{lem-poin} to absorb the main bad term $\mb B_1$ by the diffusion term $\mb D$. However, the term $w$ does not explicitly appear in $\mb D$. Therefore, we first need to extract a good term on $w$ from $\mb D$ as follows.\\
	First, by Lemma \ref{lem-poin} together with
	\[
	\int_0^1 |w-\bar w|^2 dy = \int_0^1 w^2 dy -{\bar w}^2, \quad \bar w:=\int_0^1 w dy,
	\]
	we have
	\begin{align*}
			\mb D \ge 2 \alpha^*\frac{R\gamma}{ (\gamma-1)^2} \big(1-C(\delta_0+\eps_1)\big)  \deltas \bigg[  \frac{1}{\theta^*} \left(\int_0^1 |\theta^{\mb X}-\bar \theta^{\mb X}|^2  dy -\Big(\int_0^1(\theta^{\mb X}-\bar \theta^{\mb X}) dy\Big)^2\right) \bigg] .
	\end{align*}
	By the definition of $w$ in \eqref{omega}, Young's inequality, and the change of variables $\xi\mapsto\xi+\mb X(t)$, we obtain
	\begin{align*}
			\int_0^1 |\theta^{\mb X}-\bar \theta^{\mb X}|^2  dy & =  \int_0^1  \left| \Big( (\theta^{\mb X}-\bar\theta^{\mb X})-\frac{(\gamma-1)\theta^*}{v^*\s^*}(u^{\mb X}-\bar u^{\mb X}) \Big) + \frac{(\gamma-1)\theta^*}{v^*\s^*}(u^{\mb X}-\bar u^{\mb X}) \right|^2 dy  \\
			&\ge  \left(\frac{(\gamma-1)\theta^*}{v^*\s^*}\right)^2 \left(1-\Big(\frac{\deltas}{\l}\Big)^{\frac12} \right) \int_0^1 w^2 dy \\
			&\quad - C\Big(\frac{\deltas}{\l}\Big)^{-\frac12}  \int_0^1   \Big( (\theta^{\mb X}-\bar\theta^{\mb X})-\frac{(\gamma-1)\theta^*}{v^*\s^*}(u^{\mb X}-\bar u^{\mb X}) \Big)^2 dy,
	\end{align*}
	and 
	\begin{align*}
			\left(\int_0^1(\theta^{\mb X}-\bar \theta^{\mb X}) dy\right)^2 & \le  2\left(\int_0^1 \Big( \frac{(\gamma-1)\theta^*}{v^*\s^*}(u^{\mb X}-\bar u^{\mb X}) \Big)  dy\right)^2 \\
			&\quad +2\left(\int_0^1 \Big( (\theta^{\mb X}-\bar\theta^{\mb X})-\frac{(\gamma-1)\theta^*}{v^*\s^*}(u^{\mb X}-\bar u^{\mb X}) \Big)  dy\right)^2 \\
			&\le 2 \left(\frac{(\gamma-1)\theta^*}{v^*\s^*}\right)^2 \bar w^2 +2 \int_0^1   \Big( (\theta^{\mb X}-\bar\theta^{\mb X})-\frac{(\gamma-1)\theta^*}{v^*\s^*}(u^{\mb X}-\bar u^{\mb X}) \Big)^2 dy,
	\end{align*}
	which together with $\s^*=\frac{\sqrt{\gamma R \theta^*}}{v^*}$, \eqref{dery} and \eqref{a-prime} implies
	\begin{align}
		\begin{aligned}\label{D12}
			\mb D &\ge 2 \alpha^* (1-C(\delta_0+\eps_1))  \deltas  \left(1-\Big(\frac{\deltas}{\l}\Big)^{\frac12} \right) \int_0^1 w^2 dy \\
			&\qquad  -2\alpha^* \deltas \frac{R\gamma}{ (\gamma-1)^2} \left[\frac{2}{ \theta^*} \left(\frac{(\gamma-1)\theta^*}{v^*\s^*}\right)^2 \right] \bar w^2  - C\Big(\frac{\deltas}{\l}\Big)^{\frac12} \mb G_2\\
			&= 2 \alpha^* \left(1-C(\delta_0+\eps_1)-\Big(\frac{\deltas}{\l}\Big)^{\frac12} \right)  \deltas \int_0^1 w^2 dy \\
			&\quad -4\alpha^* \deltas  \bar w^2  - C\Big(\frac{\deltas}{\l}\Big)^{\frac12} \mb G_2.
		\end{aligned}
	\end{align}
	
	\noindent$\bullet$ {\bf Conclusion:}\\
	Combining \eqref{b2g}, \eqref{bnew}, \eqref{gxest}, \eqref{b1-est}, and \eqref{D12}, and using the smallness of $\frac{\deltas}{\l}$, $\l$ (as in \eqref{lamsmall}), $\delta_0$, and $\eps_1$, we obtain
	\begin{align*}
			-\frac{\deltas}{2M} |\dot{\mb X}|^2 + \mb B_1+\mb B_2-\mb{G} -\frac{3}{4}\mb D & \le C (\l+\eps_1+\delta_0) \mathcal{G}^S +C \deltas^{4/3} ( \deltar^{4/3}+ \delta_C^{4/3}) e^{-C\deltas t}  \\
			&\quad  -\frac{\alpha^*}{4} \deltas \int_0^1 w^2 dy- \frac{1}{2} (\mb G_1+ \mb G_2)+C (\eps_1+\delta_0) \\
            &\quad \cdot \| (u-\bar u)_\x\|^2_{L^2(\bbr)}
			  -(\sigma^*)^2M\deltas \bar w^2   + 3\alpha^* \deltas  \bar w^2.
	\end{align*}
	Choosing the constant $M$ in \eqref{X(t)} as
\[
M=\frac{3\alpha^*}{(\sigma^*)^2},
\]
we obtain
	\begin{align*}
			-\frac{\deltas}{2M} |\dot{\mb X}|^2 + \mb B_1+\mb B_2-\mb{G} -\frac{3}{4}\mb D & \le C (\l+\eps_1+\delta_0) \mathcal{G}^S +C \deltas^{4/3} ( \deltar^{4/3}+ \delta_C^{4/3}) e^{-C\deltas t}  \\
			&\quad +C (\eps_1+\delta_0) \| (u-\bar u)_\x\|^2_{L^2(\bbr)} -\frac{\alpha^*}{4} \deltas \int_0^1 w^2 dy\\
            &\quad - \frac{1}{2} (\mb G_1+ \mb G_2) .
	\end{align*}
	Finally, using
	\[
	\deltas \int_0^1 w^2 dy = \int_\bbr |v^S_\x| |u^{\mb X}-\bar u^{\mb X}|^2 d\x=\int_\bbr |(v^S)^{-\mb X}_\x| |u-\bar u|^2 d\x,
	\]
	and
	\begin{align*}
			\int_\bbr |(v^S)^{-\mb X}_\x| (v-\bar v)^2 d\xi &\leq 2\int_\bbr |(v^S)^{-\mb X}_\x| \Big|(v-\bar v)+\frac{u-\bar u}{\sigma^*}\Big|^2 d\xi + 2\int_\bbr |(v^S)^{-\mb X}_\x| \Big(\frac{u-\bar u}{\sigma^*}\Big) ^2d\xi \\
			&\leq C\frac{\deltas}{\l}\mb G_1 + C \int_\bbr   |(v^S)^{-\mb X}_\x|  \big(u-\bar u\big)^2d\xi,
	\end{align*}
	and
	\begin{align*}
			\int_\bbr|(v^S)^{-\mb X}_\x| (\theta-\bar \theta)^2 d\xi &\leq 2\int_\bbr  |(v^S)^{-\mb X}_\x| \Big|(\theta-\bar\theta)-\frac{(\gamma-1)\theta^*}{v^*\s^*}(u-\bar u)\Big|^2 d\xi \\
            &\quad+ C\int_\bbr   |(v^S)^{-\mb X}_\x| (u-\bar u)^2d\xi\\
			&\leq C\frac{\deltas}{\l}\mb G_2 + C \int_\bbr   |(v^S)^{-\mb X}_\x|  \big(u-\bar u\big)^2d\xi,
	\end{align*}
	we have
    \begin{equation*}
        \begin{aligned}
        &-\frac{\deltas}{2M} |\dot{\mb X}|^2 + \mb B_1+\mb B_2-\mb{G} -\frac{3}{4}\mb D \\
        &\le - C\int_\bbr|(v^S)^{-\mb X}_\x| |(v-\bar v, u-\bar u, \theta-\bar \theta)|^2 d\x+C (\eps_1+\delta_0) \| (u-\bar u)_\x\|^2_{L^2(\bbr)} \\
        &\quad+C \deltas^{4/3} ( \deltar^{4/3}+ \delta_C^{4/3}) e^{-C\deltas t},
        \end{aligned}
    \end{equation*}
	which proves the desired estimate $\eqref{eq-sharp1}$.
	
\end{proof}

\subsection{Proof of Lemma \ref{lem-zvh}}
In order to prove Lemma \ref{lem-zvh}, we need to estimate the right-hand side of \eqref{ineq-2}. Since the leading-order terms have already been estimated in Lemma \ref{lem-sharp1}, we only need to handle the remaining terms in \eqref{ineq-2} as follows.
\begin{lemma}\label{lem-sharp2}
	There exists $C>0$ such that 
	\begin{align}
		\begin{aligned}\label{eq-sharp2}
		&\frac{C}{\deltas}\sum_{i=4}^6 |\mb{Y}_i|^2  +\sum_{i=3}^6 \mb B_i +\mb S_1+\mb S_2   -\frac{1}{8} \mb D\\
		&\leq -\frac{C}{2}\mathcal{G}^R + C(\delta_0+\eps_1)  \mathcal{G}^S + C \dc(1+t)^{-1} \int_\bbr  e^{-\frac{\tilde C |\xi+\sigma t|^2}{1+t}} |(v-\bar v,\theta-\bar\theta)|^2 d\xi\\
		&\quad +C\big[\deltas (\deltar+\delta_C) e^{-C \deltas t}+\deltar\delta_Ce^{-Ct}+\delta_C (1+t)^{-\frac 54}\big]\|(u-\bar u,\theta-\bar \theta)\|_{L^2}\\
		&\quad  +C\| Q_2^R\|_{L^1}^{\frac 43}\| \theta-\bar\theta\|_{L^2}^{\frac 23},
		\end{aligned}
	\end{align}
	where the good terms $ \mathcal{G}^R(U),  \mathcal{G}^S(U)$ are as in \eqref{maingood}.
\end{lemma}
\begin{proof}
    In the proof, we estimate each term on the left-hand side of \eqref{eq-sharp2}. The estimates for $\mb B_i$ $(i=3,4,5,6)$, $\mb Y_i$ $(i=4,5,6)$, and $\mb S_i$ $(i=1,2)$ are given as follows.
			
	\noindent$\bullet$ {\bf Estimates on the terms $\mb B_i~(i=3,4,5,6)$:}\\
	First, from the definition of $\mb B_3$  and the definition of $a^{-\mb X}$ in \eqref{weight}, we have
			\begin{align*}
					\mb B_3(U) 
					&\leq C \int_\bbr \! |a_\xi^{-\mb X} | |\theta-\bar \theta| \Big(  |(\theta-\bar \theta)_\xi| + |\bar \theta_\xi| |v-\bar v| \Big) d\xi \\
					&\le \frac{1}{80} \mb D +  C\lambda\deltas \int_\bbr \! |(v^S)_\xi^{-\mb X} | |\theta-\bar \theta |^2 d\xi  + C \int_\bbr  |\bar \theta_\xi|^2  |v-\bar v|^2 d\xi \\
					&\le  \frac{1}{80}  \mb D + C\deltas \mathcal{G}^S  + C \int_\bbr  |\bar \theta_\xi|^2  |v-\bar v|^2 d\xi.
			\end{align*}
			Since Lemma \ref{lemma1.2}, \eqref{vc-pe} and \eqref{shock-base} yield
			\begin{align}
				\begin{aligned} \label{b3model}
					& \int_\bbr  |\bar \theta_\xi|^2  |v-\bar v|^2 d\xi \\
					&\quad\le C\int_\bbr \Big[  |(\theta^R)_\xi|^2+|(\theta^C)_\xi|^2+|(\theta^S)^{-\mb X}_\xi|^2\Big] |v-\bar v|^2 \\
					&\quad\le C\deltar \mathcal{G}^R+C\deltas \mathcal{G}^S+ C(\delta_C)^2 (1+t)^{-1}\int_\bbr  e^{-\frac{2 \tilde C|\xi+\sigma t|^2}{1+t}}|v-\bar v|^2 d\xi ,
				\end{aligned}
			\end{align}
			we have
			\[
			\mb B_3(U)  \le  \frac{1}{40}  \mb D + C\deltar \mathcal{G}^R+C\deltas \mathcal{G}^S + C(\delta_C)^2 (1+t)^{-1}\int_\bbr  e^{-\frac{2 \tilde C|\xi+\sigma t|^2}{1+t}}|v-\bar v|^2 d\xi .
			\]
			Likewise, we have
			\begin{align*}
					\mb B_4 &\le C \int_\bbr \Big[  |\theta-\bar \theta| \Big(  |(\theta-\bar \theta)_\xi| + |\bar \theta_\xi| \Big) \Big(  |(\theta-\bar \theta)_\xi| + |\bar \theta_\xi| |v-\bar v| \Big) \Big]d\xi \\
					&\quad + C \int_\bbr \Big( |(\theta-\bar \theta)_\xi| |\bar \theta_\xi| |v-\bar v|  \Big) d\xi 
					+C \int_\bbr |\theta-\bar\theta|^2\Big(  |\bar \theta_{\xi\xi}|+ |(\bar v_\xi,\bar \theta_\xi)|^2  \Big)  d\xi ,
			\end{align*}
			which together with \eqref{smp1} yields
			\begin{align*}
				\begin{aligned}
					\mb B_4
					&\le \frac{1}{80} \mb D  + C \int_\bbr   |\bar \theta_\xi|^2  |(v-\bar v,\theta-\bar\theta)|^2 d\xi \\
					&\le   \frac{1}{80}  \mb D + C\deltar \mathcal{G}^R+C\deltas \mathcal{G}^S  + C(\delta_C)^2 (1+t)^{-1}\int_\bbr  e^{-\frac{2 \tilde C |\xi+\sigma t|^2}{1+t}}  |(v-\bar v,\theta-\bar\theta)|^2 d\xi.
				\end{aligned}
			\end{align*}
    Then, we can estimate $\mb B_5(U)$ as follows. Set 
    \begin{equation*}
        \begin{split}
            J_2 :=  R\left((\theta^R_t-\s \theta^R_\x)+(\theta^C_t-\s \theta^C_\x)- \sigma (\theta^S)^{-\mathbf{X}} _{\xi}\right)\Phi\left( \frac{v}{\bar{v}} \right) - \frac{\bar{p}\bar u_\x}{v \bar{v}} (v - \bar{v})^2.
        \end{split}
    \end{equation*}
    Since
\[
\bar u_\xi=u_\xi^R+u_\xi^C+u_\xi^S,
\]
and by $\eqref{rarexi}_3$, we have
\[
\theta_t^R-\sigma\theta_\x^R
=-\frac{\gamma-1}{R}p^R u_\xi^R.
\]
Substituting this relation into $J_2$, we obtain
\begin{align*}
J_2 & =\left[-u^R_\xi\Big((\gamma-1)p^R\Phi(\frac{v}{\bar v})+\frac{\bar p(v-\bar v)^2}{v\bar v}\Big)\right]+\left[R(\theta_t^C-\sigma \theta_\x^C)\Phi(\frac{v}{\bar v})-\frac{\bar p u^C_\xi}{v\bar v}(v-\bar v)^2\right] \\
&\quad \ +\left[-R\sigma (\theta^S)^{-\mb X}_\xi \Phi(\frac{v}{\bar v})-\frac{(u^S)^{-\mb X}_\xi \bar p }{v\bar v}(v-\bar v)^2\right]:=J_{21}+J_{22}+J_{23}.
\end{align*}
Using the fact that $|v-\bar v|\le C\eps_1$ and the properties
\[
\Phi(1)=\Phi'(1)=0,\quad \Phi''(1)=1,
\]
we obtain the expansion
\beq\label{app-phi}
\Phi(\frac{v}{\bar v}) =  \frac{(v-\bar v)^2}{2\bar v^2} + O(|v-\bar v|^3).
\eeq
Moreover,
\begin{align*}
\begin{aligned}
|\bar p - p^R| & \le C( |\bar v - v^R| +|\bar \theta - \theta^R| ) \\
&\le C(|v^S - v^*| +|v^C - v_*| +|\theta^S - \theta^*| +|\theta^C - \theta_*| )  \le C(\deltas +\delta_C), 
\end{aligned}
\end{align*}
which, together with the above expansion, gives
\begin{align*}
J_{21} & \le -\frac{(\gamma+1)\bar p}{2\bar v^2}u^R_\xi (v-\bar v)^2 + C(\delta_0 +\eps_1 ) |u^R_\xi| |v-\bar v|^2.
\end{align*}
Using $\eqref{vcex}_3$, we have
\[
J_{22} \le C( |u^C_\xi|+ |\theta^C_{\xi\xi}| + |\theta^C_\xi||v^C_\xi| ) (v-\bar v)^2,
\]
which, together with \eqref{vc-pe} and \eqref{QC1}, yields
\[
J_{22} \le C \dc(1+t)^{-1}e^{-\frac{\tilde C |\xi+\sigma t|^2}{1+t}}(v-\bar v)^2.
\]
Using \eqref{shock-vu}, \eqref{theta-s}, \eqref{sm1} and 
\begin{align} 
\begin{aligned}\label{perror}
|\bar p - p^*| & \le C( |\bar v - v^*| +|\bar \theta - \theta^*| ) \\
&\le C(|v^R - v_*| + |v^S - v^*| +|v^C - v^*| +|\theta^R - \theta_*|+|\theta^S - \theta^*| +|\theta^C - \theta^*| ) \\
& \le C(\deltar+\deltas +\delta_C) \le C\delta_0, 
\end{aligned}
\end{align}
we have
\[
J_{23} \le \frac{(\gamma+1)p^*\sigma^*}{2(v^*)^2}(v^S)^{-\mb X}_\xi (v-\bar v)^2+ C(\delta_0 +\eps_1 ) |(v^S)^{-\mb X}_\xi | |v-\bar v|^2.
\]
Thus, we get 
  \begin{equation*}
            \begin{aligned}
         J_{2} &\leq \frac{(\gamma+1)p^*\sigma^*}{2(v^*)^2}(v^S)^{-\mb X}_\xi (v-\bar v)^2+ C(\delta_0 +\eps_1 ) |(v^S)^{-\mb X}_\xi | |v-\bar v|^2\\
         &\quad -\frac{(\gamma+1)\bar p}{2\bar v^2}u^R_\xi (v-\bar v)^2 + C(\delta_0 +\eps_1 ) |u^R_\xi| |v-\bar v|^2+C \dc(1+t)^{-1}e^{-\frac{\tilde C |\xi+\sigma t|^2}{1+t}}(v-\bar v)^2,
          \end{aligned}		
        \end{equation*}
        which means
    \begin{equation}\label{J2}
        \begin{aligned}
        \mb B_5=&\int_{\mathbb{R}} a^{-\mathbf{X}}J_{2} d \xi-\int_{\mathbb{R}} a^{-\mathbf{X}} \frac{(\gamma+1)p^*\sigma^*}{2(v^*)^2}(v^S)_\xi^{\mathbf{-\mathbf{X}}} (v - \bar{v})^2 d\x \\
         \leq
        &C\left(\delta_0+\varepsilon_{1}\right) \int_{\mathbb{R}} a^{-\mathbf{X}}\left( |(v^{S})_{\xi}^{-\mathbf{X}}|+|u^R_\xi|  \right)|(v-\bar{v}, \theta-\bar{\theta})|^{2} d \xi\\
        &- \int_{\mathbb{R}} a^{-\mathbf{X}} \frac{(\gamma+1)\bar p}{2\bar v^2}u^R_\xi (v-\bar v)^2 d \xi+C\dc \int_{\mathbb{R}} a^{-\mathbf{X}} (1+t)^{-1}e^{-\frac{\tilde C |\xi+\sigma t|^2}{1+t}}|v-\bar v|^2 d \xi.
    \end{aligned}
 \end{equation}
   Similarly, for $\mb B_6$, we set 
   $$
   J_3:=\Big[ \frac{R}{\gamma-1}\Big((\theta_t^R-\sigma \theta_\x^R)+(\theta_t^C-\sigma \theta_\x^C)-\sigma (\theta^S)^{-\mb X}_\xi \Big) \Phi\left(\frac{\theta}{\bar \theta}\right) - \frac{\bar u_\xi}{\theta}(\theta-\bar \theta)(p-\bar p) +\bar p \bar u_\xi \frac{(\theta-\bar\theta)^2}{\theta\bar\theta}  \Big].
   $$  
   Then, we obtain
    \begin{equation}\label{J3}
    \begin{aligned}
    \mb B_6&=\int_{\mathbb{R}} a^{-\mathbf{X}}J_{3} d \xi -\int_{\mathbb{R}} a^{-\mathbf{X}} \frac{\sigma^* (v^S)^{-\mb X}_\xi}{2v^*\theta^*}\left(R(\theta-\bar{\theta})-2 p^*(v-\bar{v})\right)(\theta-\bar{\theta}) d\x \\
    &\leq C\left(\delta_0+\varepsilon_{1}\right) \int_{\mathbb{R}} a^{-\mathbf{X}}\left( |(v^{S})_{\xi}^{-\mathbf{X}}|+|u^R_\xi|  \right)|(v-\bar{v}, \theta-\bar{\theta})|^{2} d \xi\\
    &\quad +C \int_{\mathbb{R}} a^{-\mathbf{X}}u^R_\x \Big[ -\frac{R}{2\bar v\bar\theta}(\theta-\bar \theta) + \frac{\bar p}{\bar v \bar \theta} (v-\bar v)\Big] (\theta-\bar \theta) d \xi\\
    &\quad +C\dc \int_{\mathbb{R}} a^{-\mathbf{X}} (1+t)^{-1}e^{-\frac{\tilde C |\xi+\sigma t|^2}{1+t}}|(v-\bar v,\theta-\bar\theta)|^2 d \xi.
    \end{aligned}
    \end{equation}
    Therefore, combining the estimates \eqref{J2} and \eqref{J3} together with $\bar p=\frac{R\bar\theta}{\bar v}$, we obtain
    \begin{align*}
        \mb B_5(U)+\mb B_6(U) &=
        \int_{\mathbb{R}} a^{-\mathbf{X}}\left(J_{2}+J_{3}\right) d \xi -\mb B_1(U)\\ 
        &\leq C\left(\delta_0+\varepsilon_{1}\right) \int_{\mathbb{R}} a^{-\mathbf{X}}\left( |(v^{S})_{\xi}^{-\mathbf{X}}|+|u^R_\xi|  \right)|(v-\bar{v}, \theta-\bar{\theta})|^{2} d \xi\\
        &\quad+C\dc \int_{\mathbb{R}} a^{-\mathbf{X}} (1+t)^{-1}e^{-\frac{\tilde C |\xi+\sigma t|^2}{1+t}}|(v-\bar v,\theta-\bar\theta)|^2 d \xi\\
        &\quad -  \int_\bbr \!a^{-\mb X}   \frac{R u^R_\xi }{2\bar v^2}\Big[\frac{\gamma \bar\theta}{\bar v}(v-\bar v)^2+\Big(\sqrt{\frac{\bar\theta}{\bar v}}(v-\bar v) - \sqrt{\frac{\bar v}{\bar\theta}}(\theta-\bar\theta)\Big)^2\Big] d\xi\\
        &:= C\left(\delta_0+\varepsilon_{1}\right) \int_{\mathbb{R}} a^{-\mathbf{X}}\left( |(v^{S})_{\xi}^{-\mathbf{X}}|+|u^R_\xi|  \right)|(v-\bar{v}, \theta-\bar{\theta})|^{2} d \xi \\
        &\quad+C\dc \int_{\mathbb{R}} a^{-\mathbf{X}} (1+t)^{-1}e^{-\frac{\tilde C |\xi+\sigma t|^2}{1+t}}|(v-\bar v,\theta-\bar\theta)|^2 d \xi+ Z.
    \end{align*}
    We now derive the simpler form $\mathcal{G}^R$ from the above good term $Z$. Using
        \[
        2(v-\bar v)(\theta-\bar\theta) \le  \frac{ 2\bar\theta}{\bar v} (v-\bar v)^2 + \frac{\bar v}{2\bar\theta}(\theta-\bar\theta)^2,
        \]
    and $u^R_\xi \sim v^R_\xi >0$ from Lemma \ref{lemma1.2} (1), we have
    \[
     Z  \leq \int_\bbr \!a^{-\mb X}   \frac{R u^R_\xi }{2\bar v^2}\Big[\frac{(\gamma-1) \bar\theta}{\bar v}(v-\bar v)^2 + \frac{\bar v}{2\bar\theta}(\theta-\bar\theta)^2 \Big] d\xi  \leq -C \mathcal{G}^R(U).
    \]
     Using the smallness of $\delta_1$ and $\eps_1$, we have
         \begin{equation*}
            \mb B_5 + \mb B_6 \le   -\frac{2C}{3}\mathcal{G}^R + C(\delta_0+\eps_1)  \mathcal{G}^S + C \dc(1+t)^{-1} \int_\bbr  e^{-\frac{\tilde C |\xi+\sigma t|^2}{1+t}} |(v-\bar v,\theta-\bar\theta)|^2 d\xi.
         \end{equation*}
			Therefore, combining the above estimates, we obtain
			\beq\label{bis}
            \begin{aligned}
			\sum_{i=3}^6 \mb B_i  &\le  \frac{1}{20}\mb D + C(\delta_0+\eps_1) \mathcal{G}^S-\frac{C}{2}\mathcal{G}^R   \\
            &\quad + C\delta_C (1+t)^{-1}\int_\bbr  e^{-\frac{\tilde C|\xi+\sigma t|^2}{1+t}} |(v-\bar v,\theta-\bar\theta)|^2 d\xi.
            \end{aligned}
			\eeq

    \noindent$\bullet$ {\bf Estimates on the terms $\mb Y_i~(i=4, 5, 6)$}\\
       Using \eqref{theta-s}, we have
\[
\di |(\mb Y_4, \mb Y_5)|\le C\int_\bbr |(v^S)^{-\mb{X}}_\xi| |(v-\bar v , \theta-\bar\theta)|^2d\xi.
\]
In addition, since Lemma \ref{lemma1.3} and \eqref{apri-ass} yield
\[
 |(\mb Y_4, \mb Y_5)|\le C\deltas^2 \int_\bbr |(v-\bar v , \theta-\bar\theta)|^2d\xi \le C\deltas^2\eps_1^2,
\]
we have
\[
\frac{C}{\deltas}|(\mb{Y}_4, \mb Y_5) |^2 \le C\deltas \eps_1^2\mathcal{G}^S.
\]
Similarly, we have
\begin{align*}
 \frac{C}{\deltas}|\mb{Y}_6|^2 &\le \frac{C}{\deltas} \left(\int_\bbr  |a^{-\mb{X}}_\xi| |(v-\bar v, u-\bar u, \theta-\bar\theta)|^2 d\xi\right)^2\\ 
&\le \frac{C\lam^2}{\deltas^3} \left(\int_\bbr |(v^S)_\xi|  |(v-\bar v, u-\bar u, \theta-\bar\theta)|^2 d\xi\right)^2\\
&\le \frac{C\lam^2}{\deltas} \|(v-\bar v, u-\bar u, \theta-\bar\theta)\|_{L^2(\bbr)}^2 \int_\bbr |(v^S)^{-\mb{X}}_\xi| |(v-\bar v, u-\bar u, \theta-\bar\theta)|^2d\xi\\
&\le C\eps_1^2  \int_\bbr |(v^S)^{-\mb{X}}_\xi| |(v-\bar v, u-\bar u, \theta-\bar\theta)|^2 d\xi\le C\eps_1^2 \mathcal{G}^S.
\end{align*}
Thus,
\beq\label{yis}
\frac{C}{\deltas}\sum_{i=4}^6 |\mb{Y}_i|^2 \le  C \eps_1\mathcal{G}^S.
\eeq

\noindent$\bullet$ {\bf Estimates on the terms $\mb S_i~(i=1,2)$:}\\
First, by \eqref{QR}, for $Q_2^R$, we have
\beq\label{Q-est}
\begin{aligned}
|Q^R_2| &\le C \Big[|\theta^R_{\xi\xi}|+|\theta^R_{\xi}||v^R_{\xi}|\Big]\\
& \le C  \Big[\big|\theta^R_{\xi\xi}\big|+\big|\big(v^R_{\xi},\theta^R_{\xi}\big)\big|^2\Big],
\end{aligned}
\eeq
and then
$$
\|Q^R_2\|_{L^1(\bbr)}\leq C\Big[\big\|\theta^R_{\xi\xi}\big\|_{L^1}+\big\|\big(v^R_{\xi}, \theta^R_{\xi}\big)\big\|_{L^2}^2\Big].
$$
Then, by Lemma \ref{lemma1.2}, we have
\[
\big\|\theta^R_{\xi\xi}\big\|_{L^1}
\leq
\begin{cases}
C\deltar, &\quad \mbox{if } 1+t\leq\deltar^{-1},\\
C\frac{1}{1+t}, &\quad \mbox{if } 1+t\geq\deltar^{-1},
\end{cases}
\]
and
\beq\label{dr2}
\big\|\big(v^R_{\xi},\theta^R_{\xi}\big)\big\|_{L^2}
\leq
\begin{cases}
C\deltar, &\quad \mbox{if } 1+t\leq\deltar^{-1},\\
C\deltar^{1/2}\frac{1}{(1+t)^{1/2}},
&\quad \mbox{if } 1+t\geq\deltar^{-1}.
\end{cases}
\eeq
Therefore,
\beq\label{v21}
\int_0^\infty \|Q_2^R\|_{L^1}^{4/3}dt \le C\big(\deltar^{1/3}+\deltar^{5/3}\big)\le C\deltar^{1/3}.
\eeq
By \eqref{QC1}, it holds that
\beq\label{vc-21}
\|Q_1^C\|_{L^2}\le C \delta_C (1+t)^{-\frac 54}.
\eeq
Therefore, using Young's inequality, we obtain
\begin{align*}
&\mb S_1+\mb S_2=-\int_\bbr a^{-\mb X}Q_1(u-\bar u)  d\x-\int_\bbr \!a^{-\mb X} Q_2\Big[(1-\frac{\bar\theta}{\theta}) +\frac{(\theta-\bar\theta)^{2}}{\theta\bar\theta}\Big]d\xi \\
&\leq C\int_\bbr \big(|Q^I_1|+|Q_1^C|\big)|u-\bar u| d\x+ C\int_\bbr  \big(|Q^I_2|+|Q_2^R|\big)|\theta-\bar \theta|d\x\\
&\le C \big(\|Q_1^I\|_{L^2}+\|Q_1^C\|_{L^2}\big) \|u-\bar u\|_{L^2}+ C \|Q_2^I\|_{L^2} \|\theta-\bar \theta\|_{L^2}+C\|Q_2^R\|_{L^1(\bbr)} \|\theta-\bar \theta\|_{L^\infty(\bbr)} ,
\end{align*}
which, together with Lemma \ref{lemma2.2} and \eqref{vc-21}, yields
\begin{align}\label{S12-e}
\begin{aligned}
&\mb S_1+\mb S_2\\
&\leq C\big[\deltas (\deltar+\delta_C) e^{-C \deltas t}+\deltar\delta_Ce^{-Ct}+\delta_C (1+t)^{-\frac 54}\big]\big(\|u-\bar u\|_{L^2}+\|\theta-\bar \theta\|_{L^2}\big)\\
&\quad + C\| Q_2^R\|_{L^1}\|(\theta-\bar\theta)_\x\|_{L^2}^{\frac 12}\| \theta-\bar\theta\|_{L^2}^{\frac 12}\\
&\leq C\big[\deltas (\deltar+\delta_C) e^{-C \deltas t}+\deltar\delta_Ce^{-Ct}+\delta_C (1+t)^{-\frac 54}\big]\|(u-\bar u,\theta-\bar \theta)\|_{L^2}\\
&\quad +\frac{1}{20} \mb D  +C\| Q_2^R\|_{L^1}^{\frac 43}\| \theta-\bar\theta\|_{L^2}^{\frac 23}.
\end{aligned}
\end{align}
     \noindent$\bullet$ {\bf Conclusion:}\\
			Using $\eqref{yis}$, $\eqref{S12-e}$ and $\eqref{bis}$, we have
				\begin{align*}
					&\frac{C}{\deltas}\sum_{i=4}^6 |\mb{Y}_i|^2  +\sum_{i=3}^6 \mb B_i +\mb S_1+\mb S_2   -\frac{1}{8} \mb D\\
					&\leq  \frac{1}{20}\mb D + C(\delta_0+\eps_1) (\mathcal{G}^R+ \mathcal{G}^S )  + C\delta_C (1+t)^{-1}\int_\bbr  e^{-\frac{\tilde C|\xi+\sigma t|^2}{1+t}} |(v-\bar v,\theta-\bar\theta)|^2 d\xi\\
					&\quad -\frac{C}{2}\mathcal{G}^R 
					  +C\big[\deltas (\deltar+\delta_C) e^{-C \deltas t}+\deltar\delta_Ce^{-Ct}+\delta_C (1+t)^{-\frac 54}\big]\|(u-\bar u,\theta-\bar \theta)\|_{L^2}\\
					&\quad +\frac{1}{20} \mb D  +C\|Q_2^R\|_{L^1}^{\frac 43}\|(u-\bar u, \theta-\bar\theta)\|_{L^2}^{\frac 23}-\frac{1}{8} \mb D\\
					&\leq -\frac{C}{2}\mathcal{G}^R + C(\delta_0+\eps_1)  \mathcal{G}^S + C \dc(1+t)^{-1} \int_\bbr  e^{-\frac{\tilde C|\xi+\sigma t|^2}{1+t}} |(v-\bar v,\theta-\bar\theta)|^2 d\xi\\
					&\quad +C\big[\deltas (\deltar+\delta_C) e^{-C \deltas t}+\deltar\delta_Ce^{-Ct}+\delta_C (1+t)^{-\frac 54}\big]\|(u-\bar u,\theta-\bar \theta)\|_{L^2}\\
					&\quad  +C\|Q_2^R\|_{L^1}^{\frac 43}\|\theta-\bar\theta\|_{L^2}^{\frac 23},
			\end{align*}
			with the smallness of $\delta_1$ and $\eps_1$. Therefore, we complete the proof of $\eqref{eq-sharp2}$.
\end{proof}

Before proceeding with the formal proof, we also need to establish an estimate to control the stability of the viscous contact wave. Since the proof of Lemma \ref{cw-lemma} follows the same argument as in \cite{KVW25} that handles the stability of the viscous contact wave with shift, we omit the detailed proof here.

\begin{lemma}\label{cw-lemma}
	It holds that
	\beq\label{cwe}
	\begin{aligned}
		&\int_0^t(1+\tau)^{-1}\int_\bbr  e^{-\frac{\tilde  C|\xi+\sigma \tau|^2}{1+\tau}}|(v-\bar v, u-\bar u, \theta-\bar \theta)|^2 d\xi d\tau \\
		 & \le  C\sup_{t\in[0,T]}\|U-\bar U\|^2_{L^2(\bbr)}  +C\deltas\int_0^t|\dot{\mb X}(\tau)|^2 d\tau+C\int_0^t (\mathcal{G}^S+\mathcal{G}^R+ \mb D) d\tau  \\
		  &\quad+ C\int_0^t\Big(\|(u-\bar u)_\xi \|^2_{L^2(\bbr)} +\|(v-\bar v)_\xi \|^2_{L^2(\bbr)} \Big)d\tau +C\delta_0^{\frac13}.
	\end{aligned}
	\eeq
\end{lemma}

\begin{remark}
	In the proof of Lemma \ref{lem-zvh}, Lemma \ref{cw-lemma} was used  to control the term involving only the $v$ and $\theta$ variables. However, Lemma  \ref{cw-lemma} also provides the estimate on the $u$ variable: $\di \int_0^t(1+\tau)^{-1}\int_\bbr  e^{-\frac{\tilde C|\xi+\sigma \tau|^2}{1+\tau}}|u-\bar u|^2 d\xi d\tau$, which will be used in Section \ref{sec-vu}.
\end{remark}

Then, using Lemma \ref{lem-sharp1} and Lemma \ref{lem-sharp2}, with the smallness of $\delta_1, \eps_1$,  there exists a constant $C>0$ such that
\begin{align*}
		&\frac{d}{dt}\int_{\bbr} a^{-\mb X}\bar \theta \eta\big(U|\bar U \big) d\xi  \\
		&\le -\frac{\deltas}{4M} |\dot{\mb X}|^2  -\frac{C}{2}(\mathcal{G}^S+\mathcal{G}^R) -\frac{1}{8} \mb D  + C ( \deltar^{4/3}+ \delta_C^{4/3})  \deltas^{4/3} e^{-C\deltas t}   \\
		&\quad + C\delta_C (1+t)^{-1}\int_\bbr  e^{-\frac{\tilde C|\xi+\sigma t|^2}{1+t}} |(v-\bar v,\theta-\bar\theta)|^2 d\xi +C (\eps_1+\delta_0) \| (u-\bar u)_\x\|^2_{L^2(\bbr)}\\
		&\quad + C\big[\deltas (\deltar+\delta_C) e^{-C \deltas t}+\deltar\delta_Ce^{-Ct}+\delta_C (1+t)^{-\frac 54}\big] \|(u-\bar u,\theta-\bar \theta)\|_{L^\infty(0,T;L^2(\bbr))}  \\
		&\quad   +C\|Q_2^R\|_{L^1}^{\frac 43}  \|\theta-\bar \theta\|_{L^\infty(0,T;L^2(\bbr))}^{\frac 23}.
\end{align*}
Integrating the above inequality over $[0,t]$ for any $t\le T$, we have
\begin{align*}
		&\sup_{t\in[0,T]}\int_{\bbr}  \eta\big(U|\bar U \big) d\xi +\deltas\int_0^t|\dot{\mb{X}}|^2 d\tau +\int_0^t (\mathcal{G}^S+\mathcal{G}^R+ \mb D) d\tau\\
		&\quad\le C\int_{\bbr} \eta\big(U_0|\bar U(0,\xi) \big) d\xi  + C ( \deltar^{4/3}+ \delta_C^{4/3})  \deltas^{1/3} + C (\eps_1+\delta_0)\int_0^t \| (u-\bar u)_\x\|^2_{L^2(\bbr)} d\tau\\
		&\qquad   +C(\deltar+\delta_C) \|(u-\bar u,\theta-\bar \theta)\|_{L^\infty(0,T;L^2(\bbr))}+C\deltar^{1/3}  \|\theta-\bar \theta\|_{L^\infty(0,T;L^2(\bbr))}^{\frac 23}    \\
		&\qquad  +C\delta_C \int_0^t(1+\tau)^{-1}\int_\bbr  e^{-\frac{\tilde C|\xi+\sigma \tau|^2}{1+\tau}}|(v-\bar v,\theta-\bar \theta)|^2 d\xi d\tau .
\end{align*}
Then, using Young's inequality with the fact that
$$
\|U-\bar U\|^2_{L^2(\bbr)} \sim \int_{\bbr}  \eta\big(U|\bar U \big) d\xi, \quad \forall t\in[0,T],
$$
we have
\begin{align}
	\begin{aligned} \label{last-ecom}
		&\sup_{t\in[0,T]}\|U(t,\cdot)-\bar U(t,\cdot)\|^2_{L^2(\bbr)} +\deltas\int_0^t|\dot{\mb{X}}|^2 d\tau +\int_0^t (\mathcal{G}^S+\mathcal{G}^R+ \mb D) d\tau\\
		&\quad\le \|U_0(\cdot)-\bar U(0,\cdot)\|^2_{L^2(\bbr)} + C\delta_0^{1/2}+C (\eps_1+\delta_0)\int_0^t \| (u-\bar u)_\x(\tau,\cdot)\|^2_{L^2(\bbr)} d\tau \\
		&\qquad  + C\delta_0 \int_0^t(1+\tau)^{-1}\int_\bbr  e^{-\frac{\tilde C|\xi+\sigma \tau|^2}{1+\tau}}|(v-\bar v,\theta-\bar \theta)|^2 d\xi d\tau .
	\end{aligned} 
\end{align}
Finally, using Lemma \ref{cw-lemma} with $\delta_1\ll1$, we have
\begin{align*}
		&\sup_{t\in[0,T]}\|U(t,\cdot)-\bar U(t,\cdot)\|^2_{L^2(\bbr)} +\deltas\int_0^t|\dot{\mb{X}}|^2 d\tau +\int_0^t (\mathcal{G}^S+\mathcal{G}^R+ \mb D) d\tau\\
		&\quad\le \|U_0(\cdot)-\bar U(0,\cdot)\|^2_{L^2(\bbr)} + C\delta_0^{1/2} +C(\eps_1+\delta_0)\int_0^t\|(v-\bar v, u-\bar u)_\xi(\tau,\cdot)\|^2_{L^2(\bbr)}  d\tau  ,
\end{align*}
which completes the proof of Lemma \ref{lem-zvh}.

\section{Proof of Proposition \ref{prop2} and bootstrap argument}\label{sec-vu}
\setcounter{equation}{0}
 In this section, we complete the proof of Proposition \ref{prop2} and close the bootstrap argument. Before proving Proposition \ref{prop2}, we first introduce the following notations
\begin{equation}\label{per}
    \phi(t,\xi):=v(t,\xi)-\bar v(t,\xi),\quad  \psi (t,\xi):=u(t,\xi)-\bar u(t,\xi), \quad \vartheta(t,\xi) :=\theta(t,\x)-\bar \theta(t,\xi),
\end{equation}
and 
\[
W(t,\x):=(1+t)^{-\frac12} e^{-\frac{\beta|\xi+\s t|^2}{1+t}}, \quad \beta :=\frac{\tilde C}{2}.
\]
Then, the perturbation equations can be written as follows:
\beq\label{per-system}
\begin{cases}
	\di \phi_t-\s \phi_\xi-\dot{\mb X}(t)(v^S)^{-\mb X}_\xi-\psi_\x=0,\\[3mm]
	\di \psi_t-\s \psi_\xi-\dot{\mb X}(t)(u^S)^{-\mb X}_\xi+(p-\bar p)_\x=-Q_1,\\[3mm]
	\di \frac{R}{\gamma-1}\vartheta_t-\frac{R\s}{\gamma-1} \vartheta_\xi
	-\frac{R}{\gamma-1}\dot{\mb X}(t)(\theta^S)^{-\mb X}_\xi
	+(pu_\x-\bar p\bar u_\x)
	=\kappa\left(\frac{ \theta_\x}{ v}-\frac{\bar \theta_\x}{\bar v}\right)_\xi-Q_2.
\end{cases}
\eeq
Based on the perturbation equations \eqref{per-system} and Lemma \ref{lem-zvh}, we establish the full energy estimates in three steps. ‌We first derive the estimates for the first-order derivatives, then establish the second-order derivative estimates, and finally obtain the space-time estimates of density and velocity by exploiting the wave structure.
\subsection{First-order derivative estimates}
We begin with the following estimate for the first-order derivatives.
\begin{lemma}\label{lem-energy1}
	Under the hypotheses of Proposition \ref{prop2}, there exists $C>0$ (independent of $\delta_1, \eps_1, T$) such that for all $ t\in (0,T]$,
	\begin{align*}
		\begin{aligned}
			&\|(\phi,\psi,\vartheta)_\x(t,\cdot)\|_{L^2(\bbr)}^2+\int_0^t   \|\vartheta_{\x\x}(\tau,\cdot)\|^2_{L^2(\bbr)} d\tau \\
			&\leq C\|(\phi,\psi,\vartheta)_\x(0,\cdot) \|_{H^1(\bbr)}^2+C(\eps_1+\delta_0) \int_0^t \left(  \|(\phi,\psi )_\x(\tau,\cdot)\|^2_{L^2(\bbr)}\right) d\tau+ C\delta_0^{\frac{1}{2}}.
		\end{aligned}
	\end{align*}
\end{lemma}

\begin{proof}
	Multiplying $\eqref{per-system}_{1,\x}$ by $\phi_\x$, $\eqref{per-system}_{2,\x}$ by $\frac{v}{p}\psi_\x$ and $\eqref{per-system}_{3,\x}$ by $\frac{R}{p^2}\vartheta_\x$ respectively and then adding all the resultant equations, we obtain
	\begin{align*}
			&\left(\frac{1}{2}\phi_\x^2+\frac{v}{2p}\psi_{\x}^2+\frac{R}{\gamma-1}\frac{R}{2p^2}\vartheta_\x^2\right)_t+\left(\frac{R}{p}\psi_\x\vartheta_\x-\phi_\x\psi_{\x}-\frac{1}{2}\s\phi_\x^2-\frac{v}{2p}\s\psi_\x^2-\frac{R}{\gamma-1}\frac{R}{2p^2}\s\vartheta_\x^2\right)_\x\\
			&-\left(\frac{v}{2p}\right)_t\psi_\x^2-\frac{R}{\gamma-1}\left(\frac{R}{2p^2}\right)_t\vartheta_\x^2-\left(\frac{R}{p}\right)_\x\psi_\x\vartheta_\x+\s\left(\frac{v}{2p}\right)_\x\psi_\x^2+\s\frac{R}{\gamma-1}\left(\frac{R}{2p^2}\right)_\x\vartheta_\x^2\\
			&-\dot{\mb X}(t)\left[(v^S)^{-\mb X}_{\x\x}\phi_\x+\frac{v}{p}(u^S)^{-\mb X}_{\x\x}\psi_{\x}+\frac{R}{\gamma-1}\frac{R}{p^2}(\theta^S)^{-\mb X}_{\x\x}\vartheta_\x\right]-\frac{R}{pv}v_\x\psi_\x\vartheta_\x-\frac{v}{p}\left(\frac{p}{v}\right)_\x\phi_\x\psi_\x\\
			&+\frac{v}{p}\psi_\x\left[R\bar\theta_\x \left(\frac{1}{v}-\frac{1}{\bar v}\right)+R\bar v_\x\left(\frac{\bar\theta}{\bar v^2}-\frac{\theta}{v^2}\right)\right]_\x + \frac{R}{p^2}\vartheta_\x\left((p-\bar p)\bar u_\x\right)_\x + \kappa \left(\frac{\theta_\x}{v}-\frac{\bar\theta_\x}{\bar v}\right)_\x \left(\frac{R}{p^2}\vartheta_\x\right)_\x\\
			&+\frac{R}{p^2}p_\x\psi_\x\vartheta_\x+\frac{v}{p}\psi_\x Q_{1,\x}+\frac{R}{p^2}\vartheta_\x Q_{2,\x}=0.
	\end{align*}
	Equivalently, we may write
	\begin{align*}
			&\left(\frac{1}{2}\phi_\x^2+\frac{v}{2p}\psi_{\x}^2+\frac{R}{\gamma-1}\frac{R}{2p^2}\vartheta_\x^2\right)_t+\left(\frac{R}{p}\psi_\x\vartheta_\x-\phi_\x\psi_{\x}-\frac{1}{2}\s\phi_\x^2-\frac{v}{2p}\s\psi_\x^2\right.\\
            &\left.-\frac{R}{\gamma-1}\frac{R}{2p^2}\s\vartheta_\x^2\right)_\x
			+\sum_{i=1}^7 L_i=0,
	\end{align*}
	where 
	\begin{align*}
			L_1 :=& -\left(\frac{v}{2p}\right)_{\tau}\psi_\x^2-\frac{R}{\gamma-1}\left(\frac{R}{2p^2}\right)_{\tau}\vartheta_\x^2,\\
			L_2 :=& -\left(\frac{R}{p}\right)_\x\psi_\x\vartheta_\x+\s\left(\frac{v}{2p}\right)_\x\psi_\x^2+\s\frac{R}{\gamma-1}\left(\frac{R}{2p^2}\right)_\x\vartheta_\x^2     -\frac{R}{pv}v_\x\psi_\x\vartheta_\x\\
            &-\frac{v}{p}\left(\frac{p}{v}\right)_\x\phi_\x\psi_\x+\frac{R}{p^2}p_\x\psi_\x\vartheta_\x, \\
			L_3 :=& -\dot{\mb X}(t)\left[(v^S)^{-\mb X}_{\x\x}\phi_\x+\frac{v}{p}(u^S)^{-\mb X}_{\x\x}\psi_{\x}+\frac{R}{\gamma-1}\frac{R}{p^2}(\theta^S)^{-\mb X}_{\x\x}\vartheta_\x\right], \\
			L_4 :=& \frac{v}{p}\psi_\x\left[R\bar\theta_\x \left(\frac{1}{v}-\frac{1}{\bar v}\right)+R\bar v_\x\left(\frac{\bar\theta}{\bar v^2}-\frac{\theta}{v^2}\right)\right]_\x , \\
			L_5 :=&\frac{R}{p^2}\vartheta_\x\left((p-\bar p)\bar u_\x\right)_\x,  \\
			L_6 := &\kappa \left(\frac{\theta_\x}{v}-\frac{\bar\theta_\x}{\bar v}\right)_\x \left(\frac{R}{p^2}\vartheta_\x\right)_\x , \\
			L_7 :=  &\frac{v}{p}\psi_\x Q_{1,\x}+\frac{R}{p^2}\vartheta_\x Q_{2,\x} .
	\end{align*}
	
	By the Sobolev's inequality and {\it a priori} assumption‌ \eqref{apri-ass}, we have
	$$
	|v_\x| \leq |(v-\bar v)_\x| + |\bar v_\x| \leq C(\eps_1+\delta_0),
	$$
	and similarly we have $|p_\x|\leq C(\eps_1+\delta_0)$. Therefore,
	\beq\label{L2}
	 |L_2|\leq C(\eps_1+\delta_0)|(\phi_\x,\psi_\x,\vartheta_\x)|^2.
	\eeq
	For $L_1$, since $p_t=\frac{R}{v}\theta_t-\frac{R\theta}{v^2}v_t$, using $\eqref{NS-1}_1$, $\eqref{NS-1}_3$ and Sobolev's inequality, together with {\it a priori} assumption \eqref{apri-ass}, it remains only to estimate the term $|\theta_t|\psi_\x^2$ as follows,
	\begin{align*}
			|\theta_t|\psi_\x^2
			&\leq C (|\theta_\x|+|u_\x|+|v_\x|+|\theta_{\x\x}|)\psi_\x^2\leq C(\eps_1+\delta_0)\psi_\x^2+C|\theta_{\x\x}|\psi_\x^2\\
			&\leq C(\eps_1+\delta_0)\psi_\x^2+C|\theta_{\x\x}-\bar\theta_{\x\x}|\psi_\x^2+C|\bar\theta_{\x\x}|\psi^2_\x \\
			&\leq C(\eps_1+\delta_0)(\psi^2_\x+\vartheta_{\x\x}^2),
	\end{align*}
	so we have 
	\beq\label{L1}
	|L_1|\leq C(\eps_1+\delta_0)(\psi_{\x}^2+\vartheta_\x^2+\vartheta_{\x\x}^2).
	\eeq
	For $L_3$, by Lemma \ref{lemma1.3}, we have
	\beq\label{L3}
	|L_3|\leq C\deltas|\dot{\mb X}(t)| |(v^S)^{-\mb X}_{\x}| |(\phi_\x,\psi_\x,\vartheta_\x)|.
	\eeq
	For $L_4$, using Young's inequality and the properties of waves, we obtain
	\begin{equation}\label{L4}
	\begin{aligned}
		|L_4| &=\left| \frac{v}{p}\psi_\x\left[R\bar\theta_\x \left(\frac{1}{v}-\frac{1}{\bar v}\right)+R\bar v_\x\left(\frac{\bar\theta}{\bar v^2}-\frac{\theta}{v^2}\right)\right]_\x \right| \\
		& =\left| \frac{v}{p}\psi_\x\left[R\bar\theta_{\x\x} \left(\frac{1}{v}-\frac{1}{\bar v}\right)+R\bar\theta_{\x} \left(\frac{1}{v}-\frac{1}{\bar v}\right)_\x+R\bar v_{\x\x}\left(\frac{\bar\theta}{\bar v^2}-\frac{\theta}{v^2}\right)+R\bar v_{\x}\left(\frac{\bar\theta}{\bar v^2}-\frac{\theta}{v^2}\right)_\x\right] \right|\\
		&\leq C|\psi_\x| \left(|\bar\theta_{\x\x}||\phi| + |\bar\theta_\x|(|\phi|+|\phi_\x|) + |\bar v_{\x\x}||(\phi,\vartheta)|+|\bar v_\x| |(\phi,\vartheta,\phi_\x,\vartheta_\x)|\right)\\
		&\leq C |\psi_\x||\bar v_\x| |(\phi,\vartheta,\phi_\x,\vartheta_\x)|\leq C |\psi_\x||\bar v_\x| |(\phi,\vartheta)|+C \delta_0 |(\phi_\x,\psi_\x,\vartheta_\x)|^2.
	\end{aligned} 
    \end{equation}	
    Similarly, we have
    \beq\label{L5}
    |L_5|\leq C |\vartheta_\x| |\bar v_\x||(\phi,\vartheta)|+C\delta_0 |(\phi_\x,\vartheta_\x)|^2.
    \eeq   
    For $L_6$, a direct calculation gives
	\begin{align*}
		L_6 &= \kappa \left(\frac{\theta_\x}{v}-\frac{\bar\theta_\x}{\bar v}\right)_\x \left(\frac{R}{p^2}\vartheta_\x\right)_\x = \kappa \left[\frac{1}{v}\vartheta_\x+\bar\theta_\x\left(\frac{1}{v}-\frac{1}{\bar v}\right)\right]_\x \left(\frac{R}{p^2}\vartheta_\x\right)_\x \\
		&= \frac{\kappa R}{vp^2}\vartheta_{\x\x}^2+\frac{\kappa}{v}\left(\frac{R}{p^2}\right)_\x\vartheta_\x\vartheta_{\x\x}+\kappa \left(\frac{1}{v}\right)_\x\vartheta_\x  \left(\frac{R}{p^2}\vartheta_\x\right)_\x+\kappa \left[\bar\theta_\x\left(\frac{1}{v}-\frac{1}{\bar v}\right)\right]_\x \left(\frac{R}{p^2}\vartheta_\x\right)_\x\\
		&:=\frac{\kappa R}{vp^2}\vartheta_{\x\x}^2+L_{62}.
    \end{align*}
    From Sobolev's inequality and {\it a priori} assumption \eqref{apri-ass}, together with Lemmas \ref{lemma1.2}-\ref{Lemma 2.2.}, we have
    \begin{equation}\label{L6}
    	\begin{aligned}
    		|L_{62}| &=\left |\frac{\kappa}{v}\left(\frac{R}{p^2}\right)_\x\vartheta_\x\vartheta_{\x\x}+\kappa \left(\frac{1}{v}\right)_\x\vartheta_\x  \left(\frac{R}{p^2}\vartheta_\x\right)_\x+\kappa \left[\bar\theta_\x\left(\frac{1}{v}-\frac{1}{\bar v}\right)\right]_\x \left(\frac{R}{p^2}\vartheta_\x\right)_\x \right|\\
    		&\leq C(\eps_1+\delta_0) |(\phi_\x,\vartheta_\x,\vartheta_{\x\x})|^2+C |(\vartheta_\x,\vartheta_{\x\x})| |\bar v_\x||(\phi,\vartheta)|.
    	\end{aligned} 
    \end{equation}
    Combining $\eqref{L1}$-$\eqref{L6}$ and applying Young's inequality, we have
     \begin{align*}
     		&\|(\phi,\psi,\vartheta)_\x(t,\cdot)\|_{L^2(\bbr)}^2+\int_0^t   \|\vartheta_{\x\x}(\tau,\cdot)\|^2_{L^2(\bbr)} d\tau \\
     		&\leq C\|(\phi,\psi,\vartheta)_\x(0,\cdot) \|_{L^2(\bbr)}^2+C(\eps_1+\delta_0)\int_0^t\int_{\bbr}|(\phi_\x,\psi_\x,\vartheta_\x,\vartheta_{\x\x})|^2 d\x d\tau\\
     		&\quad +C\deltas\int_0^t\int_{\bbr}|\dot{\mb X}(\tau)| |(v^S)^{-\mb X}_{\x}| |(\phi_\x,\psi_\x,\vartheta_\x)|d\x d\tau+C\int_0^t\int_{\bbr} |\psi_\x Q_{1,\x}+\vartheta_\x Q_{2,\x}|d\x d\tau\\
            &\quad+C\int_0^t\int_{\bbr}  |(\psi_\x,\vartheta_\x,\vartheta_{\x\x})| |\bar v_\x||(\phi,\vartheta)| d\x d\tau\\
     		&\leq C\|(\phi,\psi,\vartheta)_\x(0,\cdot) \|_{L^2(\bbr)}^2+C(\eps_1+\delta_0) \int_0^t \left( \mathcal{G}^S(U) + \mathcal{G}^R(U) + \|(\phi,\psi )_\x(\tau,\cdot)\|^2_{L^2(\bbr)}\right) d\tau\\
     		&\quad +C(\eps_1+\delta_0)\int_0^t  \int_{\bbr} |W(\phi,\vartheta)|^2d\x d\tau + C\delta_S \int_{0}^t |\dot{\mb X}(\tau) |^2 d\tau + C\int_0^t \| \vartheta_\x(\tau,\cdot)\|_{L^2(\bbr)}^2 d\tau+C\delta_0.
     \end{align*}
      Finally, applying Lemma $\ref{lem-zvh}$ and \ref{cw-lemma}, and using the smallness of $\eps_1$ and $\delta_1$, we obtain the desired estimate and thus complete the proof of Lemma $\ref{lem-energy1}$.

\end{proof}

\subsection{Second-order derivative estimates}
We next establish the estimates for the second-order derivatives.
\begin{lemma}\label{lem-energy2}
	Under the hypotheses of Proposition \ref{prop2}, there exists $C>0$ (independent of $\delta_1, \eps_1, T$) such that for all $ t\in (0,T]$,
	\begin{align*}
			&\|(\phi,\psi,\vartheta)_{\x\x}(t,\cdot)\|_{L^2(\bbr)}^2+\int_0^t   \|\vartheta_{\x\x\x}(\tau,\cdot)\|^2_{L^2(\bbr)} d\tau \\
			&\leq C\|(\phi,\psi,\vartheta)(0,\cdot) \|_{H^2(\bbr)}^2+C(\eps_1+\delta_0) \int_0^t \left( \|(\phi,\psi )_\x(\tau,\cdot)\|^2_{H^1(\bbr)}\right) d\tau+ C\delta_0^{\frac{1}{2}}.
	\end{align*}
\end{lemma}

\begin{proof}
	Multiplying $\eqref{per-system}_{1,\x\x}$ by $\phi_{\x\x}$, $\eqref{per-system}_{2,\x\x}$ by $\frac{v}{p}\psi_{\x\x}$ and $\eqref{per-system}_{3,\x\x}$ by $\frac{R}{p^2}\vartheta_{\x\x}$ respectively and adding all the resultant equations, we have
	\begin{align*}
			&\left(\frac{1}{2}\phi_{\x\x}^2+\frac{v}{2p}\psi_{\x\x}^2+\frac{R}{\gamma-1}\frac{R}{2p^2}\vartheta_{\x\x}^2\right)_t+\left(\frac{R}{p}\psi_{\x\x}\vartheta_{\x\x}-\phi_{\x\x}\psi_{\x\x}-\frac{1}{2}\s\phi_{\x\x}^2-\frac{v}{2p}\s\psi_{\x\x}^2\right.\\
            &\left.-\frac{R}{\gamma-1}\frac{R}{2p^2}\s\vartheta_{\x\x}^2\right)_\x-\left(\frac{v}{2p}\right)_t\psi_{\x\x}^2-\frac{R}{\gamma-1}\left(\frac{R}{2p^2}\right)_t\vartheta_{\x\x}^2-\left(\frac{R}{p}\right)_\x\psi_{\x\x}\vartheta_{\x\x}+\s\left(\frac{v}{2p}\right)_\x\psi_{\x\x}^2\\
            &+\s\frac{R}{\gamma-1}\left(\frac{R}{2p^2}\right)_\x\vartheta_{\x\x}^2-\dot{\mb X}(t)\left[(v^S)^{-\mb X}_{\x\x\x}\phi_{\x\x}+\frac{v}{p}(u^S)^{-\mb X}_{\x\x\x}\psi_{\x\x}+\frac{R}{\gamma-1}\frac{R}{p^2}(\theta^S)^{-\mb X}_{\x\x\x}\vartheta_{\x\x}\right]\\
            &+\frac{
			\kappa R}{vp^2}\vartheta_{\x\x\x}^2+\frac{2v}{p}\psi_{\x\x}\left[ \left(\frac{R}{v}\right)_\x\vartheta_{\x\x}-\left(\frac{p}{v}\right)_\x\phi_{\x\x}\right]+\frac{2\kappa R}{p^2}\left(\frac{1}{v}\right)_\x\vartheta_{\x\x}\vartheta_{\x\x\x}
            +\frac{2R}{p^2}p_\x\psi_{\x\x}\vartheta_{\x\x} \\
            &+\frac{v}{p}\left(\frac{R}{v}\right)_{\x\x}\psi_{\x\x}\vartheta_\x-\frac{v}{p}\left(\frac{p}{v}\right)_{\x\x}\phi_\x\psi_{\x\x}+\frac{R}{p^2}p_{\x\x}\psi_\x\vartheta_{\x\x} +\frac{\kappa R}{p^2}\left(\frac{1}{v}\right)_{\x\x}\vartheta_\x\vartheta_{\x\x}\\
            &+\frac{v}{p}\psi_{\x\x}\left[R\bar\theta_\x \left(\frac{1}{v}-\frac{1}{\bar v}\right)+R\bar v_\x\left(\frac{\bar\theta}{\bar v^2}-\frac{\theta}{v^2}\right)\right]_{\x\x} 
            - \frac{\kappa R}{p^2} \left[\bar\theta_\x\left(\frac{1}{v}-\frac{1}{\bar v}\right)\right]_{\x\x} \vartheta_{\x\x\x}\\
			&+ \frac{R}{p^2}\vartheta_{\x\x}\left((p-\bar p)\bar u_\x\right)_{\x\x} +\kappa \left(\frac{R}{p^2}\right)_\x \left(\frac{\bar\theta}{\bar v^2}-\frac{\theta}{v^2}\right)_{\x\x}\vartheta_{\x\x}+\frac{v}{p}\psi_{\x\x} Q_{1,\x\x}\\
            &+\left[\frac{R}{p^2}\vartheta_{\x\x} Q_{2,\x}\right]_\x-\frac{R}{p^2}\vartheta_{\x\x\x} Q_{2,\x}
            -\left(\frac{R}{p^2}\right)_\x\vartheta_{\x\x} Q_{2,\x}=0.
	\end{align*}
     Integrating the above equation, we can get
 	 \begin{align*}
 		    \|(\phi,\psi,\vartheta)_{\x\x}(t,\cdot)\|_{L^2(\bbr)}^2+\int_0^t   \|\vartheta_{\x\x\x}\|^2_{L^2(\bbr)} d\tau\leq C\left( \|(\phi,\psi,\vartheta)_{\x\x}(0,\cdot) \|_{L^2(\bbr)}^2+\sum_{i=1}^{6} I_i\right),
     \end{align*}
	   where 
	   \begin{align*}
	   		I_1 :=& \int_0^t\int_{\bbr} \left|\left(\frac{v}{2p}\right)_{\tau}\psi_{\x\x}^2+\frac{R}{\gamma-1}\left(\frac{R}{2p^2}\right)_{\tau}\vartheta_{\x\x}^2\right| d\x d\tau,\\
	   		I_2 :=& \int_0^t\int_{\bbr} \left|-\left(\frac{R}{p}\right)_\x\psi_{\x\x}\vartheta_{\x\x}+\s\left(\frac{v}{2p}\right)_\x\psi_{\x\x}^2+\s\frac{R}{\gamma-1}\left(\frac{R}{2p^2}\right)_\x\vartheta_{\x\x}^2+\frac{2v}{p}\psi_{\x\x}\left[ \left(\frac{R}{v}\right)_\x\vartheta_{\x\x}\right.\right.\\
            &\left.\left.-\left(\frac{p}{v}\right)_\x\phi_{\x\x}\right] +\frac{2\kappa R}{p^2}\left(\frac{1}{v}\right)_\x\vartheta_{\x\x}\vartheta_{\x\x\x}+\frac{2R}{p^2}p_\x\psi_{\x\x}\vartheta_{\x\x}\right| d\x d\tau, \\
	   		I_3 :=& \int_0^t\int_{\bbr} \left|\dot{\mb X}(\tau)\left[(v^S)^{-\mb X}_{\x\x\x}\phi_{\x\x}+\frac{v}{p}(u^S)^{-\mb X}_{\x\x\x}\psi_{\x\x}+\frac{R}{\gamma-1}\frac{R}{p^2}(\theta^S)^{-\mb X}_{\x\x\x}\vartheta_{\x\x}\right]\right| d\x d\tau, \\
	   		I_4 :=& \int_0^t\int_{\bbr} \left|\frac{v}{p}\left(\frac{R}{v}\right)_{\x\x}\psi_{\x\x}\vartheta_\x-\frac{v}{p}\left(\frac{p}{v}\right)_{\x\x}\phi_\x\psi_{\x\x}+\frac{R}{p^2}p_{\x\x}\psi_\x\vartheta_{\x\x} +\frac{\kappa R}{p^2}\left(\frac{1}{v}\right)_{\x\x}\vartheta_\x\vartheta_{\x\x} \right| d\x d\tau, \\
	   		I_5 :=& \int_0^t\int_{\bbr} \left| \frac{v}{p}\psi_{\x\x}\left[R\bar\theta_\x \left(\frac{1}{v}-\frac{1}{\bar v}\right)+R\bar v_\x\left(\frac{\bar\theta}{\bar v^2}-\frac{\theta}{v^2}\right)\right]_{\x\x} + \frac{R}{p^2}\vartheta_{\x\x}\left((p-\bar p)\bar u_\x\right)_{\x\x} \right| \\
	   		&+\left| \kappa \left(\frac{R}{p^2}\right)_\x \left(\frac{\bar\theta}{\bar v^2}-\frac{\theta}{v^2}\right)_{\x\x}\vartheta_{\x\x}- \frac{\kappa R}{p^2} \left[\bar \theta_\x\left(\frac{1}{v}-\frac{1}{\bar v}\right)\right]_{\x\x} \vartheta_{\x\x\x} \right| d\x d\tau,\\
	   		I_6 :=& \int_0^t\int_{\bbr} \left| \frac{v}{p}\psi_{\x\x} Q_{1,\x\x}-\frac{R}{p^2}\vartheta_{\x\x\x} Q_{2,\x}-\left(\frac{R}{p^2}\right)_\x\vartheta_{\x\x} Q_{2,\x} \right| d\x d\tau.
	   \end{align*}
	   For $I_1$, since $p_t=-\frac{R\theta}{v^2}v_t+\frac{R}{v}\theta_t$, by $\eqref{NS-1}_1$ and $\eqref{NS-1}_3$, using Young's inequality and {\it a priori} assumption \eqref{apri-ass},  together with Lemmas \ref{lemma1.2}-\ref{Lemma 2.2.}, we obtain
	   \begin{equation}\label{I1}
	   	\begin{aligned}
	   		I_1 &\leq C \int_0^t\int_{\bbr} |(v_\x,u_\x)| (\psi_{\x\x}^2+\vartheta_{\x\x}^2)+ |(u_\x,\theta_\x,v_\x\theta_\x,\theta_{\x\x})| (\psi_{\x\x}^2+\vartheta_{\x\x}^2)d\x d\tau\\
	   		& \leq C \int_0^t\int_{\bbr} (\eps_1+\delta_0)(\psi_{\x\x}^2+\vartheta_{\x\x}^2) + |\vartheta_{\x\x}| (\psi_{\x\x}^2+\vartheta_{\x\x}^2) d\x d\tau\\
	   	    & \leq C (\eps_1+\delta_0) \int_0^t(\|\psi_{\x\x}\|_{L^2(\bbr)}^2+\|\vartheta_{\x\x}\|_{L^2(\bbr)}^2) d\tau \\
            &\quad+C\int_0^t \|\vartheta_{\x\x}\|_{L^{\infty}(\bbr)} (\|\psi_{\x\x}\|_{L^2(\bbr)}^2+\|\vartheta_{\x\x}\|_{L^2(\bbr)}^2) d\tau\\
	   	    & \leq C (\eps_1+\delta_0) \int_0^t(\|\psi_{\x\x}\|_{L^2(\bbr)}^2+\|\vartheta_{\x\x}\|_{L^2(\bbr)}^2) d\tau \\
            &\quad +C\int_0^t \|\vartheta_{\x\x}\|_{L^{2}(\bbr)}^{\frac{1}{2}}\|\vartheta_{\x\x\x}\|_{L^2(\bbr)}^{\frac{1}{2}} (\|\psi_{\x\x}\|_{L^2(\bbr)}^2+\|\vartheta_{\x\x}\|_{L^2(\bbr)}^2) d\tau\\
	   		& \leq C (\eps_1+\delta_0) \int_0^t(\|\psi_{\x\x}\|_{L^2(\bbr)}^2+\|\vartheta_{\x\x}\|_{L^2(\bbr)}^2) d\tau \\
            &\quad+C \eps_1 \int_0^t \|\vartheta_{\x\x\x}\|_{L^2(\bbr)}^{\frac{1}{2}} (\|\psi_{\x\x}\|_{L^2(\bbr)}^{\frac{3}{2}}+\|\vartheta_{\x\x}\|_{L^2(\bbr)}^{\frac{3}{2}}) d\tau\\
	   	    & \leq C (\eps_1+\delta_0) \int_0^t(\|\psi_{\x\x}\|_{L^2(\bbr)}^2+\|\vartheta_{\x\x}\|_{L^2(\bbr)}^2+ \|\vartheta_{\x\x\x}\|_{L^2(\bbr)}^2) d\tau.
	   	\end{aligned}
	   \end{equation}     
	    For $I_2$, $I_3$ and $I_4$, similar to $I_1$, we obtain
	    \begin{equation}\label{I2}
	   	\begin{aligned}
	   		&I_2+I_3+I_4 \\
            &\leq C (\eps_1+\delta_0) \int_0^t(\|(\phi,\psi,\vartheta)_\x\|_{H^1}^2+ \|\vartheta_{\x\x\x}\|_{L^2(\bbr)}^2) d\tau\\
            &\quad
             +C\delta_S\int_0^t|\dot{\mb X}|^2\,d\tau
             +C\delta_0.
	   	\end{aligned}
	   \end{equation}
	   For $I_5$, using Young's inequality and Sobolev's inequality, together with {\it a priori} assumption \eqref{apri-ass}, we have
	   	\begin{equation}\label{I5}
	   	\begin{aligned}
	   		I_5&\leq C(\eps_1+\delta_0) \int_0^t \left( \mathcal{G}^S(U) + \mathcal{G}^R(U) + \|(\phi,\psi ,\vartheta)_\x\|^2_{H^1(\bbr)}+\|\vartheta_{\x\x\x}\|_{L^2(\bbr)}^2\right) d\tau\\
	   		&\quad +C(\eps_1+\delta_0)\int_0^t  \int_{\bbr} |W(\phi,\vartheta)|^2d\x d\tau.
	   	\end{aligned}
	   	\end{equation}
	   		     For $I_6$, from Lemma \ref{lemma2.2}, using Young's inequality and {\it a priori }assumption \eqref{apri-ass}, we get
	   	\begin{equation}\label{I6}
	   	\begin{aligned}
	   		I_6&\leq \frac{1}{4} \int_0^t \|\vartheta_{\x\x\x}\|_{L^2(\bbr)}^2 d\tau +C\int_0^t \left(\| Q_{1,\x\x}\|_{L^2(\bbr)}\|\psi_{\x\x}\|_{L^2(\bbr)}\right) d\tau\\
            &\quad +C \int_0^t \left(\|Q_{2,\x}\|_{L^2(\bbr)}^2+\|\vartheta_{\x\x}\|_{L^2(\bbr)}^2\right) d\tau\\
	   		&\leq \frac{1}{4} \int_0^t \|\vartheta_{\x\x\x}\|_{L^2(\bbr)}^2 d\tau+C \int_0^t \|\vartheta_{\x\x}\|_{L^2(\bbr)}^2 d\tau+C\delta_0.\\
	   	\end{aligned}
	   	\end{equation}
        Combining \eqref{I1}-\eqref{I6}, applying Lemmas
    \ref{lem-energy1} and \ref{lem-zvh}, and using the smallness of $\eps_1$ and $\delta_1$, we complete the proof.
\end{proof}

\subsection{Space-time estimates for density and velocity}
Since the system \eqref{NS} contains no viscous dissipation in the momentum equation, we must establish separate space-time estimates for density and velocity to control the nonlinear terms involving velocity and density in the energy estimate. To obtain these estimates, we choose the multipliers used in the proofs of Lemmas \ref{lem-energy3} and \ref{lem-energy4} according to the wave structure \eqref{eq:rough_wave_intro} associated with the perturbation system \eqref{per-system}. For the first-order space-time estimates in Lemma \ref{lem-energy3}, we multiply $\eqref{per-system}_2$ and $\eqref{per-system}_3$ by $-\frac{p}{2}\phi_\x$ and $\psi_\x$, respectively. For the second-order space-time estimates in Lemma \ref{lem-energy4}, we first differentiate $\eqref{per-system}_2$ and $\eqref{per-system}_3$ with respect to $\xi$, and then multiply the resulting equations by $-\frac{p}{2}\phi_{\xi\xi}$ and $\psi_{\xi\xi}$, respectively, to obtain the required control of $\phi_{\xi\xi}$ and $\psi_{\xi\xi}$.

\begin{lemma}\label{lem-energy3}
	Under the hypotheses of Proposition \ref{prop2}, there exists $C>0$ (independent of $\delta_1, \eps_1, T$) such that for all $ t\in (0,T]$,
	\begin{align*}
			&\int_0^t \|(\phi,\psi)_{\x}(\tau,\cdot)\|_{L^2(\bbr)}^2 d\tau\\
			&\leq C\left(\|(\psi,\vartheta)(t,\cdot) \|_{L^2(\bbr)}^2+\|(\phi,\psi)_{\x}(t,\cdot) \|_{L^2(\bbr)}^2+\|(\psi,\vartheta)(0,\cdot) \|_{L^2(\bbr)}^2+\|(\phi,\psi)_{\x}(0,\cdot) \|_{L^2(\bbr)}^2\right)\\
			&\quad +C\int_0^t   \|\vartheta_\x(\tau,\cdot)\|_{H^1(\bbr)}^2d\tau + C\delta_0+ C\delta_S \int_{0}^t |\dot{\mb X}(\tau) |^2 d\tau\\
			&\quad +C(\eps_1+\delta_0) \int_0^t \left( \mathcal{G}^S(U) + \mathcal{G}^R(U) +\int_{\bbr} |W(\phi,\psi,\vartheta)|^2d\x \right) d\tau.
	\end{align*}
\end{lemma}

\begin{proof}
	Multiplying $\eqref{per-system}_2$ by $-\frac{p}{2}\phi_\x$ and $\eqref{per-system}_3$ by $\psi_\x$ and adding the equations, we get
		\begin{align*}
			&\frac{p}{2}\psi_\x^2+\frac{p^2}{2v}\phi_\x^2+\left(\frac{R}{\gamma-1}\vartheta\psi_\x-\frac{p}{2}\psi\phi_\x\right)_t+\left(\frac{p}{2}\psi\phi_t-\frac{R}{\gamma-1}\vartheta\psi_t\right)_\x\\
			&+\left(\frac{p}{2}\right)_t\psi \phi_\x-\left(\frac{p}{2}\right)_\x\psi\phi_t - \frac{R}{\gamma-1}(p-\bar p)_\x\vartheta_\x - \frac{Rp}{2v}\phi_\x\vartheta_\x-\kappa \left(\frac{\theta_\x}{v}-\frac{\bar\theta_\x}{\bar v}\right)_\x \psi_\x\\
			&+\dot{\mb X}(t)\left[\frac{p}{2}(u^S)^{-\mb X}_{\x}\phi_\x+\frac{R}{\gamma-1}(u^S)^{-\mb X}_{\x}\vartheta_\x-\frac{R}{\gamma-1}(\theta^S)^{-\mb X}_{\x}\psi_{\x}-\frac{p}{2}(v^S)^{-\mb X}_{\x}\psi_\x\right]\\
			&+\bar u_\x (p-\bar p) \psi_\x -\frac{p}{2}R\bar\theta_\x\left(\frac{1}{v}-\frac{1}{\bar v}\right)\phi_\x-\frac{p}{2}R\bar v_\x\left(\frac{\bar \theta}{\bar v^2}-\frac{\theta}{v^2}\right)\phi_\x\\
			&-\frac{R}{\gamma-1}\vartheta_\x Q_1 +\psi_\x Q_2 -\frac{p}{2}\phi_\x Q_1=0.
	\end{align*}
	To get the result, we integrate the equation
	\begin{align*}
			&\int_0^t\int_{\bbr}\left(\psi_\x^2+\phi_\x^2 \right)d\x d\tau\\
			&\leq C \left( \|(\psi, \vartheta, \psi_\x, \phi_\x)(t,\cdot)\|_{L^2(\bbr)} ^2+\|(\psi, \vartheta, \psi_\x,\phi_\x)(0,\cdot)\|_{L^2(\bbr)} ^2\right)+ \sum_{i=1}^8 K_i,
	\end{align*}
	where 
		\begin{align*}
			K_1 := &\int_0^t\int_{\bbr} \left|\left(\frac{p}{2}\right)_\tau\psi \phi_\x\right| d\x d\tau,\\
			K_2 := &\int_0^t\int_{\bbr} \left|\left(\frac{p}{2}\right)_\x\psi\phi_{\tau} \right| d\x d\tau, \\
			K_3 := &\int_0^t\int_{\bbr} \left|\bar u_\x (p-\bar p) \psi_\x -\frac{p}{2}R\bar\theta_\x\left(\frac{1}{v}-\frac{1}{\bar v}\right)\phi_\x-\frac{p}{2}R\bar v_\x\left(\frac{\bar \theta}{\bar v^2}-\frac{\theta}{v^2}\right)\phi_\x\right| d\x d\tau, \\
			K_4 := &\int_0^t\int_{\bbr} \left|\dot{\mb X}(\tau)\left[\frac{p}{2}(u^S)^{-\mb X}_{\x}\phi_\x+\frac{R}{\gamma-1}(u^S)^{-\mb X}_{\x}\vartheta_\x-\frac{R}{\gamma-1}(\theta^S)^{-\mb X}_{\x}\psi_{\x}\right.\right.\\
            &\left.\left.-\frac{p}{2}(v^S)^{-\mb X}_{\x}\psi_\x\right] \right| d\x d\tau, \\
			K_5 := &\int_0^t\int_{\bbr} \left| \frac{R}{\gamma-1}(p-\bar p)_\x\vartheta_\x \right| d\x d\tau,\\
			K_6 := &\int_0^t\int_{\bbr} \left| \kappa \left(\frac{\theta_\x}{v}-\frac{\bar\theta_\x}{\bar v}\right)_\x \psi_\x \right| d\x d\tau,\\
			K_7 :=  &\int_0^t\int_{\bbr} \left|  \frac{Rp}{2v}\phi_\x\vartheta_\x\right| d\x d\tau,\\
			K_8 := &\int_0^t\int_{\bbr} \left|  -\frac{R}{\gamma-1}\vartheta_\x Q_1 +\psi_\x Q_2 -\frac{p}{2}\phi_\x Q_1\right| d\x d\tau.
	\end{align*}
	For $K_1$, as $p_t=-\frac{R\theta}{v^2}v_t+\frac{R}{v}\theta_t$, from $\eqref{NS-1}_1$ and $\eqref{NS-1}_3$,  using Young's inequality and Sobolev's inequality, together with Lemmas \ref{lemma1.2}-\ref{Lemma 2.2.}, we have 
	\begin{align}\label{K1}
    \begin{aligned}
			K_1 &\leq C \int_0^t\int_{\bbr} |(v_\x,u_\x)||\psi \phi_\x| + |(u_\x,\theta_\x,v_\x\theta_\x,\theta_{\x\x})| |\psi \phi_\x| d\x d\tau\\
			&\leq C(\eps_1+\delta_0) \int_0^t \left(\mathcal{G}^S + \mathcal{G}^R+\|(\phi,\psi,\vartheta)_\x\|_{L^2(\bbr)}^2 +  \|\vartheta_{\x\x}\|_{L^2(\bbr)}^2\right)  d\tau\\
            &\quad+C(\eps_1+\delta_0) \int_0^t\int_{\bbr} |W\psi|^2d\x  d\tau,
        \end{aligned}
		\end{align}
    where we also use {\it a priori} assumption \eqref{apri-ass}.
    
	Likewise, for $K_2$, we have
	\begin{equation}\label{K2}
	\begin{aligned}
		K_2 &\leq C \int_0^t\int_{\bbr} |(v_\x,\theta_\x)||\psi \phi_{\tau}| d\x d\tau\\
		& \leq C \int_0^t\int_{\bbr} |(v_\x,\theta_\x)||\psi| (|\phi_\x|+|\psi_\x|+|\dot{\mb X}(\tau)(v^S)^{-\mb X}_{\x}|) d\x d\tau\\
		&\leq C(\eps_1+\delta_0) \int_0^t \left(\mathcal{G}^S(U) + \mathcal{G}^R(U)+\|(\phi,\psi,\vartheta)_\x\|_{L^2(\bbr)}^2 \right)  d\tau\\
		&\quad+ C(\eps_1+\delta_0) \int_0^t\int_{\bbr} |W\psi|^2d\x  d\tau+C\delta_S \int_{0}^t |\dot{\mb X}(\tau) |^2 d\tau.
	\end{aligned}
    \end{equation}
	For $K_3$, using Young's inequality and Sobolev's inequality together with Lemmas \ref{lemma1.2}-\ref{Lemma 2.2.}, we can get  
	\begin{equation}\label{K3}
		\begin{aligned}
			K_3 &\leq C \int_0^t\int_{\bbr} |\bar v_\x(\phi,\vartheta)| |\psi_\x| +|\bar v_\x\phi| |\phi_\x| +|\bar v_\x (\phi,\vartheta)||\phi_\x |d\x d\tau\\
			& \leq C\delta_0\int_0^t \left(\|(\phi,\vartheta)_\x\|_{L^2(\bbr)}^2+\mathcal{G}^S(U) + \mathcal{G}^R(U)\right) d\tau  +C\delta_0 \int_0^t\int_{\bbr} |W(\phi,\vartheta)|^2d\x  d\tau.
		\end{aligned}
	\end{equation}   
    Similarly, for $K_4$, we obtain
    \begin{equation}\label{K4}
    	\begin{aligned}
    		K_4 \leq C \deltas \int_0^t \|(\phi,\psi,\vartheta)_\x\|_{L^2(\bbr)}^2  d\tau+C\delta_S \int_{0}^t |\dot{\mb X}(\tau) |^2 d\tau.
    	\end{aligned}
    \end{equation}
     For $K_5$, from $(p-\bar p)_\xi=\frac{R\vartheta_\x}{v}-\frac{R\theta\phi_\x}{v^2}+R\bar\theta_\x\left(\frac 1v-\frac 1{\bar v}\right)+R\bar v_\x \left(\frac{\bar\theta}{\bar v^2}-\frac{\theta}{v^2}\right)$, together with Lemma \ref{lemma1.2}-\ref{Lemma 2.2.}, we have
     \begin{equation}\label{K5}
     	\begin{aligned}
     		K_5 &\leq C \int_0^t\int_{\bbr} |\vartheta_\x| (|\vartheta_\x|+|\phi_\x|+|\bar\theta_\x||\phi|+|\bar v_\x||(\phi,\vartheta)|) d\x d\tau\\
     		&\leq C\delta_0 \int_0^t \left(\mathcal{G}^S(U) + \mathcal{G}^R(U)\right)  d\tau+ C\delta_0 \int_0^t\int_{\bbr} |W(\phi,\vartheta)|^2d\x  d\tau\\
     		&\quad +\frac{1}{4} \int_0^t\| \phi_\x(\tau)\|_{L^2(\bbr)}^2 d\tau+C \int_0^t \| \vartheta_\x(\tau)\|_{L^2(\bbr)}^2 d\tau.
     	\end{aligned}
     \end{equation}
      For $K_6$, using Young's inequality together with Lemmas \ref{lemma1.2}-\ref{Lemma 2.2.}, we can get
      \begin{equation}\label{K6}
      	\begin{aligned}
      		K_6 &\leq C \int_0^t\int_{\bbr} |\psi_\x| (|\vartheta_\x|+|\vartheta_{\x\x}|+|\bar\theta_{\x\x}||\phi|+|\bar \theta_\x||\phi|) d\x d\tau\\
      		&\leq C\delta_0 \int_0^t \left(\mathcal{G}^S(U) + \mathcal{G}^R(U)\right)  d\tau+ C\delta_0 \int_0^t\int_{\bbr} |W\phi|^2d\x  d\tau\\
      		&\quad+\frac{1}{4} \int_0^t\| \psi_\x\|_{L^2(\bbr)}^2 d\tau+C \int_0^t \| \vartheta_\x\|_{H^1(\bbr)}^2 d\tau.
      	\end{aligned}
      \end{equation}
      For $K_7$ and $K_8$, using Young's inequality and Lemma \ref{lemma2.2}, we have
       \begin{equation}\label{K7}
      	\begin{aligned}
      		K_7+K_8 &\leq C \int_0^t\int_{\bbr} |\phi_\x\vartheta_\x|  +| \vartheta_\x Q_1| +|\psi_\x Q_2| +|\phi_\x Q_1| d\x d\tau\\
      		&\leq \frac{1}{4} \int_0^t\| (\phi_\x,\psi_\x)\|_{L^2(\bbr)}^2 d\tau+C \int_0^t \left(\| \vartheta_\x\|_{L^2(\bbr)}^2 +\| (Q_1,Q_2)\|_{L^2(\bbr)}^2 \right)d\tau\\
      		&\leq \frac{1}{4} \int_0^t\| (\phi_\x,\psi_\x)\|_{L^2(\bbr)}^2 d\tau+C \int_0^t \| \vartheta_\x\|_{L^2(\bbr)}^2 d\tau +C\delta_0.
      	\end{aligned}
      \end{equation}
       Combining \eqref{K1}-\eqref{K7}, and then using the smallness of $\eps_1$ and $\delta_1$, we can absorb all the terms involving $\|(\phi,\psi)_\x\|_{L^2(\bbr)}^2$ into the left-hand side.
\end{proof}

\begin{lemma}\label{lem-energy4}
	Under the hypotheses of Proposition \ref{prop2}, there exists $C>0$ (independent of $\delta_1, \eps_1, T$) such that for all $ t\in (0,T]$,
	\begin{align*}
			&\int_0^t \|(\phi,\psi)_{\x\x}(\tau,\cdot)\|_{L^2(\bbr)}^2 d\tau\\
			&\leq C\left(\|(\psi,\vartheta)_{\x}(t,\cdot) \|_{L^2(\bbr)}^2+\|(\phi,\psi)_{\x\x}(t,\cdot) \|_{L^2(\bbr)}^2+\|(\psi,\vartheta)_{\x}(0,\cdot) \|_{L^2(\bbr)}^2+\|(\phi,\psi)_{\x\x}(0,\cdot) \|_{L^2(\bbr)}^2\right)\\
			&\quad+C\int_0^t   \|\vartheta_\x(\tau,\cdot)\|_{H^2(\bbr)}^2d\tau + C\delta_0+ C\delta_S \int_{0}^t |\dot{\mb X}(\tau) |^2 d\tau\\
			&\quad+C(\eps_1+\delta_0) \int_0^t \left( \mathcal{G}^S(U) + \mathcal{G}^R(U) +\int_{\bbr} |W(\phi,\psi,\vartheta)|^2d\x + \|(\phi,\psi)_{\x}(\tau,\cdot)\|_{L^2(\bbr)}^2 \right) d\tau.
	\end{align*}
\end{lemma}

	\begin{proof}
	Multiplying $\eqref{per-system}_{2,\x\x}$ by $-\frac{p}{2}\phi_{\x\x}$ and $\eqref{per-system}_{3,\x\x}$ by $\psi_{\x\x}$ and adding the equations, we get
	\begin{align*}
			&\frac{p}{2}\psi_{\x\x}^2+\frac{p^2}{2v}\phi_{\x\x}^2+\left(\frac{R}{\gamma-1}\vartheta_\x\psi_{\x\x}-\frac{p}{2}\psi_\x\phi_{\x\x}\right)_t+\left(\frac{p}{2}\psi_\x\phi_{\x t}-\frac{R}{\gamma-1}\vartheta_\x\psi_{\x t}\right)_\x\\
			&+\left(\frac{p}{2}\right)_t\psi_\x \phi_{\x\x}-\left(\frac{p}{2}\right)_\x\psi_\x\phi_{\x t} - \frac{R}{\gamma-1}(p-\bar p)_{\x\x}\vartheta_{\x\x} - \frac{Rp}{2v}\phi_{\x\x}\vartheta_{\x\x}-\kappa \left(\frac{\theta_\x}{v}-\frac{\bar\theta_\x}{\bar v}\right)_{\x\x} \psi_{\x\x}\\
			&+\dot{\mb X}(t)\left[\frac{p}{2}(u^S)^{-\mb X}_{\x\x}\phi_{\x\x}+\frac{R}{\gamma-1}(u^S)^{-\mb X}_{\x\x}\vartheta_{\x\x}-\frac{R}{\gamma-1}(\theta^S)^{-\mb X}_{\x\x}\psi_{\x\x}-\frac{p}{2}(v^S)^{-\mb X}_{\x\x}\psi_{\x\x}\right]\\
			&+\left(\bar u_\x (p-\bar p)\right)_\x\psi_{\x\x} -\frac{p}{2}\left[R\bar\theta_\x\left(\frac{1}{v}-\frac{1}{\bar v}\right)+R\bar v_\x\left(\frac{\bar \theta}{\bar v^2}-\frac{\theta}{v^2}\right)\right]_\x\phi_{\x\x}+ p_\x \psi_\x\psi_{\x\x}\\
			&-\frac{p}{2}\phi_{\x\x}\left[\left(\frac{R}{v}\right)_\x\vartheta_\x-\left(\frac{p}{v}\right)_\x\phi_\x\right]-\frac{R}{\gamma-1}\vartheta_{\x\x} Q_{1,\x} +\psi_{\x\x} Q_{2,\x} -\frac{p}{2}\phi_{\x\x} Q_{1,\x}=0.
	\end{align*}
	So we have the following inequality
		\begin{align*}
			&\int_0^t\int_{\bbr}\left(\psi_{\x\x}^2+\phi_{\x\x}^2 \right)d\x d\tau\\
			&\leq C \left( \|(\psi_\x, \vartheta_\x, \psi_{\x\x}, \phi_{\x\x})(t,\cdot)\|_{L^2(\bbr)} ^2+\|(\psi_\x, \vartheta_\x, \psi_{\x\x},\phi_{\x\x})(0,\cdot)\|_{L^2(\bbr)} ^2\right)+ \sum_{i=1}^9 H_i,
	\end{align*}
	where 
	\begin{align*}
			H_1 := &\int_0^t\int_{\bbr} \left|\left(\frac{p}{2}\right)_\tau\psi_\x \phi_{\x\x}\right| d\x d\tau,\\
			H_2 := &\int_0^t\int_{\bbr} \left|\left(\frac{p}{2}\right)_\x\psi_\x\phi_{\x \tau} \right| d\x d\tau, \\
			H_3 := &\int_0^t\int_{\bbr} \left|    \frac{Rp}{2v}\phi_{\x\x}\vartheta_{\x\x} \right| d\x d\tau, \\
			H_4 := &\int_0^t\int_{\bbr} \left|\dot{\mb X}(\tau)\left[\frac{p}{2}(u^S)^{-\mb X}_{\x\x}\phi_{\x\x}+\frac{R}{\gamma-1}(u^S)^{-\mb X}_{\x\x}\vartheta_{\x\x}-\frac{R}{\gamma-1}(\theta^S)^{-\mb X}_{\x\x}\psi_{\x\x}\right.\right.\\
            &\left.\left.-\frac{p}{2}(v^S)^{-\mb X}_{\x\x}\psi_{\x\x}\right] \right| d\x d\tau, \\
			H_5 := &\int_0^t\int_{\bbr} \left| \frac{R}{\gamma-1}(p-\bar p)_{\x\x}\vartheta_{\x\x} \right| d\x d\tau,\\
			H_6 := &\int_0^t\int_{\bbr} \left| \kappa \left(\frac{\theta_\x}{v}-\frac{\bar\theta_\x}{\bar v}\right)_{\x\x} \psi_{\x\x} \right| d\x d\tau,\\
			H_7 := &\int_0^t\int_{\bbr} \left|   \left(\bar u_\x (p-\bar p)\right)_\x\psi_{\x\x} -\frac{p}{2}\left[R\bar\theta_\x\left(\frac{1}{v}-\frac{1}{\bar v}\right)+R\bar v_\x\left(\frac{\bar \theta}{\bar v^2}-\frac{\theta}{v^2}\right)\right]_\x\phi_{\x\x}  \right| d\x d\tau,\\
			H_8 := &\int_0^t\int_{\bbr} \left|   p_\x \psi_\x\psi_{\x\x} -\frac{p}{2}\phi_{\x\x}\left[\left(\frac{R}{v}\right)_\x\vartheta_\x-\left(\frac{p}{v}\right)_\x\phi_\x\right]  \right| d\x d\tau,\\
			H_9 := &\int_0^t\int_{\bbr} \left|  -\frac{R}{\gamma-1}\vartheta_{\x\x} Q_{1,\x} +\psi_{\x\x} Q_{2,\x} -\frac{p}{2}\phi_{\x\x} Q_{1,\x}\right| d\x d\tau.
	\end{align*}
	For $H_1$, as $p_t=-\frac{R\theta}{v^2}v_t+\frac{R}{v}\theta_t$, from $\eqref{NS-1}_1$ and $\eqref{NS-1}_3$,  using Young's inequality and Sobolev's inequality, together with Lemmas \ref{lemma1.2}-\ref{Lemma 2.2.} and {\it a priori} assumption \eqref{apri-ass}, we have
	\begin{equation}\label{H1}
		\begin{aligned}
			H_1 &\leq C \int_0^t\int_{\bbr} |(v_\x,u_\x)||\psi_\x \phi_{\x\x}| + |(u_\x,\theta_\x,v_\x\theta_\x,\theta_{\x\x})| |\psi_\x \phi_{\x\x}| d\x d\tau\\
			& \leq C \int_0^t\int_{\bbr} |(\phi_\x,\psi_\x,\vartheta_\x, \vartheta_{\x\x})| |\psi_\x \phi_{\x\x}| + |(\bar v_\x,\bar u_\x, \bar \theta_\x,\bar\theta_{\x\x})| |\psi_\x \phi_{\x\x}| d\x d\tau\\
			&\leq C(\eps_1+\delta_0) \int_0^t \left(\|\psi_\x\|_{L^2(\bbr)}^2 +  \|(\phi_{\x\x},\vartheta_{\x\x})\|_{L^2(\bbr)}^2\right)  d\tau.
		\end{aligned}
	\end{equation}
	Likewise, for $H_2$, we have
	\begin{equation}\label{H2}
		\begin{aligned}
			H_2 &\leq C \int_0^t\int_{\bbr} |(v_\x,\theta_\x)||\psi_\x \phi_{\x \tau}| d\x d\tau\\
			& \leq C \int_0^t\int_{\bbr} |(v_\x,\theta_\x)| |\psi_\x| (|\phi_{\x\x}|+|\psi_{\x\x}|+|\dot{\mb X}(\tau)(v^S)^{-\mb X}_{\x\x}|) d\x d\tau\\
			&\leq C(\eps_1+\delta_0) \int_0^t \left(\|\psi_{\x}\|_{L^2(\bbr)}^2 +\|(\phi,\psi)_{\x\x}\|_{L^2(\bbr)}^2 \right)  d\tau +C\delta_S \int_{0}^t |\dot{\mb X}(\tau) |^2 d\tau.
		\end{aligned}
	\end{equation}
	For  $H_3$, using Young's inequality, we obtain
	 \begin{equation}\label{H3}
		\begin{aligned}
			H_3 
			&\leq \frac{1}{8} \int_0^t\| \phi_{\x\x}\|_{L^2(\bbr)}^2 d\tau+C \int_0^t \| \vartheta_{\x\x} \|_{L^2(\bbr)}^2 d\tau.
		\end{aligned}
	\end{equation}	
	For $H_4$, using Young's inequality together with Lemmas \ref{lemma1.2}-\ref{Lemma 2.2.}, we have
	\begin{equation}\label{H4}
		\begin{aligned}
			H_4 \leq C \deltas \int_0^t \|(\phi,\psi,\vartheta)_{\x\x}\|_{L^2(\bbr)}^2  d\tau+C\delta_S \int_{0}^t |\dot{\mb X}(\tau) |^2 d\tau.
		\end{aligned}
	\end{equation}
	     For $H_5$, since 
	     \begin{align*}
	     	(p-\bar p)_{\x\xi}
	     	&=\left[\frac{R\vartheta_\x}{v}-\frac{R\theta\phi_\x}{v^2}+R\bar\theta_\x\left(\frac 1v-\frac 1{\bar v}\right)+R\bar v_\x \left(\frac{\bar\theta}{\bar v^2}-\frac{\theta}{v^2}\right)\right]_\x\\
	     	&=\frac{R\vartheta_{\x\x}}{v}+\left(\frac{R}{v}\right)_\x\vartheta_\x-\frac{R\theta\phi_{\x\x}}{v^2}-\left(\frac{R\theta}{v^2}\right)_\x \phi_\x+R\bar\theta_{\x\x}\left(\frac 1v-\frac 1{\bar v}\right)\\
	     	&\quad +R\bar\theta_\x\left(\frac 1v-\frac 1{\bar v}\right)_\x+R\bar v_{\x\x} \left(\frac{\bar\theta}{\bar v^2}-\frac{\theta}{v^2}\right)+R\bar v_\x \left(\frac{\bar\theta}{\bar v^2}-\frac{\theta}{v^2}\right)_\x,
	     \end{align*}
	     using Sobolev's inequality and Young's inequality, together with {\it a priori} assumption \eqref{apri-ass}, we have
	\begin{equation}\label{H5}
		\begin{aligned}
			H_5 &\leq C(\eps_1+\delta_0) \int_0^t\int_{\bbr} |\vartheta_{\x\x}| (|\vartheta_\x|+|\phi_\x|+|\bar\theta_\x||\phi|+|\bar v_\x||(\phi,\vartheta)|) d\x d\tau\\
			&\quad + C \int_0^t\int_{\bbr} |\vartheta_{\x\x}| (|\vartheta_{\x\x}|+|\phi_{\x\x}|) d\x d\tau\\
			&\leq C(\eps_1+\delta_0) \int_0^t \left(\mathcal{G}^S(U) + \mathcal{G}^R(U)+\|(\phi,\vartheta)_{\x}\|_{L^2(\bbr)}^2 \right)  d\tau\\
            &\quad + C(\eps_1+\delta_0) \int_0^t\int_{\bbr} |W(\phi,\vartheta)|^2d\x  d\tau +\frac{1}{8} \int_0^t\| \phi_{\x\x}\|_{L^2(\bbr)}^2 d\tau +C \int_0^t \| \vartheta_{\x\x}\|_{L^2(\bbr)}^2 d\tau.
		\end{aligned}
	\end{equation}
	Likewise, for $H_7$, we have
	\begin{equation}\label{H7}
		\begin{aligned}
			H_7 &\leq C \int_0^t\int_{\bbr} (|\phi_{\x\x}|,|\vartheta_{\x\x}|) (|\vartheta_\x|+|\phi_\x|+|\bar\theta_\x||\phi|+|\bar v_\x||(\phi,\vartheta)|) d\x d\tau\\
			&\leq C \delta_0 \int_0^t \left(\mathcal{G}^S(U) + \mathcal{G}^R(U)+\|(\phi,\vartheta)_{\x}\|_{H^1(\bbr)}^2 \right)  d\tau + C\delta_0 \int_0^t\int_{\bbr} |W(\phi,\vartheta)|^2d\x  d\tau.
		\end{aligned}
	\end{equation}	
	For $H_6$, using Young's inequality and Sobolev's inequality, together with Lemmas \ref{lemma1.2}-\ref{Lemma 2.2.} and {\it a priori} assumption \eqref{apri-ass}, we get 
	 	\begin{equation}\label{H6}
	 	\begin{aligned}
	 		H_6 &\leq C \int_0^t\int_{\bbr} |\psi_{\x\x}| \left(|\vartheta_{\x\x\x}|+|v_\x| |\vartheta_{\x\x}|+(|v_\x|^2 +|v_{\x\x}|) |\vartheta_{\x}|\right) d\x d\tau\\
	 		&\quad +C \int_0^t\int_{\bbr} |\psi_{\x\x}| \left| \left[\bar\theta_\x \left(\frac{1}{v}-\frac{1}{\bar v}\right)\right]_{\x\x}\right|d\x d\tau\\
	 		&\leq C \delta_0 \int_0^t \left(\mathcal{G}^S(U) + \mathcal{G}^R(U) \right)  d\tau+ C\delta_0 \int_0^t\int_{\bbr} |W(\phi,\vartheta)|^2d\x  d\tau\\
	 		&\quad + C (\eps_1+\delta_0) \int_0^t \left(\|(\phi,\vartheta)_{\x}\|_{H^1(\bbr)}^2 \right)  d\tau
            +\frac{1}{4} \int_0^t\| \psi_{\x\x}\|_{L^2(\bbr)}^2 d\tau
	 		+C \int_0^t \| \vartheta_{\x\x\x}\|_{L^2(\bbr)}^2 d\tau.
	 	\end{aligned}
	 \end{equation}
	Similar to $H_6$, for $H_8$, we have
	\begin{equation}\label{H8}
		\begin{aligned}
			H_8 \leq C (\eps_1+\delta_0) \int_0^t \|(\phi,\psi,\vartheta)_\x\|_{L^2(\bbr)}^2  +\|(\phi,\psi)_{\x\x}\|_{L^2(\bbr)}^2 d\tau.
		\end{aligned}
	\end{equation}
	 For  $H_9$, using Young's inequality and Lemma \ref{lemma2.2}, we get
	\begin{equation}\label{H9}
		\begin{aligned}
			H_9 &\leq C \int_0^t\int_{\bbr} \left(| \vartheta_{\x\x} Q_{1,\x}| +|\psi_{\x\x} Q_{2,\x}| +|\phi_{\x\x} Q_{1,\x}|\right) d\x d\tau\\
			&\leq \frac{1}{4} \int_0^t\| (\phi_{\x\x},\psi_{\x\x})\|_{L^2(\bbr)}^2 d\tau+C \int_0^t \left(\| \vartheta_{\x\x}\|_{L^2(\bbr)}^2 +\| (Q_{1,\x},Q_{2,\x})\|_{L^2(\bbr)}^2 \right)d\tau\\
			&\leq \frac{1}{4} \int_0^t\| (\phi_{\x\x},\psi_{\x\x})\|_{L^2(\bbr)}^2 d\tau+C \int_0^t \| \vartheta_{\x\x}\|_{L^2(\bbr)}^2 d\tau +C\delta_0.
		\end{aligned}
	\end{equation}
	
	Using the smallness of $\eps_1$ and $\delta_1$, together with \eqref{H1}-\eqref{H9}, we complete the proof.
\end{proof}

\begin{remark}
In Lemma \ref{lem-energy3} and Lemma \ref{lem-energy4}, by coupling the momentum equation $\eqref{per-system}_{2}$ and energy equation $\eqref{per-system}_{3}$, the dissipation generated by the heat-conductive term can be transferred, through the hyperbolic structure, to the density and velocity components, thereby enabling us to close the full energy estimates.
\end{remark}
\noindent{\bf Conclusion:}
From Lemmas \ref{lem-zvh}, \ref{lem-energy1}-\ref{lem-energy4} and \ref{cw-lemma}, together with the smallness of $\eps_1$ and $\delta_1$, we have 
\begin{align*}
		\begin{aligned}
			&\|(v-\bar v, u-\bar u, \theta-\bar \theta)(t,\cdot)\|^2_{H^2(\bbr)}+\deltas\int_0^t|\dot{\mb{X}}(\tau)|^2 d\tau+\int_0^t ( \mathcal{G}^R(U) + \mathcal{G}^S(U))d\tau \\
			&\quad+\int_0^t\|(v-\bar v, u-\bar u)_\x(\tau,\cdot)\|_{H^1(\bbr)}^2 d\tau+\int_0^t\|(\theta-\bar \theta)_\x(\tau,\cdot)\|^2_{H^2(\bbr)} d\tau\\
			&\le C_0^2 \|\big(v_0(\cdot)-\bar v(0,\cdot),u_0(\cdot) -\bar u(0,\cdot),\theta_0(\cdot)-\bar\theta(0,\cdot)\big)\|^2_{H^2(\bbr)}  + C_0^2 \delta_0^{1/2} ,\qquad 0\leq t\leq T,
		\end{aligned}
	\end{align*}
	where $ \mathcal{G}^R(U), \mathcal{G}^S(U)$ are as defined in \eqref{maingood}. Then, by the coordinate transformation \(\xi=x-\sigma t\), the proof of Proposition \(\ref{prop2}\) is completed.

\subsection{Bootstrap argument}
We now complete the proof by combining the local existence result in Proposition \ref{prop:soln} with the {\it a priori} estimates established in Proposition \ref{prop2}. We first obtain a local solution and then specify the smallness conditions on $\varepsilon_0$ and $\delta_0$, which allow us to close the bootstrap argument. Let $\varepsilon_1, \delta_1$ and $C_0$ be the constants given in Proposition \ref{prop2}, and $C,C_i(i=1,2,3)$ are constants independent of $\varepsilon_i$, $\delta_i$ and $T_i(i=0,1)$.

\noindent$\bullet$ {\bf The existence of local solution.}\\
We choose the smooth functions
$(\underline v,\underline u,\underline\theta)$ in Proposition \ref{prop:soln}
such that
\begin{equation}\label{underline}
    \sum_{\pm}
    \|(\underline v-v_\pm,\underline u-u_\pm,
    \underline\theta-\theta_\pm)\|_{L^2(\mathbb R_\pm)}
    +
    \|(\underline v,\underline u,\underline\theta)_x\|_{H^1(\mathbb R)}
    \leq  C_1\delta_0.
\end{equation}
Then, by the initial perturbation condition \eqref{i-p}, and \eqref{underline}, we obtain
\begin{equation}\label{ini-*}
    \|(v_0-\underline v,
       u_0-\underline u,
       \theta_0-\underline\theta)\|_{H^2(\mathbb R)}
    \leq \varepsilon_0+C_1\delta_0.
\end{equation}
Hence, by Sobolev's inequality,
\[
    \|(v_0-\underline v,
       \theta_0-\underline\theta)\|_{L^\infty(\mathbb R)}
    \leq C(\varepsilon_0+C_1\delta_0).
\]
Since $\underline v$ and $\underline\theta$ are uniformly bounded away from zero, taking $\varepsilon_0$ and $\delta_0$ sufficiently small yields
\[
    0<C^{-1}\leq v_0,\,\theta_0\leq C.
\]
Therefore, Proposition \ref{prop:soln} gives a $T_0>0$ such that
\eqref{NS} admits a unique local solution $(v,u,\theta)$ on $[0,T_0]$,
with
\[
    (v-\underline v,u-\underline u,\theta-\underline\theta)
    \in C([0,T_0];H^2(\mathbb R)).
\]
\noindent$\bullet$ {\bf Choice of the small parameters $\varepsilon_0$ and $\delta_0$.}\\
We now specify the choice of the small parameters $\eps_0$ and $\delta_0$ more precisely. Let $0<\delta_0<\delta_1<1$. Then, by Lemmas \ref{lemma1.2}-\ref{Lemma 2.2.}, we have
\begin{equation*}
    \sum_{\pm}
    \|(\bar v(0,\cdot)-v_\pm,\bar u(0,\cdot)-u_\pm,
    \bar\theta(0,\cdot)-\theta_\pm)\|_{L^2(\mathbb R_\pm)}
    +
    \|(\bar v,\bar u,\bar\theta)_x(0,\cdot)\|_{H^1(\mathbb R)}
    \leq C_2\sqrt{\delta_0}.
\end{equation*}
Combining the above estimates and \eqref{underline}, we obtain
\begin{equation}\label{similar}
    \|(\underline v(\cdot)-\bar v(0,\cdot),
       \underline u(\cdot)-\bar u(0,\cdot),
       \underline\theta(\cdot)-\bar\theta(0,\cdot))\|_{H^2(\mathbb R)}
    \leq C_3\sqrt{\delta_0}.
\end{equation}

Set
\[
    \varepsilon_*
    :=
    \frac{\frac{\varepsilon_1}{2}-C_0\delta_0^{1/4}}
    {C_0+1}.
\]
We choose $\delta_0>0$ sufficiently small such that
\begin{equation}\label{delta-0}
    \frac{\frac{\varepsilon_1}{4}-C_0\delta_0^{1/4}}
    {C_0+1}
    -C_1\delta_0-C_3\sqrt{\delta_0}\geq 0,
\end{equation}
and choose $\varepsilon_0>0$ such that
\begin{equation}\label{eps-0}
    \varepsilon_0\leq
    \frac{\varepsilon_1}{4(C_0+1)}.
\end{equation}
It then follows from \eqref{delta-0} and \eqref{eps-0} that
\[
    \varepsilon_0+C_1\delta_0+C_3\sqrt{\delta_0}
    \leq \varepsilon_*.
\]
Thus, by \eqref{ini-*}, \eqref{similar}, and the triangle inequality,
\begin{equation}\label{initial-small}
    \|(v_0(\cdot)-\bar v(0,\cdot),u_0(\cdot)-\bar u(0,\cdot),\theta_0(\cdot)-\bar\theta(0,\cdot))
    \|_{H^2(\mathbb R)}
    \leq \varepsilon_*,
\end{equation}
where
\[
    0<\varepsilon_*<\frac{\varepsilon_1}{2}.
\]

Together with \eqref{ini-*}, from Proposition \ref{prop:soln}, it holds that there exists $T_0>0$ such that \eqref{NS} has a unique solution $(v,u,\theta)$ on $[0,T_0]$ satisfying 
\begin{equation*}
    \|(v(t,\cdot)-\underline v(\cdot), u(t,\cdot)-\underline u(\cdot), \theta(t,\cdot)-\underline\theta(\cdot))\|_{H^2(\mathbb{R})} \leq \frac{\eps_1}{2}, \qquad 0\leq t\leq T_0.
\end{equation*}
By an argument similar to that used in deriving \eqref{similar}, we obtain
\begin{equation*}
    \|(\underline v(\cdot)-\bar v(t,\cdot), \underline u(\cdot)-\bar u(t,\cdot), \underline\theta(\cdot)-\bar\theta(t,\cdot))\|_{H^2(\mathbb{R})} \leq C \sqrt{\delta_0}(1+\sqrt{t}),\qquad 0\leq t\leq T_0.
\end{equation*}
Furthermore, using the smallness of $\delta_0$, and choosing $T_1<T_0$ small enough such that $C \sqrt{\delta_0}(1+\sqrt{T_1})<\frac{\eps_1}{2}$, we have
\begin{equation*}
    \|(\underline v(\cdot)-\bar v(t,\cdot), \underline u(\cdot)-\bar u(t,\cdot), \underline\theta(\cdot)-\bar\theta(t,\cdot))\|_{H^2(\mathbb{R})} \leq \frac{\eps_1}{2},\qquad 0\leq t\leq T_1.
\end{equation*}
Therefore, by the above estimates, we obtain
\begin{equation*}
    \|(v(t,\cdot)-\bar v(t,\cdot),\,u(t,\cdot)-\bar u(t,\cdot),\,
    \theta(t,\cdot)-\bar\theta(t,\cdot))\|_{H^2(\mathbb{R})}
    \leq \eps_1,
    \qquad 0\leq t\leq T_1.
\end{equation*}
Moreover, since $\mb X(t)$ is absolutely continuous, it follows from
Proposition \ref{prop:soln} that
\begin{equation*}
    (v-\bar v, u-\bar u, \theta-\bar\theta)
    \in C([0,T_1];H^2(\bbr)).
\end{equation*}
Hence, the {\it a priori} assumption \eqref{apri-ass} in Proposition \ref{prop2} is satisfied on a nonempty time interval.

\noindent$\bullet$ {\bf The bootstrap argument.}\\
We now define
\[
    T^*
    :=
    \sup\Big\{
    T>0:
    \|(v-\bar v,u-\bar u,\theta-\bar\theta)
    (t,\cdot)\|_{H^2(\mathbb R)}
    \leq\varepsilon_1,\quad 0\leq t\leq T
    \Big\}.
\]
We claim that
$T^*=+\infty$. Suppose, on the contrary, that $T^*<+\infty$. Then,
by Proposition \ref{prop2} and \eqref{initial-small},
\[
    \|(v-\bar v,u-\bar u,\theta-\bar\theta)
    (t,\cdot)\|_{H^2(\mathbb R)}
    \leq
    C_0\bigl(\varepsilon_*+\delta_0^{1/4}\bigr),
    \qquad 0\leq t<T^*.
\]
By the definition of $\varepsilon_*$, we have
\begin{align*}
    C_0\bigl(\varepsilon_*+\delta_0^{1/4}\bigr)
    &=
    \frac{C_0}{C_0+1}
    \left(
        \frac{\varepsilon_1}{2}
        -C_0\delta_0^{1/4}
    \right)
    +C_0\delta_0^{1/4}
    \\
    &\leq\left(
        \frac{\varepsilon_1}{2}
        -C_0\delta_0^{1/4}
    \right)
    +C_0\delta_0^{1/4}= \frac{\varepsilon_1}{2},
\end{align*}
and consequently
\begin{equation*}
    \|(v-\bar v,u-\bar u,\theta-\bar\theta)
    (t,\cdot)\|_{H^2(\mathbb R)}
    \leq \frac{\varepsilon_1}{2}<\eps_1,
    \qquad 0\leq t\leq T^*.
\end{equation*}
By the continuity in time of the solution, the above strict improvement of the bootstrap bound allows the solution to be extended beyond $T^*$ while still satisfying the bootstrap assumption, contradicting the definition of $T^*$. Therefore, we must have $T^* = +\infty$. 


\section*{Acknowledgment}
L.-A. Li was supported by Beijing Natural Science Foundation (No. 1254045), the National Natural Science Foundation of China (No. 12501295) and the Fundamental Research Funds for the Central Universities (No. 2243100008). J. Wu was partially supported by the National Science Foundation of the United States under Grants DMS-2104682 and DMS-2309748. X. Xu was partially supported by the National Natural Science Foundation of China (grants 12571244, 12171040) and the National Key R\&D Program of China (grant 2020YFA0712900).

\bibliographystyle{plain}
\bibliography{reference}

\end{document}